\documentclass[11pt]{amsart}
\usepackage{amssymb}
\usepackage{amsmath}
\usepackage{amsfonts}
\usepackage{mathrsfs}
\usepackage{psfrag}
\usepackage{verbatim}
\usepackage{color}
\usepackage{bm}
\usepackage[usenames,dvipsnames]{xcolor}
\usepackage{dsfont}
\usepackage{overpic}
\usepackage{subcaption}

\usepackage[colorlinks=true]{hyperref}
\hypersetup{urlcolor=blue, citecolor=ForestGreen}

\DeclareGraphicsExtensions{}
\theoremstyle{plain}
\newtheorem{theorem}{Theorem}[section]
\newtheorem{corollary}[theorem]{Corollary}

\newtheorem{lemma}[theorem]{Lemma}
\newtheorem{proposition}[theorem]{Proposition}

\theoremstyle{definition}
\newtheorem{assumption}[theorem]{Assumption}

\theoremstyle{remark}
\newtheorem{remark}[theorem]{Remark}

\newcommand{\eps}{\varepsilon}

\newcommand{\kap}{\varkappa}

\newcommand{\BR}{\mathbb{R}}
\newcommand{\BP}{\mathbb{P}}
\newcommand{\BE}{\mathbb{E}}
\newcommand{\BZ}{\mathbb{Z}}
\newcommand{\BN}{\mathbb{N}}

\newcommand{\filt}{\mathscr{F}}

\newcommand{\sfP}{\mathsf{P}}

\newcommand{\sfN}{\mathsf{N}}
\newcommand{\sfM}{\mathsf{M}}

\newcommand{\cc}{\mathsf{c}}
\newcommand{\pp}{\mathsf{p}}
\newcommand{\qq}{\mathsf{q}}

\newcommand{\gDk}{\mathsf{J}} 

\begin{document}

\title[Asymptotic Analysis of nonlinear systems with noise and sampling effects]{Fluctuation analysis for a class of nonlinear systems with fast periodic sampling and small state-dependent white noise
}

\author[Shivam Dhama and Chetan D. Pahlajani]{}
\address{Discipline of Mathematics\\ Indian Institute of Technology Gandhinagar}
\email{shivam.dhama@iitgn.ac.in, maths.shivam@gmail.com} 

\address{Discipline of Mathematics\\ Indian Institute of Technology Gandhinagar}
 \email{cdpahlajani@iitgn.ac.in}
\thanks{$^*$ Corresponding author: Shivam Dhama}

\maketitle

\centerline{\scshape Shivam Dhama$^*$ and Chetan D. Pahlajani}
\medskip
{\footnotesize
 
 \centerline{Discipline of Mathematics, Indian Institute of Technology Gandhinagar, Palaj, Gandhinagar 382055, India}
} 
\begin{center}
\rule{17cm}{.01mm}
\end{center}
\begin{abstract}

We consider a nonlinear differential equation under the combined influence of small state-dependent Brownian perturbations of size $\varepsilon$, 
and fast periodic sampling with period $\delta$; $0<\varepsilon, \delta \ll 1$.
Thus, state samples (measurements) are taken every $\delta$ time units, and 
the instantaneous rate of change of the state depends 
on its current value \textit{as well as} its most recent sample.
We show that the resulting stochastic process indexed by 
$\varepsilon,\delta$, can be approximated, as $\varepsilon,\delta \searrow 0$, by an ordinary differential equation ({\sc ode}) with vector field obtained by 
replacing the most recent sample by the current value of the state. We next analyze the fluctuations of the stochastic process about the limiting {\sc ode}. Our main result asserts that, for the case when $\delta \searrow 0$ at the same rate as, or faster than, $\varepsilon \searrow 0$, the rescaled fluctuations can be approximated in a suitable strong (pathwise) sense by a limiting stochastic differential equation ({\sc sde}). This {\sc sde} varies depending on the exact rates at which $\varepsilon,\delta \searrow 0$. The key contribution here involves computing the \textit{effective drift} term capturing the interplay between noise and sampling in the limiting {\sc sde}. The results essentially provide a first-order perturbation expansion, together with error estimates, for the stochastic process of interest. 
Connections with the performance analysis of feedback control systems with sampling are discussed and illustrated numerically through a simple example.

\end{abstract}
\begin{center}
\rule{17cm}{.01mm}
\end{center}

\section{Introduction}\label{S:Introduction}

In certain applications, notably computer control of physical systems, one naturally encounters differential equations with \textit{periodic sampling}; by this, we mean an equation of the form
\begin{equation}\label{E:ode-sampled}
\frac{dx}{dt}=c\left(x(t),x(\delta\lfloor t/\delta\rfloor)\right), \quad x(0)=x_0 \in \BR^n,
\end{equation}
where $t \in [0,\infty)$ represents time, $x(\cdot):[0,\infty) \to \BR^n$ represents the ``state" of a system, $\delta \ll 1$ is a positive parameter, $\lfloor\cdot\rfloor$ denotes the integer floor function, and $c:\BR^n \times \BR^n \to \BR^n$ is a sufficiently regular mapping. Thus, \textit{samples} (measurements) of $x(\cdot)$ are taken every $\delta$ time units, and the instantaneous value of $dx/dt$ at time $t \in [0,\infty)$ depends on both $x(t)$ and also on the most recent sample of $x(\cdot)$, viz., $x(\delta\lfloor t/\delta \rfloor)$. Since the first component of the function $c$ changes continuously with time whereas the second one is updated only at discrete-time instants, we have a \textit{hybrid dynamical system} \cite{GST-HybDynSys-book}.

If this system is subjected to small Brownian perturbations, then one should consider the stochastic differential equation ({\sc sde}) \cite{KS91,Oksendal} 
\begin{equation}\label{E:Blago}
dX_t^{\eps,\delta}=c(X_t^{\eps,\delta},X^{\eps,\delta}_{\delta\lfloor t/\delta\rfloor})dt + \eps\sigma(X_t^{\eps,\delta})dW_t,\quad X_0^{\eps,\delta}=x_0 \in \BR^n,
\end{equation}
where $\sigma:\BR^n \to \BR^{n\times n}$ satisfies suitable regularity conditions, $W_t$ is an $n$-dimensional Brownian motion, and $0 < \eps, \delta \ll 1$. Our main results study the limiting behavior of $X_t^{\eps,\delta}$ as $\eps,\delta \searrow 0$ when $c$ takes the specific form $c(x,y)=f(x)+g(x)\kappa(y), x,y\in \BR^n$, with the functions $f:\BR^n \to \BR^n$, $g:\BR^n \to \BR^{n\times m}$, $\kappa:\BR^n \to \BR^m$ satisfying certain regularity conditions stated in  Assumptions \ref{A:Smooth-LinearGrow} and \ref{A:Boundedness}. More precisely, we show (Theorem \ref{T:LLN-II}) that for any fixed $T<\infty$, the process $X^{\eps,\delta}_t$ converges, as $\eps,\delta \searrow 0$, to the solution $x_t$ of the equation $x_t  = x_0 + \int_0^t [f(x_s)+g(x_s)\kappa(x_s)]\thinspace ds$ on the interval $[0,T]$.\footnote{We will be indicating time dependence by a subscript, e.g., we will write $x_t$ in place of $x(t)$.} We next attempt to understand the limiting behavior of the (rescaled) fluctuations of the process $X_t^{\eps,\delta}$ about $x_t$. Depending on the relative rates at which $\eps,\delta \searrow 0$, we have three different regimes. For the cases when $\delta \searrow 0$ faster than or at the same rate as $\eps \searrow 0$ (Regimes 1 and 2 in what follows, see \eqref{E:c} below), we are able to show (Theorem \ref{T:fluctuations-R-1-2}) that for any $T<\infty$, the fluctuation process $\eps^{-1}(X^{\eps,\delta}_t - x_t)$ can be well approximated (in a suitable pathwise sense) on the time interval $[0,T]$ by a diffusion process $Z_t$ given by a time-inhomogeneous linear {\sc sde}. Part of the novelty of the result is that the limiting {\sc sde} for $Z_t$ contains an extra \textit{effective drift} term encoding the limiting cumulative effects of small noise and fast sampling.

%
%
%


One way to mathematically motivate the problem at hand comes from the work of Blagoveshchenskii \cite{blagoveshchenskii1962diffusion} and material presented in chapter 2 of \cite{FW_RPDS}. In these references, the authors consider dynamical systems of the form $\dot{y}=b(y), y(0)=y_0 \in \BR^n$, 
and study small random perturbations given by {\sc sde} of the form
\begin{equation*}
dY_t^{\eps}=b(Y_t^{\eps})dt+\eps \sigma(Y_t^{\eps})dW_t , \quad Y_0^{\eps}=y_0.
\end{equation*}
Here, $b:\BR^n \to \BR^n$, $\sigma:\BR^n \to \BR^{n\times n}$ are sufficiently regular, $W_t$ is an $n$-dimensional Brownian motion, and $0<\eps \ll 1$. 
Note that the {\sc ode} for $y$ and the {\sc sde} for $Y^\eps_t$ can be thought of as limiting cases of \eqref{E:ode-sampled} and \eqref{E:Blago} when $\eps$ is held fixed, $\delta \searrow 0$, and we write $b(y)=c(y,y)$ for $y \in \BR^n$. In \cite{blagoveshchenskii1962diffusion,FW_RPDS}, it is shown that $Y_t^\eps$ converges in probability to $y_t$ as $\eps \searrow 0$ on any fixed time interval $[0,T]$. Further, under certain stringent conditions on $b, \sigma$, e.g., boundedness of all partial derivatives up to and including order $n+1$ (here, $n \in \BN$), the solution process $Y_t^\eps$ can be expanded in powers of $\eps$ as  
$Y_t^\eps=Y^{(0)}_t+\eps Y_t^{(1)}+ \dots + \eps^n Y_t^{(n)}+ o(\eps^n)$ for $t \in [0,T]$ where $Y^{(0)}_t=y_t$ and $Y_t^{(1)},\dots, Y_t^{(n)} $ satisfy a system of {\sc sde} with each $Y^{(k)}_t$ depending on $Y^{(0)}_t,\dots,Y^{(k-1)}_t$, $1 \le k \le n$. In particular, letting $Db$ denote the Jacobian matrix of the drift $b$, the process $Y^{(1)}_t$ solves the {\sc sde} 
\begin{equation}\label{E:BF-sde}
dY_t^{(1)}=Db(y_t)Y_t^{(1)}dt+\sigma(y_t)dW_t, \quad Y_0^{(1)}=0,
\end{equation}
obtained by linearizing the dynamics of $Y^\eps_t$ about $y_t$. Thus, the process $Y^\eps_t$ can be approximated on the time interval $[0,T]$ to within terms of order $\eps^2$ by the process $y_t + \eps Y^{(1)}_t$, implying that $\eps^{-1}\left(Y^\eps_t - y_t\right)$ converges to $Y^{(1)}_t$ solving \eqref{E:BF-sde} as $\eps \searrow 0$. Returning to \eqref{E:ode-sampled},\eqref{E:Blago}, one might say that our problem of interest involves investigating how the results of \cite{blagoveshchenskii1962diffusion,FW_RPDS} need to be modified to include fast sampling effects.

Since our investigations involve the asymptotic analysis of {\sc sde} with small parameters, we take a moment to briefly review a few of the numerous studies on multiscale stochastic processes. For stochastic processes with multiple temporal scales, much of the work has been organized around the so-called \textit{averaging principle}: rigorous methods for approximating dynamics of slowly-varying quantities by taking long-term averages in quickly-varying quantities. The classical work here dates back to Khasminskii \cite{has1966stochastic,
khasminskij1968principle}, while some recent contributions include  \cite{freuidlin1994random,pardoux2001poisson,
pardoux2003poisson,pardoux2005poisson,bakhtin2004diffusion,AthreyaBorkerSureshSundaresan19}, among several others. 
The papers \cite{rockner2021averaging,rockner2021diffusion,KS2014} explore the averaging principle and fluctuation analysis (about the averaged limit) for stochastic systems with multiple time scales and multiple small parameters in different asymptotic regimes, based on the relative rates at which the small parameters approach zero. The interplay between large deviations and homogenization for some problems with small noise and multiple spatial scales is explored in \cite{FreidlinSowers}, while the interaction of large deviations and averaging is studied in \cite{veretennikov1999large,veretennikov2000large,Spil-AppMathOptim}. Averaging and fluctuation analysis for stochastic partial differential equations have been addressed in \cite{cerrai2009averaging,cerrai2009khasminskii,wang2012average,shi2021homogenization}. Finally, problems with small noise and control magnitude for controlled diffusion process are discussed in 
\cite{AriAnupBor18}, while \cite{AnupBorkar09} explores the limiting behavior of invariant densities of small noise diffusions with state-dependent random perturbations. Many of the results above can be interpreted as results along the lines of the classical limit theorems of probability: Law of Large Numbers ({\sc lln}) when one shows almost sure convergence to a deterministic limit, the Central Limit Theorem ({\sc clt}) when one looks at convergence in distribution of fluctuations about this limit, or  Large Deviation Principle ({\sc ldp}) when the probability of rare events is being quantified. 

Although various types of multiscale problems for stochastic processes have been extensively studied, as may be evident from the rather abbreviated list of references above, it appears that the interaction between small noise and fast sampling has received very little attention. Further, such questions arise very naturally in control theory when a physical system governed by an {\sc ode} is controlled by a digital computer (which evolves in discrete time) \cite{YuzGoodwin-book}, while being subjected to small white noise perturbations. For such \textit{sampled-data systems}, the so-called ``emulation approach" is frequently used, where a controller designed in continuous time is implemented using a sample-and-hold device. By the latter, we mean the following: the state of the system is \textit{sampled} at closely spaced discrete time instants, the control action is computed based on these measurements, and the control input is then \textit{held} fixed until the next sample is taken. Much of the work in the relevant literature, see for instance \cite{Khalil-TAC-2004,nesic2009explicit,dong2020asymptotic}, focuses on sufficient conditions on system parameters and time between samples which ensure that the system with sampling retains the desired qualitative properties of the idealized system (with ``infinitely fast" sampling).

A few comments about the contribution of the present work, and how it fits against the backdrop of the literature described above, are now in order. As noted earlier, we focus in this article on the asymptotic behavior, as $\eps,\delta \searrow 0$, of the process $X_t^{\eps,\delta}$ given by equation \eqref{E:Blago}, when $c$ takes the specific form $c(x,y)=f(x)+g(x)\kappa(y)$ with suitable assumptions on $f,g,\kappa$. Our first result, Theorem \ref{T:LLN-II}, which asserts closeness of $X^{\eps,\delta}_t$ to $x_t$ over finite time intervals in the limit $\eps,\delta \searrow 0$, can be interpreted as a {\sc lln} result. An analysis of the fluctuations of $X^{\eps,\delta}_t$ about $x_t$ requires one to carefully account for the relative rates at which $\eps,\delta \searrow 0$. Assuming that $\delta=\delta_\eps$, we follow \cite{FreidlinSowers,KS2014} and consider three different regimes
\begin{equation}\label{E:c}
\cc \triangleq \lim_{\eps \searrow 0}\delta_\eps/\eps \begin{cases}=0 & \text{Regime 1,}\\ \in (0,\infty) & \text{Regime 2,}\\ = \infty & \text{Regime 3.}\end{cases}
\end{equation}
Our next result, Theorem \ref{T:fluctuations-R-1-2}---which is the principal contribution of this paper---proves, for regimes 1 and 2, convergence of the rescaled fluctuations $Z^{\eps,\delta}_t \triangleq \eps^{-1}(X^{\eps,\delta}_t - x_t)$ to a limiting time-inhomogeneous diffusion process $Z_t$ over finite time intervals in the limit as $\eps,\delta \searrow 0$. Part of the challenge here is identification of the limiting {\sc sde} for $Z_t$, more specifically, identification of the extra effective drift term which captures the cumulative limiting effect of small noise and fast sampling.\footnote{In terms of identifying an effective drift in the limiting dynamics, our results appear similar to those of \cite{KS2014}. One notes, however, that the effective drift in \cite{KS2014} is found by understanding the interplay between small noise and fast time scales, while we obtain our effective drift by analyzing the interplay between small noise and fast sampling.} The result Theorem \ref{T:fluctuations-R-1-2} is certainly in the spirit of the {\sc clt}, given that it studies asymptotic behavior of fluctuations. We note, however, that since Theorem \ref{T:fluctuations-R-1-2} estimates the expected value of the maximum separation between the sample paths of $Z^{\eps,\delta}_t$ and $Z_t$, it is actually stronger than (and thus implies) convergence in distribution, as noted in Corollary \ref{C:FCLT}.

The paper is organized as follows. Section \ref{S:ProblemStatementResults} is dedicated to problem formulation and statement of our main results. After setting the stage in Section \ref{SS:Setup}, we state our main results, viz., Theorems \ref{T:LLN-II} and \ref{T:fluctuations-R-1-2}, in Section \ref{SS:MainResults}. Connections of the problem being studied with control theory are discussed in Section \ref{SS:Control}. Next, Section \ref{S:FLLN} is devoted to the proof of Theorem \ref{T:LLN-II}. The proof of Theorem \ref{T:fluctuations-R-1-2} is considerably more involved and the details are spread out over Sections \ref{S:FCLT} and \ref{S:CLT-Lemmas}.
For ease of exposition, Theorem \ref{T:fluctuations-R-1-2} is proved in Section \ref{S:FCLT} via a series of Propositions, which in turn are proved using a family of auxiliary Lemmas. The statements and proofs of the Propositions, and the statements of the Lemmas are all given in Section \ref{S:FCLT}. The proofs of the Lemmas are furnished in Section \ref{S:CLT-Lemmas}, thereby completing the proof of Theorem \ref{T:fluctuations-R-1-2}. Finally, we close out the article with a numerical illustration of our main results in Section \ref{S:NumExa}, followed by a few concluding remarks with possible future directions in Section \ref{S:conclusion}.

%
%

\subsection*{Notation and Conventions}
We end this section by listing some commonly used notation, and perhaps more importantly, outlining some notational conventions to be used in deriving various estimates. For $\pp,\qq \in \BN$, we will let $\BR^\pp$ and $\BR^{\qq\times \pp}$ denote, respectively, $\pp$-dimensional Euclidean space and the space of all $\qq\times \pp$ matrices with real entries. For $x=(x_1,\dots,x_\pp) \in \BR^\pp$, let $|x| \triangleq \sum_{i=1}^{\pp}|x_i|$ and $\|x\| \triangleq \sqrt{\sum_{i=1}^{\pp}|x_i|^2}$ denote the one- and two- norms on $\BR^\pp$; here, the symbol $\triangleq$ is read ``is defined to equal". For $M\in \BR^{\qq\times \pp}$, $|M|$ and $\|M\|$ denote the corresponding induced matrix norms (see, for instance, \cite{Hes_LST}). For a smooth function $h:\BR^\pp \rightarrow \BR^\qq,$ $Dh$ represents the Jacobian matrix, and for $h:\BR^\pp \rightarrow \BR$, $D^2h$ represents the Hessian matrix whose entries are the second order partial derivatives.

We will frequently use the fact that for any $n \in \BN$, $a_i\in \BR,i=1,...,n$, $p>0$, we have the inequality 
\begin{equation}\label{E:Triangle-Inequ}
(|a_1|+...+|a_n|)^p \le n^p(|a_1|^p+...+|a_n|^p).
\end{equation}
To simplify notation while estimating various quantities, we will adopt the following conventions. Positive constants appearing in the statement of a lemma, proposition, theorem, etc. will be denoted with an identifying numerical subscript, e.g., $C_{\ref{T:LLN-II}}$ in Theorem \ref{T:LLN-II}. In contrast, while proving said lemma, proposition, theorem, etc., we will let $C$ denote a suitable positive constant whose exact value may (and typically will) change from line to line. Since our focus is asymptotic analysis, as $\eps,\delta \searrow 0$, of the stochastic process $\{X^{\eps,\delta}_t: t \in [0,T]\}$ for fixed $T<\infty$,  
we will---with the \textit{exception} of $\eps,\delta$---absorb dependence on all other problem parameters ($f,g,\kappa,\sigma,T,n,m$ and $x_0$) into the generic constants $C$ and specific constants (e.g., $C_{\ref{T:LLN-II}}$) without explicit mention.

\section{Problem Statement and Results}\label{S:ProblemStatementResults}
\noindent

\subsection{Problem Setup}\label{SS:Setup}
Fix positive integers $n,m$ and suppose we have functions $f:\BR^n \to \BR^n$, $g:\BR^n\to \BR^{n\times m}$, $\kappa: \BR^n \to \BR^m$. We are interested in the {\sc ode}
$\dot{x}_t=f(x_t) + g(x_t)\kappa (x_{\delta\lfloor t/\delta\rfloor})$, with initial condition $x_0 \in \BR^n$, where $\dot{x}$ denotes the derivative of $x$ with respect to time $t$, and the parameter $\delta \in (0,1)$. This is simply equation \eqref{E:ode-sampled} when $c(x,y)=f(x)+g(x)\kappa(y)$, and models a situation where samples of $x$ are taken every $\delta$ units of time and the instantaneous value of $\dot{x}_t$ depends on both $x_t$ and its most recent sample $x_{\delta\lfloor t/\delta\rfloor}$. This dependence is through the possibly nonlinear functions $f,g,\kappa$ which will be assumed to satisfy certain regularity conditions in Assumptions \ref{A:Smooth-LinearGrow}, \ref{A:Boundedness} below. 

Since we will be interested in asymptotic analysis as $\delta \searrow 0$, we will emphasize the dependence of $x$ on $\delta$ using a superscript. As mentioned earlier, we will find it notationally convenient to indicate the dependence of $x$ on time $t$ by a subscript, i.e., $x^\delta_t \triangleq x^\delta(t)$.\footnote{This convention also conforms with the stochastic process notation in the sequel, where time dependence will be indicated by a subscript.} For each $\delta \in (0,1)$, we now define a time-discretization operator $\pi_\delta:[0,\infty) \to \delta \BZ^+$ by
\begin{equation*}
\pi_\delta(t) \triangleq \delta \lfloor {t}/{\delta} \rfloor \qquad \text{for $t \in [0,\infty)$,} 
\end{equation*}
which rounds down the continuous time $t \in [0,\infty)$ to the nearest multiple of $\delta$. Then, our quantity of interest is the continuous function $x^\delta_t:[0,\infty) \to \BR^n$ which solves the differential equation 
\begin{equation}\label{E:ie}
\dot{x}^{\delta}_t = f(x_t^{\delta}) + g(x_t^{\delta})\kappa (x_{\pi_\delta(t)}^{\delta}), \qquad x^\delta_0=x_0. 
\end{equation}

To now consider the situation where the equation \eqref{E:ie} is subjected to Brownian perturbations, we fix a probability space $(\Omega,\filt,\BP)$ equipped with a filtration $\{\filt_t:t \ge 0\}$ satisfying the usual conditions \cite{KS91}, which supports the $n$-dimensional Brownian motion $W=\{W_t,\filt_t: t \ge 0\}$. The expectation operator with respect to the probability measure $\BP$ will be denoted by $\BE\left[\cdot\right]$. If the Brownian perturbations are of size $0<\eps \ll 1$ with state-dependent diffusion matrix $\sigma:\BR^n \to \BR^{n\times n}$, the resulting dynamics are governed by the stochastic process $X^{\eps,\delta}_t$ which solves the {\sc sde} 
\begin{equation}\label{E:sie}
dX^{\eps,\delta}_t = \left[f(X^{\eps,\delta}_t) + g(X^{\eps,\delta}_t)\kappa (X^{\eps,\delta}_{\pi_\delta(t)})\right]\thinspace dt + \eps \sigma(X_t^{\eps,\delta}) \thinspace dW_t, \quad X^{\eps,\delta}_0=x_0. 
\end{equation}
Of course, the {\sc sde} \eqref{E:sie} is understood rigorously as the stochastic integral equation 
$X^{\eps,\delta}_t = x_0 + \int_0^t \left[f(X^{\eps,\delta}_s) + g(X^{\eps,\delta}_s)\kappa (X^{\eps,\delta}_{\pi_\delta(s)})\right]\thinspace ds + \eps \int_0^t \sigma(X_s^{\eps,\delta}) dW_s$.
The mapping $\sigma:\BR^n \to \BR^{n\times n}$ and the mappings $f,g,\kappa$ will be assumed to satisfy the following regularity conditions.


\begin{assumption}\label{A:Smooth-LinearGrow}
The functions $f:\BR^n\rightarrow \BR^n,~g:\BR^n\rightarrow \BR^{n\times m}$, $\kappa: \BR^n \rightarrow \BR^m$, $\sigma:\BR^n \to \BR^{n \times n}$ are sufficiently smooth and globally Lipschitz continuous. Consequently, these functions have linear growth, i.e., there exists a positive constant $C$ such that 
\begin{equation*}
\left(|f(x)|+|g(x)|+|\kappa(x)|+|\sigma(x)|\right) \le C(1+|x|), \quad \text{for any $x\in \BR^n$}.
\end{equation*} 
\end{assumption}

\begin{assumption}\label{A:Boundedness}
The function $g$ and the partial derivatives of $\kappa$, $g$ and $f$ up to and including the second order are bounded. \color{black}
\end{assumption}

Let's see what happens to $x^\delta_t$ and $X^{\eps,\delta}_t$ governed by equations \eqref{E:ie} and \eqref{E:sie}, respectively, in the limit as $\eps,\delta$ vanish. One expects that as $\eps,\delta \searrow 0$, $X^{\eps,\delta}_t$ and $x^\delta_t$ converge, each in a suitable sense, to $x_t$ solving
\begin{equation}\label{E:nonlin-sys}
\dot{x}_t =f(x_t)+g(x_t)\kappa(x_t) \triangleq f(x_t)+\sum_{i=1}^{m}g_i(x_t)\kappa_i(x_t) \quad \text{with initial condition $x_0$ as in \eqref{E:ie},\eqref{E:sie}.} 
\end{equation}
The rightmost expression in \eqref{E:nonlin-sys} is obtained by denoting the columns of $g(x) \in \BR^{n\times m}$ by $g_1,\dots,g_m \in \BR^n$, writing $\kappa(x)=[\kappa_1(x) \medspace \dots \medspace \kappa_m(x)]^\top$, $\kappa_i(x) \in \BR$, $1 \le i \le m$, and simply computing the matrix-vector product $g(x)\kappa(x)$. In the sequel, when referencing and using equations \eqref{E:ie},\eqref{E:sie}, and \eqref{E:nonlin-sys}, we will always be using the corresponding equivalent integral equations.


\subsection{Main Results}\label{SS:MainResults}
Here is our first main result.

\begin{theorem}(Law of Large Numbers Type Result)\label{T:LLN-II}
Let $X_t^{\varepsilon,\delta}$ and $x_t$ solve \eqref{E:sie} and \eqref{E:nonlin-sys}, respectively.
Then, for any fixed $T>0$ and $p\in \{1,2\},$ there exists a positive constant $C_{\ref{T:LLN-II}}$ such that 
for any $\eps,\delta>0$
\begin{equation*}
\BE\left[\sup_{0\le t \le T}|X_t^{\eps,\delta}-x_t|^p\right]\le 
(\eps^p + \delta^p)C_{\ref{T:LLN-II}}e^{C_{\ref{T:LLN-II}}T}.
\end{equation*}
\end{theorem}

Thus, Theorem \ref{T:LLN-II} assures us that for any fixed $T>0$, the quantity $\sup_{0 \le t \le T}|X^{\eps,\delta}_t-x_t|^p$, $p \in \{1,2\}$, converges to zero in $L^1$. Given that $|X^{\eps,\delta}_t-x_t|^p$ may be thought of as being of order $\max\{\eps^p,\delta^p\}$, it is natural to ask whether one can identify any of the higher order terms in an expansion of $X^{\eps,\delta}_t$ in powers of the small parameters. To proceed further along this route, we will assume, as described in \eqref{E:c}, that $\delta=\delta_\eps$ and $\cc \triangleq \lim_{\eps \searrow 0}\delta_\eps/\eps$ exists in $[0,\infty]$. The cases $\cc=0$, $\cc \in (0,\infty)$, and $\cc=\infty$ will correspond, respectively, to Regimes 1, 2, 3. We will now seek an expansion of $X^{\eps,\delta}_t$ in powers of the \textit{coarser} parameter. Thus, for Regimes 1 and 2, we would like to identify a stochastic process $Z_t$, independent of both $\delta$ and $\eps$, such that 
$X^{\eps,\delta}_t = x_t + \eps Z_t + o(\eps)$,
where we can explicitly estimate the error $|X^{\eps,\delta}_t - x_t -\eps Z_t|$.

To get started, we define the rescaled fluctuation process 
\begin{equation}\label{E:fluct-processes}
Z^{\eps,\delta}_t \triangleq \frac{X^{\eps,\delta}_t - x_t}{\eps} \qquad \text{for Regimes 1 and 2.}
\end{equation}
Note that $X^{\eps,\delta}_t$ is identically equal to $x_t + \eps Z^{\eps,\delta}_t$; thus, identifying $Z_t$ is tantamount to identifying the limit of $Z^{\eps,\delta}_t$ as $\eps,\delta \searrow 0$.
Recalling equations \eqref{E:sie} and \eqref{E:nonlin-sys}, we have
\begin{multline}\label{E: cent-limit-eq}
\begin{aligned}
Z_t^{\eps,\delta}&=\int_0^t \frac{f(X_s^{\eps,\delta})-f(x_s)}{\eps}\thinspace ds + \int_0^t \frac{g(X_s^{\eps,\delta})\kappa(X_{\pi_\delta(s)}^{\eps,\delta})-g(x_s)\kappa(x_s)}{\eps} \thinspace ds + \int_0^t \sigma(X_s^{\eps,\delta})dW_s\\
& = \int_0^t \frac{f(X_s^{\eps,\delta})-f(x_s)}{\eps} ds +\int_0^t \frac{g(X_s^{\eps,\delta})[\kappa(X_{\pi_\delta(s)}^{\eps,\delta})-\kappa(X_s^{\eps,\delta})]}{\eps} ds + \int_0^t \frac{[g(X_s^{\eps,\delta})-g(x_s)]\kappa(X_s^{\eps,\delta})}{\eps} ds\\
&  \qquad \qquad \qquad \qquad \qquad \qquad \qquad \qquad \qquad \quad \quad + \int_0^t \frac{g(x_s)[\kappa(X_s^{\eps,\delta})-\kappa(x_s)]}{\eps}ds + \int_0^t \sigma(X_s^{\eps,\delta})dW_s.
\end{aligned}
\end{multline}
For notational convenience, we define $\gDk : \BR^n \rightarrow \BR^{n\times n}$ by
\begin{equation*}
\gDk(x)\triangleq g(x)D\kappa(x) \qquad \text{for $x \in \BR^n$.}
\end{equation*}

If we now replace all terms, with the exception of the stochastic integral, on the right hand side of the second equation in \eqref{E: cent-limit-eq} above by their linearized counterparts, while accounting for the resulting errors separately (see the term $\mathsf{R}^{\eps,\delta}_t=\sum_{i=1}^4 \mathsf{R}^{\eps,\delta}_i(t)$ below), we get
\begin{multline}\label{E:Cent-limit-eq-Rem}
Z_t^{\eps,\delta} = \int_0^t [Df(x_s)+\gDk(x_s)]Z_s^{\eps,\delta}\thinspace ds - \int_0^t \gDk(x_s)\frac{X_s^{\varepsilon, \delta}-X_{\pi_\delta(s)}^{\varepsilon, \delta}}{\eps}\thinspace ds + \int_0^t \sum_{i=1}^{m} Dg_i(x_s)Z_s^{\eps,\delta} \kappa_i(x_s)\thinspace ds \\
 + \int_0^t \sigma(X_s^{\eps,\delta})dW_s + {\sf R}_t^{\eps,\delta}, \quad \text{where}
\end{multline}

\begin{multline}\label{E:Remainder-Terms}
{\sf R}_t^{\eps,\delta}  \triangleq {\sf R}_1^{\eps,\delta}(t)+{\sf R}_2^{\eps,\delta}(t)+{\sf R}_3^{\eps,\delta}(t)+{\sf R}_4^{\eps,\delta}(t)\triangleq \\
\\ \int_0^t \left[\frac{f(X_s^{\eps,\delta})-f(x_s)}{\eps}- Df(x_s)Z_s^{\eps,\delta} \right] ds + \int_0^t \left[\frac{g(X_s^{\eps,\delta})[\kappa(X_{\pi_\delta(s)}^{\eps,\delta})-\kappa(X_s^{\eps,\delta})]}{\eps}-\gDk(x_s)\frac{X_{\pi_\delta(s)}^{\varepsilon, \delta}-X_s^{\varepsilon, \delta}}{\eps} \right]ds\\
\qquad \qquad \qquad \qquad \qquad \qquad \qquad \qquad \quad \quad + \int_0^t \left[\frac{[g(X_s^{\eps,\delta})-g(x_s)]\kappa(X_s^{\eps,\delta})}{\eps}- \sum_{i=1}^{m} Dg_i(x_s)Z_s^{\eps,\delta} \kappa_i(x_s) \right] ds\\
 \qquad \qquad \qquad \quad \quad+\int_0^t \left[\frac{g(x_s)[\kappa(X_s^{\eps,\delta})-\kappa(x_s)]}{\eps}-\gDk(x_s)Z_s^{\eps,\delta}\right] ds.
\end{multline}

Let's now try to chart out the path to identification of the limiting process $Z_t$ and the derivation of estimates for the error $|Z^{\eps,\delta}_t-Z_t|$, and hence for $|X^{\eps,\delta}_t - x_t - \eps Z_t|=\eps|Z^{\eps,\delta}_t-Z_t|$. We start with the error terms ${\sf R}_i^{\eps,\delta}(t)$, $1 \le i \le 4$, which appear in \eqref{E:Cent-limit-eq-Rem}, and have explicit expressions as given in \eqref{E:Remainder-Terms}. Precise estimates for these terms 
will be obtained in Propositions \ref{P:f-Taylor-Approx} through \ref{P:Sampling-Term-Approx} in Section \ref{S:FCLT}, showing that the ${\sf R}_i^{\eps,\delta}(t)$, $1 \le i \le 4$, are small and can be thought of as being of order $\eps$.\footnote{Actually, these terms end up being of order $\max\{\eps,\delta\}$. Since we are working in regimes 1 and 2, $\delta$ is either smaller than, or comparable to, $\eps$ as $\eps \searrow 0$.} A bit informally, then, equation \eqref{E:Cent-limit-eq-Rem} becomes 
\begin{multline}\label{E:fluct-processes}
Z_t^{\eps,\delta}= \int_0^t [Df(x_s)+\gDk(x_s)]Z_s^{\eps,\delta}\thinspace ds - \int_0^t \gDk(x_s)\frac{X_s^{\varepsilon, \delta}-X_{\pi_\delta(s)}^{\varepsilon, \delta}}{\eps}\thinspace ds  \\ + \int_0^t \sum_{i=1}^{m} Dg_i(x_s)Z_s^{\eps,\delta} \kappa_i(x_s)\thinspace ds  + \int_0^t \sigma(X_s^{\eps,\delta}) dW_s + \mathscr{O}(\eps).
\end{multline}
Equation \eqref{E:fluct-processes} certainly suggests that if 
\begin{equation*} 
\ell(t) \triangleq \lim_{\eps,\delta \searrow 0}\int_0^t \gDk(x_s)\frac{X_s^{\varepsilon, \delta}-X_{\pi_\delta(s)}^{\varepsilon, \delta}}{\eps}\thinspace ds
\end{equation*} 
exists in a suitable sense, then the process $Z^{\eps,\delta}_t$ converges (again, in a suitable sense) to the stochastic process $Z_t$ solving 
$Z_t = \int_0^t [Df(x_s)+\gDk(x_s)]Z_s\thinspace ds - \ell(t) + \int_0^t \sum_{i=1}^{m} Dg_i(x_s)Z_s \kappa_i(x_s)\thinspace ds  + \int_0^t \sigma(x_s) dW_s$.
We now have our work clearly cut out for us: identify (i.e., explicitly compute) $\ell(t)$ and then quantify and make precise the convergence of $Z^{\eps,\delta}_t$ to $Z_t$. The former task is accomplished in Proposition \ref{P:MainTermApprox} in Section \ref{S:FCLT}, with explicit expressions for $\ell(t)$ given in \eqref{E:ell_1-and-ell_2}. The latter task also requires Proposition \ref{P:Noise-Term-Est} proved in Section \ref{S:FCLT}, which allows one to replace $\int_0^t \sigma(X^{\eps,\delta}_s) \thinspace dW_s$ by $\int_0^t \sigma(x_s) \thinspace dW_s$ in the limit. The culmination of the above effort is our main result, stated in Theorem \ref{T:fluctuations-R-1-2} below, which estimates, for any fixed $T>0$, the error $\BE\{\sup_{0 \le t \le T}|Z^{\eps,\delta}_t-Z_t|\}$. This estimate enables us to rigorously prove the convergence of $Z^{\eps,\delta}_t$ to $Z_t$ over $t \in [0,T]$ in the limit as $\eps,\delta \searrow 0$. 

To enable estimating $\delta$ in terms of $\eps$, we recall that in regimes 1 and 2, $\lim_{\eps \searrow 0}\delta_\eps/\eps = \cc \in [0,\infty)$. Consequently, there exists $\eps_0 \in (0,1)$ such that
\begin{equation}\label{E:eps0}
\left|\frac{\delta_{\eps}}{\eps}-\cc\right|<1 \quad \text{whenever} \quad 0<\eps<\eps_0. \quad \text{In particular, for $0<\eps<\eps_0$, we have $\delta_{\eps}<(\cc+1)\eps$.}
\end{equation}
We need one last piece of notation: For the cases $\cc=0$, $\cc \in (0,\infty)$, set
\begin{equation}\label{E:kappa}
\kap(\eps) \triangleq \left|\frac{\delta}{\eps}-\cc\right|,
\end{equation}
noting, of course, that $\lim_{\eps \searrow 0} \kap(\eps)=0$. Without further delay, let's state Theorem \ref{T:fluctuations-R-1-2}.


%
%
%

%

%

\begin{theorem}(Central Limit Theorem Type Result)\label{T:fluctuations-R-1-2}
Let $X_t^{\varepsilon,\delta}$ and $x_t$ solve \eqref{E:sie} and \eqref{E:nonlin-sys}, respectively. Suppose that we are in Regime $i \in \{1,2\}$, i.e., $\lim_{\eps \searrow 0}\delta_\eps/\eps = \cc \in [0,\infty)$. 
Let $Z=\{Z_t: t \ge 0\}$ be the unique strong solution of 
\begin{multline}\label{E:lim-fluct-R-1-2}
Z_t=\int_0^t [Df(x_s)+\gDk(x_s)]Z_s \thinspace ds - \frac{\cc}{2} \int_0^t \gDk(x_s)[f(x_s)+g(x_s)\kappa(x_s)] \thinspace ds + \int_0^t \sum_{i=1}^{m} Dg_i(x_s)Z_s \kappa_i(x_s) \thinspace ds \\ + \int_0^t \sigma(x_s) dW_s.
\end{multline}
Then, for any fixed $T \in (0,\infty)$, there exists a positive constant $C_{\ref{T:fluctuations-R-1-2}}$ such that for $0<\eps<\eps_0$, we have
\begin{equation}\label{E:FCLT}
\BE\left[\sup_{0 \le t \le T} |Z^{\eps,\delta}_t - Z_t|\right] = 
\frac{1}{\eps}\BE\left[\sup_{0 \le t \le T} |X^{\eps,\delta}_t - x_t - \eps Z_t|\right] \le  (\cc+1)\left[\eps+\sqrt{\delta}+\varkappa(\eps)\right]C_{\ref{T:fluctuations-R-1-2}}e^{C_{\ref{T:fluctuations-R-1-2}}T},
\end{equation}
where $\eps_0$ and $\kap(\eps) \searrow 0$ are as in equations \eqref{E:eps0} and \eqref{E:kappa}, respectively.
\end{theorem}

\begin{remark}\label{R:Interpretation}
We comment here on the significance and implications of Theorem \ref{T:fluctuations-R-1-2}, and the insights that one can draw from it. First, this result identifies the effective drift term $\ell(t) \triangleq \frac{\cc}{2} \int_0^t \gDk(x_s)[f(x_s)+g(x_s)\kappa(x_s)] \thinspace ds$ in equation \eqref{E:lim-fluct-R-1-2}, thereby characterizing the first-order fluctuation term $Z_t$ in the expansion $x_t + \eps Z_t + \dots$ for $X^{\eps,\delta}_t$ for Regimes 1 and 2. Just as importantly, Theorem \ref{T:fluctuations-R-1-2} furnishes us with estimates of the error incurred in approximating $X^{\eps,\delta}_t$ by $x_t + \eps Z_t$. A glance at equation \eqref{E:FCLT} reveals that $\BE[\sup_{0 \le t \le T} |X^{\eps,\delta}_t - x_t - \eps Z_t|]$ goes to zero faster than $\eps$. This gives us a precise sense in which we might interpret the statement $X^{\eps,\delta}_t = x_t + \eps Z_t + o(\eps)$. Finally, we note that Theorem \ref{T:fluctuations-R-1-2} allows us to replace, in an asymptotic regime, the non-Markovian process $X^{\eps,\delta}_t$ by the Markov process $x_t + \eps Z_t$. Indeed, the dependence of the dynamics of $X^{\eps,\delta}_t$ on not just the instantaneous value $X^{\eps,\delta}_t$, but also on the most recent sample $X^{\eps,\delta}_{\pi_\delta(t)}$, gives the system ``short memory" of duration $\delta$, thereby ruling out the Markov property. In contrast, the stochastic process $Z_t$ is given by a time-inhomogeneous Ito diffusion, rendering $x_t + \eps Z_t$ a time-inhomogeneous Markov process. 
\end{remark}

\begin{remark}\label{R:Regime3}
We note that our first result, Theorem \ref{T:LLN-II}, only required that both $\eps,\delta \searrow 0$. In particular, it did \textit{not} depend on the relative rates at which these parameters vanish. In contrast, Theorem \ref{T:fluctuations-R-1-2} above has been stated specifically for Regimes 1 and 2, and as will be seen, the calculations involved critically depend on the assumption that $\cc=\lim_{\eps \searrow 0}\delta_\eps/\eps \in [0,\infty)$. One can carry out a similar analysis of fluctuations for Regime 3 ($\cc=\infty$), but now, one works with the rescaled fluctuation process $U^{\eps,\delta}_t \triangleq \delta^{-1}(X^{\eps,\delta}_t-x_t)$. Here, the coarser parameter used to rescale the fluctuations $X^{\eps,\delta}_t-x_t$ is now $\delta$, rather than $\eps$. The calculations for Regime 3 are very similar. It is worth noting that since the $\eps/\delta$ term multiplying the Brownian noise (in the equation for $U^{\eps,\delta}_t$) vanishes as $\eps \searrow 0$, the limiting fluctuation process in Regime 3 is \textit{deterministic}.
\end{remark}

\begin{remark}\label{R:linear}
It is easily checked that the functions $f(x)=Ax$, $g(x)=B$, $\kappa(x)=-Kx$, $\sigma(x)=I_n$ where $A \in \BR^{n\times n}$, $B \in \BR^{n \times m}$, $K \in \BR^{m \times n}$ are constant matrices and $I_n$ denotes the $n \times n$ identity matrix,\footnote{Thus, in the notation of \eqref{E:ode-sampled}, $c(x,y)=Ax-BKy$.} satisfy Assumptions \ref{A:Smooth-LinearGrow} and \ref{A:Boundedness} and hence, Theorems \ref{T:LLN-II} and \ref{T:fluctuations-R-1-2} do apply. This case was studied in \cite{dhama2020approximation}, where the explicit expressions for $X^{\eps,\delta}_t$---possible due to linearity---were of critical importance. Returning to the present paper, we note that for the case of general $f,g,\kappa,\sigma$, it is impossible to obtain explicit formulas for $X^{\eps,\delta}_t$. Thus, part of the innovation in the present work is in the techniques used, which yield a result similar to the one obtained in \cite{dhama2020approximation}, but under \textit{significantly} weaker hypotheses.
\end{remark}


As was briefly noted in Section \ref{S:Introduction} earlier, Theorem \ref{T:fluctuations-R-1-2} is stronger than a result of {\sc clt} type since the former directly compares sample paths of $Z^{\eps,\delta}_t$ and $Z_t$, while the latter would involve a weaker notion of convergence, viz., convergence in distribution. In more detail, a family $\{Y^\eps:\eps \in (0,1)\}$ of random variables taking values in a complete, separable metric space $S$ is said to  \textit{converge in distribution} to the $S$-valued random variable $Y$  as $\eps \searrow 0$, denoted $Y^\eps \Rightarrow Y$, if for every bounded continuous function $f:S \to \BR$, we have $\lim_{\eps \searrow 0} \BE[f(Y^\eps)]=\BE[f(Y)]$ \cite{ConvProbMeas,EK86}. Since going from stronger (pathwise) approximation results to weaker (in distribution) approximation results is well-understood, we state the following corollary without proof. 

\begin{corollary}\label{C:FCLT}
Let $X_t^{\varepsilon,\delta}$ and $x_t$ solve \eqref{E:sie} and \eqref{E:nonlin-sys}, respectively. Suppose that we are in Regime $i \in \{1,2\}$, i.e., $\lim_{\eps \searrow 0}\delta_\eps/\eps = \cc \in [0,\infty)$. Let $Z=\{Z_t: t \ge 0\}$ be the unique strong solution of \eqref{E:lim-fluct-R-1-2}. Then, for any $T>0$, we have $Z^{\eps,\delta} \Rightarrow Z$ on $C([0,T];\BR^n)$ as $\eps \searrow 0$, where the space $C([0,T];\BR^n)$ of continuous functions mapping $[0,\infty)$ to $\BR^n$ is equipped with the metric induced by the sup norm.
\end{corollary}

\subsection{Applications to control}\label{SS:Control}
It was alluded to in Section \ref{S:Introduction} that the mathematical problems studied in this paper arise naturally in control theory. We now describe this connection in more detail; in the process, we will see how taking $c(x,y)$ of the particular form $f(x)+g(x)\kappa(y)$ enables us to derive results applicable to a fairly large class of nonlinear control systems \cite{vidyasagar2002nonlinear,khalil2002nonlinear,Zabczyk-MCT}. Just as importantly, we will see how the results of this paper can be interpreted in the control context. 

Suppose we have a physical system whose state $x(t)$---in the absence of any external control inputs---evolves in $\BR^n$ according to the {\sc ode} $\dot{x}=f(x)$ with smooth drift vector field $f:\BR^n \to \BR^n$. Assume further that we have $m$ control inputs $u_1,\dots,u_m \in \BR$ at our disposal, in the presence of which the evolution of the system is governed by the {\sc ode}
\begin{equation*}
\dot{x}=f(x) + \sum_{i=1}^m g_i(x)u_i,
\end{equation*}
where the control vector fields $g_i:\BR^n \to \BR^n$, $1 \le i \le m$, are sufficiently regular. For the nonlinear system affine in the control that we have described (taken from \cite{maheshwari2021stabilization}), the choice of control inputs $u_1,\dots,u_m$ is naturally dictated by the desired system behavior (e.g., asymptotic stabilization, minimization of a cost functional, etc). For our analysis, we will assume that a suitable \textit{feedback control law} of the form $u=\kappa(x)$ has \textit{already} been identified. Here, $u=(u_1,\dots,u_m) \in \BR^m$ and the vector-valued function $\kappa:\BR^n \to \BR^m$ has components $\kappa_1,\dots,\kappa_m$, i.e., the $i$-th control input $u_i$  is given by $u_i=\kappa_i(x)$. Thus, the \textit{idealized} behavior of the system is now governed by the {\sc ode}
$\dot{x}=f(x) + \sum_{i=1}^m g_i(x)\kappa_i(x)$,
which, of course, can be expressed as the integral equation \eqref{E:nonlin-sys}. 

The foregoing is idealized for multiple reasons, most notably (for our purposes) since it does not incorporate the effects of sampling and random perturbations.\footnote{It deserves mention that from the control standpoint, our calculations still involve many idealizations. For example, we have completely ignored issues of quantization of signals, issues of robustness, etc.} To understand the first of these, one notes that in most modern control systems, the control is effected by a digital computer, whose actions can be applied only at discrete time instants. This being the case, a \textit{sample-and-hold} implementation is frequently used: here, the state of the system is \textit{sampled} (measured) at closely spaced discrete time instants, the corresponding control action is computed according to the prescription $u=\kappa(x)$, and the control is \textit{held} fixed until the next sample is taken. For the case of periodic sampling, where samples are spaced $\delta$ units of time apart, a bit of reflection shows that the state of the system $x^\delta_t$ is governed by the differential equation
$\dot{x}_t^\delta=f(x^\delta_t) + g(x^\delta_t)\kappa(x^\delta_{\pi_\delta(t)})$,
which is precisely \eqref{E:ie}. Of course, we have taken $g$ to be the matrix with columns $g_1,\dots,g_m$. If we would now like to account for the effect of small state-dependent Brownian noise of size $\eps \in (0,1)$, we get the {\sc sde}  \eqref{E:sie}. 

Having observed that our setup covers a reasonably large class of control systems, let's try to interpret the results in the control context. It is worth emphasizing that we do \textit{not} concern ourselves with how to pick a suitable control law, but rather with issues of \textit{performance degradation} that inevitably arise due to sampling effects and external random perturbations. Our first result, Theorem \ref{T:LLN-II} confirms what one would intuitively expect: over finite time intervals, as the frequency of sampling increases and the strength of the noise vanishes, one recovers the idealized behavior given by $x_t$. For the cases when the time $\delta$ between samples decreases faster than, or at the same rate as, the strength $\eps$ of the noise, our second result, Theorem \ref{T:fluctuations-R-1-2}, refines this zeroth-order term $x_t$ by giving a first-order correction term given by the stochastic process $Z_t$. Theorem \ref{T:fluctuations-R-1-2} also quantifies (in a suitable sense) the error incurred in approximating the process $X^{\eps,\delta}_t$ which involves hybrid dynamics by the \textit{non-hybrid} proxy $x_t + \eps Z_t$.

\section{Limiting Mean Behavior}\label{S:FLLN}
This section is devoted to the proof of Theorem \ref{T:LLN-II}, which, as noted earlier, can be interpreted as a result of the form of the {\sc lln}. The primary ingredients in the proof are Propositions \ref{P:LLN-Stoc-App} and \ref{P:LLN-Det-Approx} which provide, respectively, pathwise estimates for the quantities $\sup_{0\le t \le T}|X_t^{\eps,\delta}-x_t^\delta|^p$ and $\sup_{0\le t \le T}|x_t^\delta-x_t|^p$, $p \in \{1,2\}$. In Section \ref{S:LLN-Theorem}, we state and prove these propositions, relying on a series of auxiliary lemmas, and conclude with the proof of Theorem \ref{T:LLN-II}. For ease of reading, the proofs of the auxiliary lemmas are deferred to Section \ref{S:LLN-Lemma}.



\subsection{Proof of Theorem \ref{T:LLN-II}}\label{S:LLN-Theorem}

\begin{proposition}\label{P:LLN-Stoc-App}
Let $x_t^\delta$ and $X_t^{\eps,\delta}$ be the solutions of \eqref{E:ie} and \eqref{E:sie}, respectively. Then, for $p\in \{1,2\}$ and any $T>0$, there exists a positive constant $C_{\ref{P:LLN-Stoc-App}}$ 
such that for any $\eps, \delta>0$, we have
\begin{equation*}
\BE\left[ \sup_{0\le t \le T}|X_t^{\eps,\delta}-x_t^\delta|^p \right] \le \eps^p C_{\ref{P:LLN-Stoc-App}} e^{C_{\ref{P:LLN-Stoc-App}}T}.
\end{equation*}
\end{proposition}

\begin{proposition}\label{P:LLN-Det-Approx}
Let $x_t^\delta$ and $x_t$ be the solutions of \eqref{E:ie} and \eqref{E:nonlin-sys}, respectively. Then, for any $T>0,$ and $p \in \{1,2\},$ there exists a positive constant $C_{\ref{P:LLN-Det-Approx}}$ 
such that for any $\eps,\delta>0$, we have 
\begin{equation*}
\sup_{0\le t \le T} |x_t^\delta-x_t|^p \le \delta^p C_{\ref{P:LLN-Det-Approx}}e^{C_{\ref{P:LLN-Det-Approx}}T}.
\end{equation*}
\end{proposition}

To prove Propositions \ref{P:LLN-Stoc-App}, \ref{P:LLN-Det-Approx} (and also later in the sequel), we require estimates on the terms $\sup_{0\le t \le T}|x_t^\delta|^p,\thinspace \sup_{0\le t \le T}|x_t|^p, \thinspace |x_t-x_{\pi_{\delta}(t)}|^p$, and $\BE \left\{\left|\int_0^T \sigma(X_s^{\eps,\delta})dW_s \right|^p\right\}$ for $p \in \{1,2\}$.
The bounds for the quantities $\sup_{0\le t \le T}|x_t^\delta|^p,\thinspace \sup_{0\le t \le T}|x_t|^p$, and $|x_t-x_{\pi_{\delta}(t)}|^p$ are obtained in Lemmas \ref{L:Nonlinear-Sys-Sampling-Est} and \ref{L:Sampling-Difference} below, while the term $\BE \left\{\left|\int_0^T \sigma(X_s^{\eps,\delta})dW_s \right|^p\right\}$ is estimated in Lemma \ref{L:L2-Est}. The proofs of Lemmas \ref{L:Nonlinear-Sys-Sampling-Est}, \ref{L:Sampling-Difference} and \ref{L:L2-Est} are provided in Section \ref{S:LLN-Lemma}.

\begin{lemma}\label{L:Nonlinear-Sys-Sampling-Est}
Let $x_t^\delta$ and $x_t$ be the solutions of the equations \eqref{E:ie} and \eqref{E:nonlin-sys}, respectively. Then, for $p \in \{1,2\}$ and $T > 0$,  there exists a positive constant $C_{\ref{L:Nonlinear-Sys-Sampling-Est}}$ 
such that 
\begin{equation*}
\begin{aligned}
\sup_{0\le t \le T}|x_t^\delta|^p  \le C_{\ref{L:Nonlinear-Sys-Sampling-Est}}e^{C_{\ref{L:Nonlinear-Sys-Sampling-Est}}T} \qquad \text{and} \qquad
\sup_{0\le t \le T}|x_t|^p   \le C_{\ref{L:Nonlinear-Sys-Sampling-Est}}e^{C_{\ref{L:Nonlinear-Sys-Sampling-Est}}T}.
\end{aligned}
\end{equation*}
\end{lemma}

\begin{lemma}\label{L:Sampling-Difference}
Let $x_t$ be the solution of \eqref{E:nonlin-sys}. Then, for $p \in \{1,2\}$ and $T > 0,$ there exists a positive constant $C_{\ref{L:Sampling-Difference}}$ 
such that 
$|x_t-x_{\pi_{\delta}(t)}|^p \le \delta^p C_{\ref{L:Sampling-Difference}}e^{C_{\ref{L:Sampling-Difference}}T}, \thinspace t\in[0,T].$
\end{lemma}

\begin{lemma}\label{L:L2-Est}
Let $X_t^{\eps,\delta}$ be the strong solution of \eqref{E:sie}. Then, for $p \in \{1,2\}$ and $T > 0,$ there exists a positive constant $C_{\ref{L:L2-Est}}$ 
such that
\begin{equation}\label{E:L2-Est}
\begin{aligned}
\BE \left[\sup_{0\le t \le T}\left|X_t^{\eps,\delta}\right|^2 \right]  \le C_{\ref{L:L2-Est}}e^{C_{\ref{L:L2-Est}}T},\quad \text{and} \quad 
\BE \left\{\left|\int_0^T \sigma(X_s^{\eps,\delta})dW_s \right|^p\right\}  \le C_{\ref{L:L2-Est}}e^{C_{\ref{L:L2-Est}}T}.
\end{aligned}
\end{equation}
\end{lemma}

\begin{proof}[Proof of Proposition \ref{P:LLN-Stoc-App}]
For $t\ge 0$, $p \in \{1,2\}$, we use the integral representations of $x_t^\delta$ and $X_t^{\eps,\delta}$ from \eqref{E:ie} and \eqref{E:sie} to get 
\begin{multline*}
|X_t^{\eps,\delta}-x_t^\delta|^p  \le \\
C\left[ \left|\int_0^t[f(X_s^{\eps,\delta})-f(x_s^{\delta})]ds\right|^p + \left|\int_0^t\left[g(X_s^{\eps,\delta})\kappa(X_{\pi_\delta(s)}^{\eps,\delta})-g(x_s^\delta)\kappa(x_{\pi_{\delta}(s)}^\delta)\right]ds\right|^p + \eps^p\left|\int_0^t \sigma(X_s^{\eps,\delta})dW_s \right|^p\right].
\end{multline*}
Now, using H$\ddot{\text{o}}$lder's inequality for $p=2$ and the inequality $|\int \cdot|\le \int|\cdot|$ for Lebesgue-Stieltjes integrals for $p=1$, we get
\begin{multline}\label{E:LLN-Stoch-App}
|X_t^{\eps,\delta}-x_t^{\delta}|^p \le \\ C \left[ \int_0^t \left|f(X_s^{\eps,\delta})-f(x_s^\delta)\right|^p ds + \underbrace{\int_0^t \left|g(X_s^{\eps,\delta})\kappa(X_{\pi_\delta(s)}^{\eps,\delta})-g(x_s^\delta)\kappa(x_{\pi_{\delta}(s)}^\delta)\right|^p ds}_{J_1} +\eps^p\left|\int_0^t \sigma(X_s^{\eps,\delta})dW_s \right|^p\right].
\end{multline}
Adding and subtracting $\kappa(x_{\pi_\delta(s)}^{\delta})$ in the first term of the integrand of $J_1,$ and then using \eqref{E:Triangle-Inequ}, we have
\begin{equation*}
\begin{aligned}
J_1 & = \int_0^t \left|g(X_s^{\eps,\delta})\left\{\kappa(X_{\pi_\delta(s)}^{\eps,\delta})-\kappa(x_{\pi_\delta(s)}^{\delta})+\kappa(x_{\pi_\delta(s)}^{\delta})\right\}-g(x_s^\delta)\kappa(x_{\pi_\delta(s)}^{\delta})\right|^p ds\\
& \le C\left\{\int_0^t |g(X_s^{\eps,\delta})|^p\left|\kappa(X_{\pi_\delta(s)}^{\eps,\delta})-\kappa(x_{\pi_\delta(s)}^{\delta})\right|^p ds + \int_0^t \left|g(X_s^{\eps,\delta})-g(x_s^\delta)\right|^p\left|\kappa(x_{\pi_\delta(s)}^{\delta})\right|^p ds \right\}.
\end{aligned}
\end{equation*}
Now, using the Lipschitz continuity of $\kappa$ and $g$, boundedness of $g$, and the linear growth of $\kappa$ as in Assumptions \ref{A:Smooth-LinearGrow}, \ref{A:Boundedness}, we get
\begin{equation*}
J_1 \le C\left\{ \int_0^t \sup_{0\le r \le s}|X_r^{\eps,\delta}-x_r^{\delta}|^p ds + \int_0^t \left(1+\sup_{0\le r \le s}|x_r^\delta|^p\right) \sup_{0\le r \le s}|X_r^{\eps,\delta}-x_r^\delta|^p ds \right\}.
\end{equation*}
Returning to \eqref{E:LLN-Stoch-App}, we use the above estimate on $J_1$, the Lipschitz continuity of $f$, and Lemma \ref{L:Nonlinear-Sys-Sampling-Est} to get $|X_t^{\eps,\delta}-x_t^\delta|^p \le C\left[\int_0^t \sup_{0\le r \le s}|X_r^{\eps,\delta}-x_r^\delta|^p ds + \eps^p\sup_{0\le t \le T}\left|\int_0^t \sigma(X_s^{\eps,\delta})dW_s \right|^p\right]$. Taking expectations, we have 
\begin{equation*}
\BE \left[ \sup_{0\le t \le T}|X_t^{\eps,\delta}-x_t^\delta|^p \right] \le C\left[ \int_0^T \BE \left(\sup_{0\le r \le s}|X_r^{\eps,\delta}-x_r^\delta|^p\right)ds + \eps^p \BE \left\{ \sup_{0\le t \le T}\left|\int_0^t \sigma(X_s^{\eps,\delta})dW_s \right|^p\right\}\right].
\end{equation*}
Using Doob's maximal inequality \cite[Theorem 1.3.8(iv)]{KS91} for $p=2$ and \eqref{E:L1-Est} for $p=1$ for last term in the above equation, and recalling Lemma \ref{L:L2-Est}, the claim follows by an application of Gronwall's inequality. 

\end{proof}

\begin{proof}[Proof of Proposition \ref{P:LLN-Det-Approx}]
Using the integral representations for $x_t^\delta$ and $x_t$ from \eqref{E:ie} and \eqref{E:nonlin-sys}, followed by H$\ddot{\text{o}}$lder's inequality for $p=2$ and the integral inequality $|\int \cdot|\le \int |\cdot|$ for $p=1$, we get
\begin{equation}\label{E:LLN-Det-Approx}
|x_t^\delta-x_t|^p \le C\left[\int_0^t \left|f(x_s^\delta)-f(x_s)\right|^pds + \underbrace{\int_0^t \left|g(x_s^\delta)\kappa(x_{\pi_{\delta}(s)}^\delta)-g(x_s)\kappa(x_s)\right|^p ds}_{J_2} \right].
\end{equation}
For $J_2$, we now use \eqref{E:Triangle-Inequ} to get
\begin{equation*}
\begin{aligned} 
J_2 & = \int_0^t \left|g(x_s^\delta)\left\{\kappa(x_{\pi_{\delta}(s)}^\delta)-\kappa(x_{\pi_{\delta}(s)})\right\}+ \left\{g(x_s^\delta)-g(x_s)\right\}\kappa(x_{\pi_{\delta}(s)})+g(x_s)\left[\kappa(x_{\pi_{\delta}(s)})-\kappa(x_s)\right]\right|^p ds\\
& \le C\left( \int_0^t|g(x_s^\delta)|^p\left|\kappa(x_{\pi_{\delta}(s)}^\delta)-\kappa(x_{\pi_{\delta}(s)})\right|^p ds+ \int_0^t |g(x_s^\delta)-g(x_s)|^p\left|\kappa(x_{\pi_{\delta}(s)})\right|^p ds\right.\\
& \left.+\int_0^t |g(x_s)|^p \left|\kappa(x_{\pi_{\delta}(s)})- \kappa(x_s)\right|^p ds \right).
\end{aligned}
\end{equation*}
Next, using the boundedness of $g$, and Lipschitz continuity and linear growth of $\kappa$ as in Assumptions \ref{A:Smooth-LinearGrow}, \ref{A:Boundedness}, followed by Lemmas \ref{L:Nonlinear-Sys-Sampling-Est} and \ref{L:Sampling-Difference}, we get
\begin{equation*}
J_2 \le C\left\{\int_0^T \sup_{0\le r \le s}|x_r^\delta-x_r|^p ds + \int_0^T \sup_{0\le r \le s}|x_r^\delta-x_r|^p(1+|x_{\pi_{\delta}(s)}|^p)ds + \int_0^T |x_{\pi_{\delta}(s)}-x_s|^p ds \right\}.
\end{equation*}
Thus,
$J_2 \le C\int_0^T \sup_{0\le r \le s}|x_r^\delta-x_r|^p ds + \delta^p Ce^{CT}$.
Now, using the Lipschitz continuity of $f$ in \eqref{E:LLN-Det-Approx} and the obtained estimate for $J_2$, an application of Gronwall's inequality gives the required result.
\end{proof}
We now prove our main result.
\begin{proof}[Proof of Theorem \ref{T:LLN-II}]
For $p=1,2,$ using the inequality \eqref{E:Triangle-Inequ}, we have 
\begin{equation*}
\begin{aligned}
|X_t^{\eps,\delta}-x_t|^p \le \left[|X_t^{\eps,\delta}-x_t^{\delta}|+|x_t^{\delta}-x_t|\right]^p
\le C\left[|X_t^{\eps,\delta}-x_t^{\delta}|^p + |x_t^{\delta}-x_t|^p \right].
\end{aligned}
\end{equation*}
Hence 
\begin{equation*}
\sup_{0\le t \le T}|X_t^{\eps,\delta}-x_t|^p \le C\left[\sup_{0\le t \le T}|X_t^{\eps,\delta}-x_t^{\delta}|^p + \sup_{0 \le t \le T}|x_t^{\delta}-x_t|^p \right].
\end{equation*}
The proof is now completed using Propositions \ref{P:LLN-Stoc-App} and \ref{P:LLN-Det-Approx}. 
\end{proof}


\subsection{Proofs of supporting Lemmas}\label{S:LLN-Lemma}
We now provide the proofs of Lemmas \ref{L:Nonlinear-Sys-Sampling-Est}, \ref{L:Sampling-Difference}, and \ref{L:L2-Est}. Here, we will frequently be using (without explicit mention) Assumptions \ref{A:Smooth-LinearGrow} and \ref{A:Boundedness}, which encode Lipschitz and linear growth assumptions on $f,g,\kappa,\sigma$, and boundedness of $g$. 

\begin{proof}[Proof of Lemma \ref{L:Nonlinear-Sys-Sampling-Est}]
Since $x_t^\delta = x_0 + \int_0^t [f(x_s^\delta)+ g(x_s^\delta)\kappa(x_{\pi_\delta(s)}^\delta)] ds$, using \eqref{E:Triangle-Inequ} and H$\ddot{\text{o}}$lder's inequality for $p=2$ and $|\int \cdot| \le \int |\cdot|$ for $p=1$, 
we get
\begin{equation*}
\begin{aligned}
|x_t^\delta|^p &\le C\left\{|x_0|^p + \int_0^t \left((1+|x_s^\delta|^p)+(1+|x_{\pi_{\delta}(s)}^\delta|^p) \right) ds \right\} 
 \le C\left\{1 + \int_0^t \sup_{0\le r \le u}|x_r^\delta|^p du\right\}.
\end{aligned}
\end{equation*}
The desired bound on $\sup_{0 \le t \le T}|x^\delta_t|^p$ now follows from Gronwall's inequality. A similar straightforward calculation yields
$|x_t|^p \le C \left\{ 1 + \int_0^t \sup_{0 \le r \le s}|x_r|^p ds \right\}$.
Once again, using Gronwall's inequality, we get the required bound. 
\end{proof}


\begin{proof}[Proof of Lemma \ref{L:Sampling-Difference}]
For $p=1,$ 
equation \eqref{E:nonlin-sys} implies
\begin{equation}\label{E:System-Difference-Est}
|x_t-x_{\pi_{\delta}(t)}|\le \int_{\pi_{\delta}(t)}^t|f(x_s)+g(x_s)\kappa(x_s)|\thinspace ds \le C\left\{\delta + \int_{\pi_{\delta}(t)}
^{t}|x_s|\thinspace ds \right\}.
\end{equation}
The required estimate 
is obtained by combining \eqref{E:System-Difference-Est} and Lemma \ref{L:Nonlinear-Sys-Sampling-Est}. For $p=2,$ 
we use H$\ddot{\text{o}}$lder's inequality in \eqref{E:nonlin-sys}, and then \eqref{E:Triangle-Inequ} to get
\begin{equation*}
\begin{aligned}
|x_t-x_{\pi_{\delta}(t)}|^2 =\left|\int_{\pi_{\delta}(t)}^{t}[f(x_u)+g(x_u)\kappa(x_u)]du\right|^2 & \le (t-\pi_{\delta}(t))\int_{\pi_{\delta}(t)}^{t} \left|f(x_u)+g(x_u)\kappa(x_u)\right|^2 du\\
& \le \delta C \int_{\pi_{\delta}(t)}^{t} \left\{|f(x_u)|^2 +|g(x_u)\kappa(x_u)|^2\right\}du.
\end{aligned}
\end{equation*}
Therefore, we have
$|x_t-x_{\pi_{\delta}(t)}|^2 \le \delta^2 C + \delta C \int_{\pi_{\delta}(t)}^{t}|x_u|^2 du.$
Using Lemma \ref{L:Nonlinear-Sys-Sampling-Est}, we get the desired result.
\end{proof}



\begin{proof}[Proof of Lemma \ref{L:L2-Est}]
Let $|\cdot|$ be the one norm. Using \eqref{E:Triangle-Inequ} and H$\ddot{\text{o}}$lder's inequality in \eqref{E:sie},
we get
 \begin{equation*}
 \begin{aligned}
 \left|X_t^{\eps,\delta}\right|^2 & \le C \left[|x_0|^2 + T \int_0^t \left(|f(X_s^{\eps,\delta})|^2+\left|g(X^{\eps,\delta}_s)\kappa (X^{\eps,\delta}_{\pi_\delta(s)})\right|^2 \right)ds + \eps^2 \left|\int_0^t \sigma(X_s^{\eps,\delta})dW_s \right|^2  \right]\\
 & \le C \left[ 1+ \int_0^T \sup_{0 \le r \le s}\left|X_r^{\eps,\delta}\right|^2ds + \eps^2 \sup_{0\le t \le T} \left|\int_0^t \sigma(X_s^{\eps,\delta})dW_s \right|^2 \right].
 \end{aligned}
 \end{equation*}
 Now, using Doob's maximal inequality \cite[Theorem 1.3.8(iv)]{KS91}, we get
\begin{equation*}
\BE \left[ \sup_{0 \le t \le T}\left|X_t^{\eps,\delta}\right|^2\right] \le C \left[ 1+ \BE\left(\int_0^T \sup_{0 \le r \le s}\left|X_r^{\eps,\delta}\right|^2ds\right) + \eps^2 \BE \left(\left|\int_0^T \sigma(X_s^{\eps,\delta})dW_s \right|^2\right) \right].
\end{equation*} 
Since $\left|\int_0^T \sigma(X_s^{\eps,\delta})dW_s \right|^2\le C \sum_{i=1}^n\sum_{j=1}^n\left|\int_0^T \sigma_{ji}(X_s^{\eps,\delta})dW_s^i \right|^2$, using the Ito isometry, we get
\begin{equation}\label{E:L2-Est-1}
\BE\left(\left|\int_0^T \sigma(X_s^{\eps,\delta})dW_s \right|^2\right)\le C\left[1+ \BE \left( \int_0^T \sup_{0 \le r \le s}|X_r^{\eps,\delta}|^2ds \right)\right].
\end{equation} 
Hence,\\
$\BE \left[ \sup_{0 \le t \le T}\left|X_t^{\eps,\delta}\right|^2\right] \le C \left[ 1+ (\eps^2 +1)\int_0^T \BE \left[\sup_{0 \le r \le s}\left|X_r^{\eps,\delta}\right|^2 \right]ds \right]$. 
The first claim in \eqref{E:L2-Est} now\\ easily follows by Gronwall's inequality. For the second part of the lemma (for $p=2$), we use the obtained estimate on $\BE\left[\sup_{0 \le t \le T} |X^{\eps,\delta}_t|^2 \right]$ together with equation \eqref{E:L2-Est-1} to get $\BE\left(\left|\int_0^T \sigma(X_s^{\eps,\delta})dW_s \right|^2\right)\le Ce^{CT}.$ For the case $p=1$, we start by noting that 
 \begin{equation*}
\BE \left(\left|\int_0^T \sigma(X_s^{\eps,\delta})dW_s \right|\right) \le\BE \left[\sup_{0\le t \le T}\left|\int_0^t \sigma(X_s^{\eps,\delta})dW_s \right|\right] \le \sum_{i=1}^{n}\sum_{j=1}^{n}\BE\left[ \sup_{0\le t \le T}\left|\int_0^t \sigma_{ji}(X_s^{\eps,\delta})dW_s^i \right| \right].
 \end{equation*}
Now, using the Burkholder-Davis-Gundy inequalities \cite[Theorem 3.3.28]{KS91}, followed by Jensen's inequality for concave functions,\footnote{For a concave function $\varphi$ and a random variable $X$ with $\BE[|X|]$, $\BE[|\varphi(X)|]<\infty$, we have $ \BE[\varphi(X)]\le \varphi(\BE[X])$.} 
 we get
\begin{multline}\label{E:L1-Est}
\BE \left[\sup_{0\le t \le T}\left|\int_0^t \sigma(X_s^{\eps,\delta})dW_s \right|\right] \le C \sum_{i=1}^{n}\sum_{j=1}^{n} \BE \left[ \left( \int_0^T \sigma_{ji}^2(X_s^{\eps,\delta})ds\right)^{\frac{1}{2}}\right]\\ \le C\sum_{i=1}^{n}\sum_{j=1}^{n} \left( \BE \int_0^T \sigma_{ji}^2(X_s^{\eps,\delta})ds\right)^{\frac{1}{2}}
 \le  C  \left(  \int_0^T \left\{1+\BE\left(|X_s^{\eps,\delta}|^2\right) \right\}ds\right)^{\frac{1}{2}}.
\end{multline} 
The required result is now easily obtained by using the first part of this lemma.
\end{proof}

\section{Analysis of fluctuations: Regimes 1 and 2}\label{S:FCLT}
In this section, we prove our second main result, namely Theorem \ref{T:fluctuations-R-1-2}. Our thoughts are structured as follows. First, in Section \ref{S:CLT-Prop}, we state (without proof) Propositions \ref{P:MainTermApprox} through \ref{P:Noise-Term-Est}, which are the principal building blocks in the proof of Theorem \ref{T:fluctuations-R-1-2}, and show how to assemble these pieces to obtain the proof of Theorem \ref{T:fluctuations-R-1-2}. The subsequent sections are devoted to the proofs of the aforementioned propositions. In Section \ref{S:MainTermApp}, we prove Proposition \ref{P:MainTermApprox}. Since this involves several intricate calculations, our arguments are broken down into a series of bite-sized lemmas, the proofs of which are deferred to Section \ref{S:CLT-Lemmas}. The proofs of Propositions \ref{P:f-Taylor-Approx} through \ref{P:Noise-Term-Est} are provided in Section \ref{S:PropositionsProofs}.

%
%

%
%
%
%

\subsection{Proof of Theorem \ref{T:fluctuations-R-1-2}}\label{S:CLT-Prop}

We saw in Section \ref{S:ProblemStatementResults} that for Regimes 1 and 2, the rescaled fluctuation process $Z_t^{\eps,\delta}\triangleq {\eps}^{-1}({X_t^{\eps,\delta}-x_t})$ satisfies the equation \eqref{E:Cent-limit-eq-Rem}. Further, anticipating that the remainder term $\mathsf{R}^{\eps,\delta}_t = \sum_{i=1}^4 \mathsf{R}_i^{\eps,\delta}(t)$ appearing in \eqref{E:Cent-limit-eq-Rem} (with decomposition given in \eqref{E:Remainder-Terms}) would be of order $\mathscr{O}(\eps)$, we informally summarized the description of $Z^{\eps,\delta}_t$ by \eqref{E:fluct-processes}. To transition, then, from equation \eqref{E:Cent-limit-eq-Rem} (or \eqref{E:fluct-processes}) for $Z^{\eps,\delta}_t$ to the equation \eqref{E:lim-fluct-R-1-2} for the limiting fluctuation process $Z_t$, it is evident that we need to show that 
\begin{equation}\label{E:ell}
\lim_{\substack{\eps,\delta \searrow 0\\ \delta/\eps \to \cc}} \int_0^t \gDk(x_s)\frac{X_s^{\varepsilon, \delta}-X_{\pi_\delta(s)}^{\varepsilon, \delta}}{\eps}\thinspace ds =\ell(t) \quad \text{where} \quad 
\ell(t) \triangleq \frac{\cc}{2} \int_0^t \gDk(x_s)[f(x_s)+g(x_s)\kappa(x_s)] \thinspace ds, 
\end{equation}
while also obtaining estimates for $\left|\int_0^t \gDk(x_s)\frac{X_s^{\varepsilon, \delta}-X_{\pi_\delta(s)}^{\varepsilon, \delta}}{\eps}\thinspace ds - \ell(t)\right|$  and the remainder terms $\mathsf{R}_i^{\eps,\delta}(t)$, $1 \le i \le 4$. We will find it useful later to decompose the term $\ell(t)$ according to 
\begin{equation}\label{E:ell_1-and-ell_2}
\ell(t)= \ell_1(t)+\ell_2(t)\triangleq \frac{\cc}{2} \int_0^t \gDk(x_s)[f(x_s)]ds+ \frac{\cc}{2} \int_0^t\gDk(x_s)[g(x_s)\kappa(x_s)] ds.
\end{equation}


We start with Proposition \ref{P:MainTermApprox}, which allows us to estimate the error between $\int_0^t \gDk(x_s)\frac{X_s^{\varepsilon, \delta}-X_{\pi_\delta(s)}^{\varepsilon, \delta}}{\eps}\thinspace ds$ and $\ell(t)$, yielding, in particular, the convergence of the former to the latter. This is the key step in identification of the effective drift term in \eqref{E:lim-fluct-R-1-2}. Before we begin, we note that since we will always be working in Regime $i \in \{1,2\}$, i.e., $\lim_{\eps \searrow 0}\delta_\eps/\eps = \cc \in [0,\infty)$, we recall from \eqref{E:eps0} that choosing $0<\eps<\eps_0$ ensures that $\delta<(\cc+1)\eps$. 
 

\begin{proposition}\label{P:MainTermApprox}
Let $\ell(t)$ be as in equation \eqref{E:ell_1-and-ell_2}. For any fixed $T>0$, there exists a positive constant $C_{\ref{P:MainTermApprox}}$ such that whenever $0<\eps<\eps_0$, we have
\begin{equation*}
\BE \left[\sup_{0 \le t \le T} \left|\int_0^t \gDk(x_s)\frac{X_s^{\varepsilon, \delta}-X_{\pi_\delta(s)}^{\varepsilon, \delta}}{\eps}\thinspace ds - 
\ell(t) \right|\right] \le  (\cc+1)\left(\varkappa(\eps)+\sqrt{\delta}+\eps\right)C_{\ref{P:MainTermApprox}}e^{C_{\ref{P:MainTermApprox}}T},
\end{equation*}
where $\kap(\eps)$ defined in \eqref{E:kappa} satisfies $\lim_{\eps \searrow 0}\kap(\eps)=0$.
\end{proposition}


As noted earlier, we will tackle the proof of this result in Section \ref{S:MainTermApp}. We move next to Propositions \ref{P:f-Taylor-Approx} through \ref{P:Noise-Term-Est}. 


\begin{proposition}\label{P:f-Taylor-Approx}
Let $X_t^{\varepsilon,\delta}$ and $x_t$ solve \eqref{E:sie} and \eqref{E:nonlin-sys}, respectively. Then, for any fixed $T>0$, there exists a positive constant $C_{\ref{P:f-Taylor-Approx}}$ such that for any $0< \eps <\eps_0$, we have
\begin{equation*}
\begin{aligned}
\BE \left[\sup_{0\le t \le T}\int_0^t\left|\frac{f(X_s^{\eps,\delta})-f(x_s)}{\eps}-Df(x_s) Z_s^{\eps,\delta}\right|\thinspace ds\right] &\le (\eps +\cc \delta)C_{\ref{P:f-Taylor-Approx}}e^{C_{\ref{P:f-Taylor-Approx}}T},\\
\BE\left[\sup_{0\le t \le T}\int_0^t\left|g(x_s)\frac{\kappa(X_s^{\eps,\delta})-\kappa(x_s)}{\eps}-\gDk(x_s)Z_s^{\eps,\delta}\right|\thinspace ds\right] &\le (\eps +\cc \delta)C_{\ref{P:f-Taylor-Approx}}e^{C_{\ref{P:f-Taylor-Approx}}T}.
\end{aligned}
\end{equation*}
\end{proposition}


\begin{proposition}\label{P:Add-Term-Approx}
For any fixed $T>0$, there exists a positive constant $C_{\ref{P:Add-Term-Approx}}$ such that for any $0< \eps <\eps_0$, we have
\begin{equation*}
\BE\left[\sup_{0 \le t \le T}\int_0^t\left|\frac{[g(X_s^{\eps,\delta})-g(x_s)]\kappa(X_s^{\eps,\delta})}{\eps}- \sum_{i=1}^{m}Dg_i(x_s)Z_s^{\eps,\delta}\kappa_i(x_s)\right|\thinspace ds\right]\le (\eps +\cc \delta)C_{\ref{P:Add-Term-Approx}}e^{C_{\ref{P:Add-Term-Approx}}T}.
\end{equation*}
\end{proposition}



\begin{proposition}\label{P:Sampling-Term-Approx}
Let $X_t^{\varepsilon,\delta}$ be the solution of {\sc sde} \eqref{E:sie}. Then, for any fixed $T>0$, there exists a positive constant $C_{\ref{P:Sampling-Term-Approx}}$ such that for any $0< \eps <\eps_0$, we have
\begin{equation*}
\BE\left[\sup_{0\le t \le T}{\int_0^t}\left|\frac{g(X_s^{\eps,\delta})\left[\kappa(X_s^{\eps,\delta})-\kappa(X_{\pi_\delta(s)}^{\eps,\delta})\right]}{\eps}-\gDk(x_s)\frac{X_s^{\varepsilon, \delta}-X_{\pi_\delta(s)}^{\varepsilon, \delta}}{\eps}\right|\thinspace ds\right]\le \{(\cc+1)\eps +\delta \cc\}C_{\ref{P:Sampling-Term-Approx}}e^{C_{\ref{P:Sampling-Term-Approx}}T}.
\end{equation*}
\end{proposition}

\begin{proposition}\label{P:Noise-Term-Est}
Let $X_t^{\varepsilon,\delta}$ and $x_t$ solve \eqref{E:sie} and \eqref{E:nonlin-sys}, respectively. Then, for any fixed $T>0,$ there exists a positive constant $C_{\ref{P:Noise-Term-Est}}$ such that 
\begin{equation*}
\BE\left[ \sup_{0\le t \le T}\left|\int_0^t \{\sigma(X_s^{\eps,\delta})-\sigma(x_s)\}dW_s\right| \right] \le  (\eps +\delta)C_{\ref{P:Noise-Term-Est}}e^{C_{\ref{P:Noise-Term-Est}}T}.
\end{equation*}
\end{proposition}


While the proofs of Propositions \ref{P:f-Taylor-Approx} through \ref{P:Noise-Term-Est} will be discussed in Section \ref{S:PropositionsProofs}, we pause to comment on the significance of these results. Proposition \ref{P:f-Taylor-Approx} enables us to estimate the error terms $\mathsf{R}^{\eps,\delta}_1(t)$, $\mathsf{R}^{\eps,\delta}_4(t)$ from equation \eqref{E:Remainder-Terms}. Propositions \ref{P:Add-Term-Approx} and \ref{P:Sampling-Term-Approx} allow us to bound $\mathsf{R}^{\eps,\delta}_3(t)$ and $\mathsf{R}^{\eps,\delta}_2(t)$, respectively. Finally, 
Proposition \ref{P:Noise-Term-Est} estimates $\sup_{0\le t \le T}\left|\int_0^t \{\sigma(X_s^{\eps,\delta})-\sigma(x_s)\}dW_s\right|$, effectively enabling us to replace $\int_0^t \sigma(X^{\eps,\delta}_s)\thinspace dW_s$ by $\int_0^t \sigma(x_s)\thinspace dW_s$ in \eqref{E:lim-fluct-R-1-2}.

We now prove our main result Theorem \ref{T:fluctuations-R-1-2}.
\begin{proof}[Proof of Theorem \ref{T:fluctuations-R-1-2}]
Recalling that $X^{\eps,\delta}_t - x_t - \eps Z_t = \eps (Z^{\eps,\delta}_t - Z_t )$, where $Z^{\eps,\delta}_t$ and $Z_t$ are given by \eqref{E: cent-limit-eq} and \eqref{E:lim-fluct-R-1-2}, respectively, some simple algebra yields
%
\begin{multline*}
\begin{aligned}
Z_t^{\eps,\delta}-Z_t &= \int_0^t \left\{ \frac{f(X_s^{\eps,\delta})-f(x_s)}{\eps}-Df(x_s)Z_s^{\eps,\delta}+Df(x_s)Z_s^{\eps,\delta}- Df(x_s)Z_s \right\}ds\\
&+ \int_0^t \left\{\frac{[g(X_s^{\eps,\delta})-g(x_s)]\kappa(X_s^{\eps,\delta})}{\eps}-\sum_{i=1}^{m} Dg_i(x_s)Z_s^{\eps,\delta} \kappa_i(x_s) \right.\\
& \qquad \qquad \qquad \qquad \qquad \qquad \qquad \quad \quad \left. +\sum_{i=1}^{m} Dg_i(x_s)Z_s^{\eps,\delta} \kappa_i(x_s) -\sum_{i=1}^{m} Dg_i(x_s)Z_s \kappa_i(x_s) \right\}ds
\end{aligned}
\end{multline*}

\begin{multline*}
\begin{aligned}
& \qquad  + \int_0^t \left\{ \frac{g(x_s)[\kappa(X_s^{\eps,\delta})-\kappa(x_s)]}{\eps}-\gDk(x_s)Z_s^{\eps,\delta} + \gDk(x_s)Z_s^{\eps,\delta} -\gDk(x_s)Z_s \right\}ds\\
& \qquad  +\int_0^t \left\{ \frac{g(X_s^{\eps,\delta})[\kappa(X_{\pi_\delta(s)}^{\eps,\delta})-\kappa(X_s^{\eps,\delta})]}{\eps}+ \gDk(x_s)\frac{X_s^{\eps,\delta}-X_{\pi_{\delta}(s)}^{\eps,\delta}}{\eps} \right\}ds\\
& \quad \quad + \int_0^t \left\{-\gDk(x_s)\frac{X_s^{\eps,\delta}-X_{\pi_{\delta}(s)}^{\eps,\delta}}{\eps} + \frac{\cc}{2} \gDk(x_s)[f(x_s)+g(x_s)\kappa(x_s)] \right\} ds + \int_0^t \Big\{\sigma(X_s^{\eps,\delta})-\sigma(x_s)\Big\} dW_s.
\end{aligned}
\end{multline*}

Now, using the triangle inequality, we get
\begin{multline*}
\begin{aligned}
|Z_t^{\eps,\delta}-Z_t| & \le \int_0^t \left| \frac{f(X_s^{\eps,\delta})-f(x_s)}{\eps}-Df(x_s)Z_s^{\eps,\delta} \right| ds + \int_0^t \left| Df(x_s) \right| \left| Z_s^{\eps,\delta}-Z_s \right|ds\\
&+ \int_0^t \left|\frac{[g(X_s^{\eps,\delta})-g(x_s)]\kappa(X_s^{\eps,\delta})}{\eps}-\sum_{i=1}^{m} Dg_i(x_s)Z_s^{\eps,\delta} \kappa_i(x_s) \right|ds\\
&  + \int_0^t \left |\sum_{i=1}^{m} Dg_i(x_s)(Z_s^{\eps,\delta}-Z_s) \kappa_i(x_s) \right|ds
\end{aligned}
\end{multline*}
\begin{multline*}
\begin{aligned}
& \qquad \qquad \quad + \int_0^t \left| \frac{g(x_s)[\kappa(X_s^{\eps,\delta})-\kappa(x_s)]}{\eps}-\gDk(x_s)Z_s^{\eps,\delta} \right|ds + \int_0^t \left| \gDk(x_s)(Z_s^{\eps,\delta} -Z_s) \right|ds \\
& \qquad \qquad \quad+\int_0^t \left| \frac{g(X_s^{\eps,\delta})[\kappa(X_{\pi_\delta(s)}^{\eps,\delta})-\kappa(X_s^{\eps,\delta})]}{\eps}+ \gDk(x_s)\frac{X_s^{\eps,\delta}-X_{\pi_{\delta}(s)}^{\eps,\delta}}{\eps} \right|ds \\
& \qquad \qquad \quad \thinspace + \left|\int_0^t \gDk(x_s)\frac{X_s^{\varepsilon, \delta}-X_{\pi_\delta(s)}^{\varepsilon, \delta}}{\eps}\thinspace ds - \ell(t)\right| +\left|\int_0^t \Big\{\sigma(X_s^{\eps,\delta})-\sigma(x_s)\Big\} dW_s \right|,
\end{aligned}
\end{multline*}

where $\ell(t)$ is defined in \eqref{E:ell_1-and-ell_2}. We now obtain 
\begin{multline*}
\begin{aligned}
& \sup_{0\le s \le T}|Z_s^{\eps,\delta}-Z_s|  \le \\
& \qquad \quad \sup_{0 \le t \le T}\int_0^t \left| \frac{f(X_s^{\eps,\delta})-f(x_s)}{\eps}-Df(x_s)Z_s^{\eps,\delta} \right| ds\\ 
& \qquad \thinspace + \sup_{0 \le t \le T}\int_0^t \left|\frac{[g(X_s^{\eps,\delta})-g(x_s)]\kappa(X_s^{\eps,\delta})}{\eps}-\sum_{i=1}^{m} Dg_i(x_s)Z_s^{\eps,\delta} \kappa_i(x_s) \right|ds
\end{aligned}
\end{multline*}
\begin{multline*}
\begin{aligned}
& \quad  +\int_0^T \sum_{i=1}^{m}|Dg_i(x_s)|\sup_{0\le r \le s}|Z_r^{\eps,\delta}-Z_r||\kappa_i(x_s)|\thinspace ds
 + \sup_{0\le t \le T}\int_0^t \left| \frac{g(x_s)[\kappa(X_s^{\eps,\delta})-\kappa(x_s)]}{\eps}-\gDk(x_s)Z_s^{\eps,\delta} \right|ds
\end{aligned}
\end{multline*}
\begin{multline*}
\begin{aligned}
& \qquad \quad   +\sup_{0\le t \le T}\int_0^t \left| \frac{g(X_s^{\eps,\delta})[\kappa(X_{\pi_\delta(s)}^{\eps,\delta})-\kappa(X_s^{\eps,\delta})]}{\eps}+ \gDk(x_s)\frac{X_s^{\eps,\delta}-X_{\pi_{\delta}(s)}^{\eps,\delta}}{\eps} \right|ds\\ 
& \qquad \quad + \sup_{0\le s \le T}\left|\int_0^s \gDk(x_r)\frac{X_r^{\varepsilon, \delta}-X_{\pi_\delta(r)}^{\varepsilon, \delta}}{\eps}\thinspace dr - 
\ell(s)\right|
  + \int_0^T \left\{|Df(x_s)|+|\gDk(x_s)|\right\}\sup_{0\le r \le s}|Z_r^{\eps,\delta}-Z_r|\thinspace ds\\
  & \qquad \qquad \qquad \qquad \qquad \qquad \qquad \qquad \qquad \qquad \qquad \qquad \qquad+ \sup_{0\le t \le T}\left|\int_0^t \{\sigma(X_s^{\eps,\delta})-\sigma(x_s)\} dW_s \right|.
\end{aligned}
\end{multline*}

Next, taking expectation on both sides, and then using Propositions \ref{P:MainTermApprox} through \ref{P:Noise-Term-Est}, followed by Gronwall's inequality, we get the required result.
\end{proof}


\subsection{Proof of Proposition \ref{P:MainTermApprox}}\label{S:MainTermApp}
This section is devoted to the proof of Proposition \ref{P:MainTermApprox}. Our plan proceeds as follows. 
First, in Lemma \ref{L:M_estimates}, we write the term $(1/\eps)\int_0^t \gDk(x_s)[{X_s^{\varepsilon, \delta}-X_{\pi_\delta(s)}^{\varepsilon, \delta}}]\thinspace ds$ appearing in Proposition \ref{P:MainTermApprox} as the sum $\sum_{i=1}^{4}{\sf M}_i^{\eps,\delta}(t)$. 
Next, we carefully analyze each of the terms ${\sf M}_i^{\eps,\delta}(t)$, $1 \le i \le 4$, to identify which terms contribute to the limit $\ell(t)$ and which terms are asymptotically negligible. 
To aid in this process, we frequently decompose terms into sums of simpler terms, as summarized in Figure \ref{F:Decomposition}. 
A high-level summary of this development is as follows. 
Lemmas \ref{L:l_estimates} through \ref{L:l2Andl3_Estimates} deal with ${\sf M}_1^{\eps,\delta}(t)$, with Lemma \ref{L:ell_1Approx} identifying the limiting contribution $\ell_1(t)$ (see equation \eqref{E:ell_1-and-ell_2}) of the quantity ${\sf M}_1^{\eps,\delta}(t)$.
Lemmas \ref{L:L_1_Estimate} to  \ref{L:P1AndP3Estimates} describe the asymptotic behavior of ${\sf M}_2^{\eps,\delta}(t)$, with Lemma \ref{L:ell2Approximation} identifying its contribution $\ell_2(t)$ (again, see equation \eqref{E:ell_1-and-ell_2}).
Bounds for ${\sf M}_3^{\eps,\delta}(t)$ and ${\sf M}_4^{\eps,\delta}(t)$ are provided in Lemmas \ref{L:M3_Estimate} and \ref{L:M4-Noise-Term-Small}, respectively. 
The proofs of all these lemmas are deferred to Section \ref{S:CLT-Lemmas}. We close out this section with the proof of Proposition \ref{P:MainTermApprox} obtained by putting various pieces together.

\begin{figure}[h!]
\begin{center}
\includegraphics[height=7cm,width=14cm]{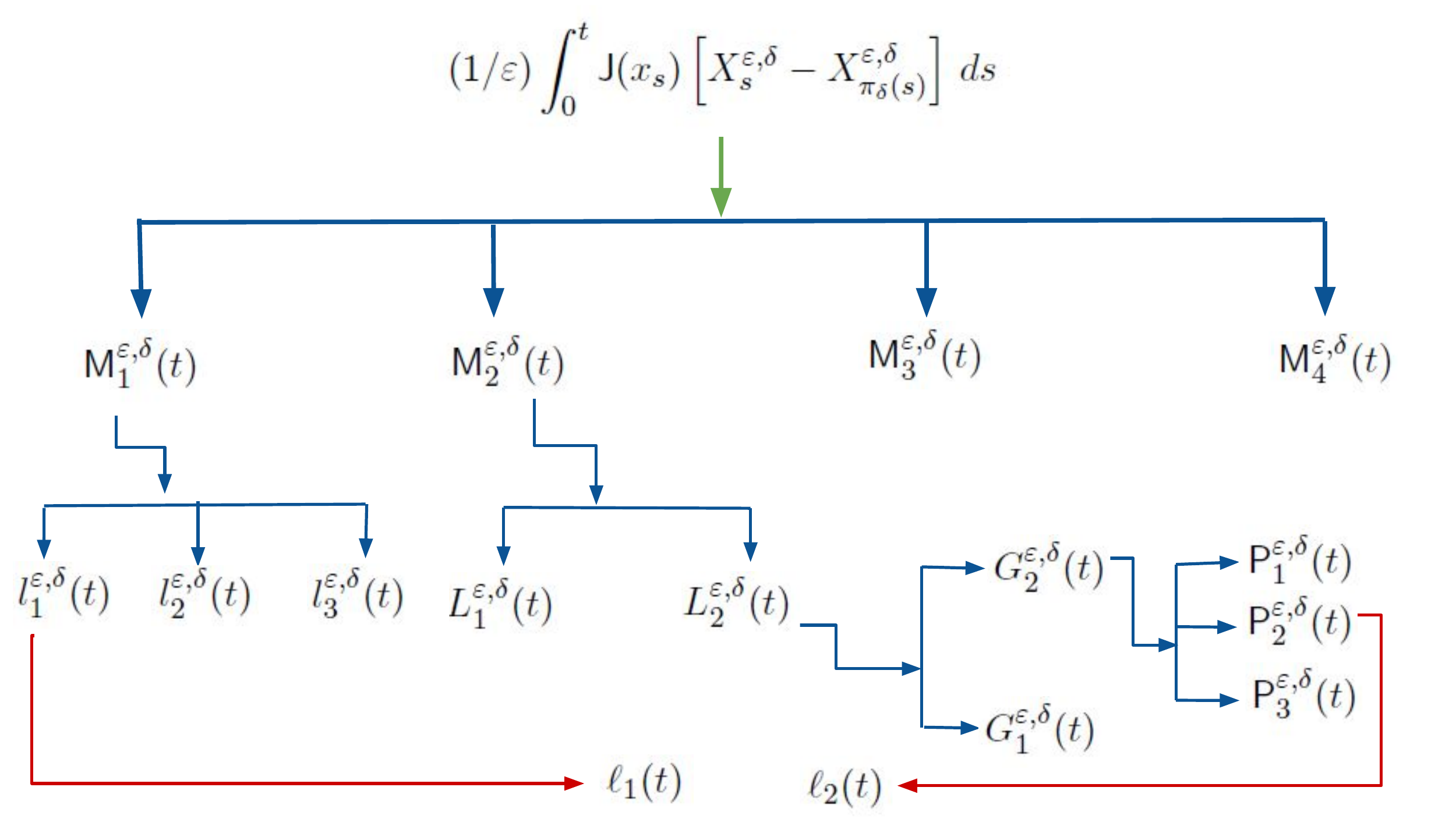}
\caption{The term $(1/\eps)\int_0^t \gDk(x_s)[{X_s^{\varepsilon, \delta}-X_{\pi_\delta(s)}^{\varepsilon, \delta}}]\thinspace ds$ is first decomposed as a sum of ${\sf M}_i^{\eps,\delta}(t)$, $1 \le i \le 4$. Some of these ${\sf M}_i^{\eps,\delta}(t)$ are further split into simpler terms (sometimes in multiple steps) as shown in the branching tree-like diagram. The only terms that survive in the limit as $\eps \searrow 0$ are $l_1^{\eps,\delta}(t)$ and $\mathsf{P}_2^{\eps,\delta}(t)$; these terms converge to $\ell_1(t)$ and $\ell_2(t)$, respectively. All the remaining terms are small and thus asymptotically negligible.} \label{F:Decomposition}
\end{center}
\end{figure}



\begin{lemma}\label{L:M_estimates}
Let $X_t^{\eps,\delta}$ be the strong solution of \eqref{E:sie}. Then, for any $\eps,\delta \in (0,1)$ and $t \in [0,T]$, we have
\begin{equation}\label{E:M1234}
\begin{aligned}
& (1/\eps)\int_0^t \gDk(x_s) \left[{X_s^{\varepsilon, \delta}-X_{\pi_\delta(s)}^{\varepsilon, \delta}}\right]\thinspace ds = \sum_{i=1}^{4} {\sfM}_i^{\eps,\delta}(t), \quad \text{where}\\
{\sf M}_1^{\eps,\delta}(t) & \triangleq (1/\eps)\int_0^t \gDk(x_s) f(x_{\pi_\delta(s)})[{s-\pi_\delta(s)}] ds, \quad
{\sf M}_2^{\eps,\delta}(t) \triangleq (1/\eps)\int_0^t \gDk(x_s)\int_{\pi_\delta(s)}^{s}{g(X_r^{\eps,\delta})\kappa(X_{\pi_\delta(r)}^{\eps,\delta})} dr \thinspace ds,\\
{\sf M}_3^{\eps,\delta}(t) & \triangleq \frac{1}{\eps}\int_0^t \gDk(x_s)\int_{\pi_\delta(s)}^{s}\{{f(X_r^{\eps,\delta})-f(X_{\pi_\delta(r)}^{\eps,\delta})}\} dr ds
+ \frac{1}{\eps}\int_0^t \gDk(x_s) [{f(X_{\pi_\delta(s)}^{\varepsilon, \delta})-f(x_{\pi_\delta(s)})}][s-\pi_\delta(s)]ds,\\
{\sf M}_4^{\eps,\delta}(t) & \triangleq \int_0^t \gDk(x_s) \int_{\pi_\delta(s)}^{s}\sigma(X_u^{\eps,\delta})dW_u \thinspace ds.
\end{aligned}
\end{equation}
\end{lemma}

Next, in Lemma \ref{L:l_estimates}, we write ${\sf M}_1^{\eps,\delta}(t)$ as a sum of terms $l_j^{\eps,\delta}(t)$, $1 \le j \le 3$. Lemma \ref{L:ell_1Approx} estimates the error between $l_1^{\eps,\delta}(t)$ and $\ell_1(t)$ (see equation \eqref{E:ell_1-and-ell_2}), showing convergence of the former to the latter. Lemma \ref{L:l2Andl3_Estimates} estimates $|l_2^{\eps,\delta}(t)|+|l_3^{\eps,\delta}(t)|$, assuring us that this term converges to zero as $\eps \searrow 0$. 
\begin{lemma}\label{L:l_estimates}
For any $\eps,\delta \in (0,1)$ and $t\ge 0$, we have
${\sf M}_1^{\eps,\delta}(t) = \sum_{i=1}^{3}  l_{i}^{\eps,\delta}(t), \quad \text{where}$
\begin{equation*}
\begin{aligned}
 l_{1}^{\eps,\delta}(t) & \triangleq \frac{1}{\eps}\int_0^t \gDk(x_{\pi_\delta(s)})f(x_{\pi_\delta(s)})[{s-\pi_\delta(s)}]\thinspace ds,\\ 
{l_{2}^{\eps,\delta}}(t) & \triangleq \frac{1}{\eps}\int_0^t g(x_s)\left\{D\kappa(x_s)-D\kappa(x_{\pi_\delta(s)})\right\}f(x_{\pi_\delta(s)})[{s-\pi_\delta(s)}]\thinspace ds,\\
 l_{3}^{\eps,\delta}(t) & \triangleq \frac{1}{\eps}\int_0^t \left\{g(x_s)-g(x_{\pi_\delta(s)})\right\}D\kappa(x_{\pi_\delta(s)})
f(x_{\pi_\delta(s)})
[{s-\pi_\delta(s)}]\thinspace ds. 
\end{aligned}
\end{equation*}
\end{lemma}


\begin{lemma}\label{L:ell_1Approx}
Let $\ell_1(t)\triangleq \frac{1}{2}\thinspace \cc \int_0^t \gDk(x_s)f(x_s)\thinspace ds$ and let $l_1^{\eps,\delta}(t)$ be as defined in Lemma \ref{L:l_estimates}. 
Then, for any fixed $T>0$, there exists a positive constant $C_{\ref{L:ell_1Approx}}$ such that for any $0< \eps < \eps_0$,
\begin{equation*}
\sup_{0\le t \le T}|l_1^{\eps,\delta}(t)-\ell_1(t)|\le \left[ \delta(\cc+1)+\varkappa(\eps)\right]C_{\ref{L:ell_1Approx}}e^{C_{\ref{L:ell_1Approx}}T},
\end{equation*}
where $\kap(\eps)$ is defined in \eqref{E:kappa} and satisfies $\lim_{\eps \searrow 0}\kap(\eps)=0$. 
\end{lemma}

\begin{lemma}\label{L:l2Andl3_Estimates}
For any fixed $T > 0$, there exists a positive constant $C_{\ref{L:l2Andl3_Estimates}}$ such that for any $0<\eps < \eps_0,$ 
\begin{equation*}
\begin{aligned}
\sup_{0\le t \le T}\left(|l_2^{\eps,\delta}(t)|+|l_3^{\eps,\delta}(t)|\right) &\le \delta(\cc+1)C_{\ref{L:l2Andl3_Estimates}}e^{C_{\ref{L:l2Andl3_Estimates}}T}.
\end{aligned} 
\end{equation*}
\end{lemma}

Our next series of lemmas comprising Lemmas \ref{L:L_1_Estimate} through \ref{L:P1AndP3Estimates} estimate the term ${\sf M}_{2}^{\eps,\delta}(t)$. Following the path charted in Figure \ref{F:Decomposition}, we start by writing 
\begin{equation}\label{E:L_1AndL_2Estimates}
\begin{aligned}
\sfM_2^{\eps,\delta}(t) &= L_1^{\eps,\delta}(t)+ L_2^{\eps,\delta}(t), \qquad \text{where}\\
L_1^{\eps,\delta}(t)&\triangleq\int_0^t \gDk(x_s)\int_{\pi_\delta(s)}^{s} \frac{g(X_r^{\eps,\delta})-g(x_{\pi_\delta(r)})}{\eps}\kappa(X_{\pi_\delta(r)}^{\eps,\delta}) \thinspace dr \thinspace ds,\\
L_2^{\eps,\delta}(t)&\triangleq\int_0^t \gDk(x_s)\int_{\pi_\delta(s)}^{s}\frac {g(x_{\pi_\delta(r)})}{\eps}\kappa(X_{\pi_\delta(r)}^{\eps,\delta}) \thinspace dr \thinspace ds.
\end{aligned}
\end{equation}
The term $L_1^{\eps,\delta}(t)$ is estimated and shown to be small in Lemma \ref{L:L_1_Estimate}, whereas to deal with $L_2^{\eps,\delta}(t)$, we further decompose it as follows:
\begin{equation}\label{E:G1-Rep}
\begin{aligned}
L_2^{\eps,\delta}(t) &\triangleq G_{1}^{\eps,\delta}(t)+G_2^{\eps,\delta}(t), \qquad \text{where}\\
G_1^{\eps,\delta}(t)&=\int_0^t \gDk(x_s)\int_{\pi_\delta(s)}^{s}g(x_{\pi_\delta(r)})\frac{\kappa(X_{\pi_\delta(r)}^{\eps,\delta})-\kappa(x_{\pi_\delta(r)})}{\eps}\thinspace dr ds \\
G_2^{\eps,\delta}(t)&=\int_0^t \gDk(x_s)\int_{\pi_\delta(s)}^{s}\frac{g(x_{\pi_\delta(r)})\kappa(x_{\pi_\delta(r)})}{\eps}\thinspace dr \thinspace ds.
\end{aligned}
\end{equation}

%

\begin{lemma}\label{L:L_1_Estimate}
Let $L_1^{\eps,\delta}(t)$ be as in \eqref{E:L_1AndL_2Estimates}. 
Then, for any $T > 0$, there exists a positive constant $C_{\ref{L:L_1_Estimate}}$ such that for any $0<\eps< \eps_0$, we have
\begin{equation*}
\BE\left[\sup_{0\le t \le T}|L_1^{\eps,\delta}(t)|\right] \le (\cc+1)(\eps+\delta)C_{\ref{L:L_1_Estimate}}e^{C_{\ref{L:L_1_Estimate}}T}.
\end{equation*}
\end{lemma}

Coming now to $L_2^{\eps,\delta}(t)$ in equation \eqref{E:G1-Rep}, we first tackle $G_1^{\eps,\delta}(t)$ in Lemma \ref{L:G1-Est}. 
\begin{lemma}\label{L:G1-Est}
Let $G_1^{\eps,\delta}(t)$ be defined in \eqref{E:G1-Rep}. Then for any fixed $T > 0,$ there exists a positive constant $C_{\ref{L:G1-Est}}$ such that for any $0<\eps< \eps_0$, we have
\begin{equation*}
\BE\left[\sup_{0\le t \le T}|G_1^{\eps,\delta}(t)| \right]\le (\cc+1)(\eps +\delta)C_{\ref{L:G1-Est}}e^{C_{\ref{L:G1-Est}}T}.
\end{equation*}
\end{lemma}

To facilitate analysis of $G_2^{\eps,\delta}(t)$, we further decompose it as the sum $\sum_{i=1}^{3} \sfP_{i}^{\eps,\delta}(t)$ in Lemma \ref{L:P_estimates}.

\begin{lemma}\label{L:P_estimates}
Let $G_2^{\eps,\delta}(t)$ be defined in \eqref{E:G1-Rep}. Then, for any $\eps,\thinspace \delta \in (0,1)$ and $t\ge 0,$ we have
\begin{equation*}
\begin{aligned}
G_2^{\eps,\delta}(t) &= \sum_{i=1}^{3} \sfP_{i}^{\eps,\delta}(t), \thinspace \text{where}\\
{\sf P}_{1}^{\eps,\delta}(t)& \triangleq \frac{1}{\eps}\int_0^t g(x_s)\left[D\kappa(x_s)-D\kappa(x_{\pi_\delta(s)})\right]g(x_{\pi_\delta(s)})\int_{\pi_\delta(s)}^{s}\kappa(x_{\pi_\delta(r)})\thinspace dr \thinspace ds,\\
{\sf P}_{2}^{\eps,\delta}(t)& \triangleq \frac{1}{\eps}\int_0^t \gDk(x_{\pi_\delta(s)})g(x_{\pi_\delta(s)})\int_{\pi_\delta(s)}^{s}\kappa(x_{\pi_\delta(r)})\thinspace dr \thinspace ds,\\
{\sf P}_{3}^{\eps,\delta}(t)& \triangleq \frac{1}{\eps}\int_0^t\left[g(x_s)-g(x_{\pi_\delta(s)})\right]D\kappa(x_{\pi_\delta(s)})g(x_{\pi_\delta(s)})\int_{\pi_\delta(s)}^{s}\kappa(x_{\pi_\delta(r)})\thinspace dr \thinspace ds.
\end{aligned}
\end{equation*}
\end{lemma}

We next analyze the terms $\sfP_{i}^{\eps,\delta}(t)$, $1 \le i \le 3$. Lemma \ref{L:ell2Approximation} estimates the error between ${\sf P}_{2}^{\eps,\delta}(t)$ and $\ell_2(t)$ (recall equation \eqref{E:ell_1-and-ell_2}), proving the convergence of the former to the latter. Next, we show in Lemma \ref{L:P1AndP3Estimates} that $\sfP_{1}^{\eps,\delta}(t)$, $\sfP_{3}^{\eps,\delta}(t)$ vanish in the limit as $\eps \searrow 0$. 

%

\begin{lemma}\label{L:ell2Approximation}
Let $\ell_2(t)\triangleq \frac{\cc}{2} \int_0^t \gDk(x_s)g(x_s)\kappa(x_s)\thinspace ds$, and suppose that ${\sf P}_2^{\eps,\delta}(t)$ is as defined in Lemma \ref{L:P_estimates}. 
Then, for any fixed $T>0$, there exists a positive constant $C_{\ref{L:ell2Approximation}}$ such that for any $0<\eps<\eps_0$, we have
\begin{equation*}
\sup_{0\le t \le T}|\sfP_2^{\eps,\delta}(t)-\ell_2(t)|\le \left[\delta(\cc+1)+{\varkappa(\eps)}\right]C_{\ref{L:ell2Approximation}}e^{C_{\ref{L:ell2Approximation}}T},
\end{equation*}
where $\kap(\eps)$, defined in \eqref{E:kappa}, satisfies $\lim_{\eps \searrow 0}\kap(\eps)=0$.
\end{lemma}

\begin{lemma}\label{L:P1AndP3Estimates}
For any $T > 0,$ there exists a positive constant $C_{\ref{L:P1AndP3Estimates}}$ such that for $0<\eps<\eps_0$, we have
\begin{equation*}
\begin{aligned}
\sup_{0\le t \le T}\left(|{\sf P}_1^{\eps,\delta}(t)|+|{\sf P}_3^{\eps,\delta}(t)| \right)  \le \delta(\cc+1)C_{\ref{L:P1AndP3Estimates}}e^{C_{\ref{L:P1AndP3Estimates}}T}. 
\end{aligned}
\end{equation*}
\end{lemma}

Finally, the next two lemmas allow us to show that $|{\sf M}_3^{\eps,\delta}(t)|$ and $|{\sf M}_4^{\eps,\delta}(t)|$ converge to zero as $\eps \searrow 0$.

\begin{lemma}{\label{L:M3_Estimate}}
Let ${\sf M}_3^{\eps,\delta}(t)$ be as defined in equation \eqref{E:M1234}. 
Then, for any $T > 0,$ there exists a positive constant $C_{\ref{L:M3_Estimate}}$ such that for $0<\eps<\eps_0$, we have
\begin{equation*}
\BE\left[\sup_{0\le t \le T}|{\sf M}_3^{\eps,\delta}(t)|\right]\le \eps(\cc+1)^2 C_{\ref{L:M3_Estimate}}e^{C_{\ref{L:M3_Estimate}}T}.
\end{equation*}
\end{lemma}

\begin{lemma}\label{L:M4-Noise-Term-Small}
Let ${\sf M}_4^{\eps,\delta}(t)$ be as defined in equation \eqref{E:M1234}. Then, for any fixed $T>0$, there exists a positive constant $C_{\ref{L:M4-Noise-Term-Small}}$ such that for any $\eps,\delta>0$
\begin{equation*}
\BE\left[\sup_{0 \le t \le T}|{\sf M}_4^{\eps,\delta}(t)|\right] \le \Big(\eps^2 +\sqrt{\delta}\Big)C_{\ref{L:M4-Noise-Term-Small}}e^{C_{\ref{L:M4-Noise-Term-Small}}T}.
\end{equation*}
\end{lemma}

As noted at the beginning of this section, the proofs of Lemmas \ref{L:M_estimates} through \ref{L:M4-Noise-Term-Small} will be given in Section \ref{S:CLT-Lemmas}. We now give the proof of Proposition \ref{P:MainTermApprox}.


\begin{proof}[Proof of Proposition \ref{P:MainTermApprox}]
By Lemma \ref{L:M_estimates}, we get $\int_0^t \gDk(x_s)\frac{X_s^{\varepsilon, \delta}-X_{\pi_\delta(s)}^{\varepsilon, \delta}}{\eps}\thinspace ds - 
\ell(t)= \sum_{i=1}^{4} {\sfM}_i^{\eps,\delta}(t) - \ell(t).$
Recalling the decomposition outlined in Figure \ref{F:Decomposition} and equation \eqref{E:ell_1-and-ell_2}, we have 
\begin{multline*}
\left|\int_0^t \gDk(x_s)\frac{X_s^{\varepsilon, \delta}-X_{\pi_\delta(s)}^{\varepsilon, \delta}}{\eps}\thinspace ds - 
\ell(t) \right|  \le \left| l_1^{\eps,\delta}(t) - \ell_1(t)\right|+\left|l_2^{\eps,\delta}(t)+l_3^{\eps,\delta}(t)\right|+\left|L_1^{\eps,\delta}(t)+G_1^{\eps,\delta}(t)\right|\\
+\left|{\sf P}_1^{\eps,\delta}(t)+{\sf P}_3^{\eps,\delta}(t)\right|+\left|{\sf P}_2^{\eps,\delta}(t)-\ell_2(t)\right|+ \left|{\sf M}_3^{\eps,\delta}(t)+{\sf M}_4^{\eps,\delta}(t)\right|.
\end{multline*}
We now use the estimates in Lemmas \ref{L:ell_1Approx}, \ref{L:l2Andl3_Estimates}, \ref{L:P1AndP3Estimates} to handle the terms $\left| l_1^{\eps,\delta}(t) - \ell_1(t)\right|,\left|l_2^{\eps,\delta}(t)+l_3^{\eps,\delta}(t)\right|,\\ \left|{\sf P}_1^{\eps,\delta}(t)+{\sf P}_3^{\eps,\delta}(t)\right|$, respectively. The term $\left|L_1^{\eps,\delta}(t)+G_1^{\eps,\delta}(t)\right|$ can be dealt with using Lemmas \ref{L:L_1_Estimate} and \ref{L:G1-Est}. 
The proof of the proposition can now be easily completed using Lemmas \ref{L:ell2Approximation}, \ref{L:M3_Estimate}, and \ref{L:M4-Noise-Term-Small} to estimate the quantities  $\left|{\sf P}_2^{\eps,\delta}(t)-\ell_2(t)\right|, \left|{\sf M}_3^{\eps,\delta}(t)\right|$, and $\left|{\sf M}_4^{\eps,\delta}(t)\right|$, respectively.

\end{proof}


\subsection{Proofs of Propositions \ref{P:f-Taylor-Approx} through \ref{P:Noise-Term-Est}}\label{S:PropositionsProofs}
We start by briefly recalling Taylor's formula with remainder for functions of several variables, since this will be used frequently in our calculations. Our notation here follows \cite{apostol1974}. For $x,y \in \BR^n$, let $L(x,y)\triangleq \{sx+(1-s)y : 0\le s \le 1\}$ denote the line segment joining $x$ and $y$. Suppose we have a function $f: \BR^n \rightarrow \BR$ whose first and second order partial derivatives exist. Given $x,t \in \BR^n$, we let $f'(x;t)$, $f''(x;t)$ denote the quantities
\begin{equation}\label{E:prime-notation}
f'(x;t) \triangleq \sum_{i=1}^{n}\frac{\partial f}{\partial x_i}(x)t_i= \langle \nabla f(x),t \rangle, \quad \text{and} \quad
f''(x;t) \triangleq \sum_{i=1}^n\sum_{j=1}^n \frac{\partial^2 f}{\partial x_i \partial x_j}(x)t_jt_i= \langle D^2f(x)t,t \rangle.
\end{equation} 
The symbol $f^{(m)}(x;t)$ can be defined in a similar way if all $m$-th order partial derivatives exist. For instance, assuming the existence of third order partial derivatives, we write
\begin{equation*}
f'''(x;t) \triangleq \sum_{i=1}^n\sum_{j=1}^n\sum_{k=1}^n \frac{\partial^3 f}{\partial x_i \partial x_j \partial x_k}(x)t_kt_jt_i.
\end{equation*}


%

\begin{proposition}(Taylor's formula)\cite{apostol1974}\label{T:Taylor-Theorem}
Let $S$ be an open subset of $\BR^n$, and suppose we have a function $f: S \rightarrow \BR$. Assume that $f$ and all its partial derivatives of order less than $m$ are differentiable at each point of $S$. If $x,y$ are two points in $S$ such that $L(x,y)\subset S,$ then there is a point $z$ on the line segment $L(x,y)$ such that 
\begin{equation*}
f(y)-f(x)= \sum_{k=1}^{m-1}\frac{1}{k!}f^{(k)}(x;y-x)+ \frac{1}{m!}f^{(m)}(z;y-x).
\end{equation*}
\end{proposition}


Before proving Propositions \ref{P:f-Taylor-Approx} through \ref{P:Noise-Term-Est}, we quickly recall the following vector and matrix norms. For any $x\in \BR^n$, $|x| \triangleq \sum_{i=1}^{n}|x_i|$ and $\|x\| \triangleq \sqrt{\sum_{i=1}^{n}|x_i|^2}$ denote the one and two norms on $\BR^n$, respectively. For $M\in \BR^{m\times n}$, $|M|$ and $\|M\|$ denote the corresponding induced matrix norms.

\begin{proof}[Proof of Proposition \ref{P:f-Taylor-Approx}] 
Using Taylor's formula, i.e., Proposition \ref{T:Taylor-Theorem} , we get
\begin{multline*}
\left|\frac{f(X_s^{\eps,\delta})-f(x_s)}{\eps}-Df(x_s) Z_s^{\eps,\delta}\right| \\
 \qquad \qquad \qquad \qquad =\left|\left(\frac{f_1(X_s^{\eps,\delta})-f_1(x_s)}{\eps}-\nabla f_1(x_s)Z_s^{\eps,\delta},...,\frac{f_n(X_s^{\eps,\delta})-f_n(x_s)}{\eps}-\nabla f_n(x_s)Z_s^{\eps,\delta}\right)^\top \right|\\
\qquad \qquad \qquad \qquad =\frac{1}{2\eps}\left|\left(\langle D^2f_1(z_1)(X_s^{\eps,\delta}-x_s),X_s^{\eps,\delta}-x_s \rangle,...,\langle D^2f_n(z_n)(X_s^{\eps,\delta}-x_s),X_s^{\eps,\delta}-x_s \rangle   \right)^\top\right|\\
\qquad \qquad  \qquad=\frac{1}{2\eps}\sum_{i=1}^{n}|\langle D^2f_i(z_i)(X_s^{\eps,\delta}-x_s),X_s^{\eps,\delta}-x_s \rangle|, 
\end{multline*}
where $z_i$, $1 \le i \le n$, are points lying on the line segment joining $x_s$ and $X_s^{\eps,\delta}$. Now, using the Cauchy-Schwarz inequality and the submultiplicative property of the matrix norms \cite{Hes_LST}, we get
\begin{equation*}
\begin{aligned}
\left|\frac{f(X_s^{\eps,\delta})-f(x_s)}{\eps}-Df(x_s) Z_s^{\eps,\delta}\right| &\le \sum_{i=1}^{n}\frac{\|D^2f_i(z_i)\|\|X_s^{\eps,\delta}-x_s\|^2}{2\eps}
\le \frac{C}{\eps} \sum_{i=1}^{n}|D^2f_i(z_i)|\left[\sup_{0\le s \le T}|X_s^{\eps,\delta}-x_s|^2\right],
\end{aligned}
\end{equation*}
where the last inequality uses the equivalence property of matrix norms \cite{Hes_LST}. Using the boundedness of second order partial derivatives of $f$ from Assumption \ref{A:Boundedness}, we get
\begin{equation*}
\begin{aligned}
\sup_{0\le t \le T}\int_0^t\left|\frac{f(X_s^{\eps,\delta})-f(x_s)}{\eps}-Df(x_s) Z_s^{\eps,\delta}\right|\thinspace ds & \le \frac{T {C}}{2\eps} \left[\sup_{0\le s \le T}|X_s^{\eps,\delta}-x_s|^2\right].
\end{aligned}
\end{equation*}
We now get the required bound by Theorem \ref{T:LLN-II}. The second part of the lemma is obtained by similar arguments,  using the  boundedness of $g$ and the second order partial derivatives of $\kappa$ (Assumption \ref{A:Boundedness}).
\end{proof}

\begin{proof}[Proof of Proposition \ref{P:Add-Term-Approx}]
Let $g=[g_1,g_2,...,g_m],$ where $g_i \in \BR^n$ for each $i=1,...,m$.\footnote{Thus, $g$ is an $n\times m$ matrix with columns $g_i \in \BR^n$, $1 \le i \le m$.} Now, 
\begin{equation}\label{E:AdditTermRep}
{[g(X_s^{\eps,\delta})-g(x_s)]\kappa(X_s^{\eps,\delta})}= {[g(X_s^{\eps,\delta})-g(x_s)][\kappa(X_s^{\eps,\delta})-\kappa(x_s)]}+[g(X_s^{\eps,\delta})-g(x_s)]\kappa(x_s)\triangleq I_1+I_2.
\end{equation}
For $I_1,$ using the Lipschitz continuity of $g$ and $\kappa$ from Assumption \ref{A:Smooth-LinearGrow}, we get
$|I_1|\le C |X_s^{\eps,\delta}-x_s|^2.$
Now, $I_2= \sum_{i=1}^{m}[g_i(X_s^{\eps,\delta})-g_i(x_s)]\kappa_i(x_s)$. For each $i=1,...,m,$ using Taylor's formula, i.e., Proposition \ref{T:Taylor-Theorem}, we get
\begin{multline*}
g_i(X_s^{\eps,\delta})-g_i(x_s)=Dg_i(x_s)(X_s^{\eps,\delta}-x_s)\\+\frac{1}{2}\left(\left\langle D^2g_{i,1}(z_1)(X_s^{\eps,\delta}-x_s),(X_s^{\eps,\delta}-x_s) \right \rangle,...,\left\langle D^2g_{i,n}(z_{n})(X_s^{\eps,\delta}-x_s),(X_s^{\eps,\delta}-x_s) \right\rangle \right)^\top,
\end{multline*}
where for each fixed $i,~D^2
g_{i,j}$ represents the Hessian matrix of the $j$-th component of $g_i$, and $z_j, j=1,...,n$, are points lying on the line segment joining $x_s$ and $X_s^{\eps,\delta}.$ Therefore,
\begin{multline}\label{E:AdditTermTaylorRep}
I_2= \sum_{i=1}^{m}\left\{g_i(X_s^{\eps,\delta})-g_i(x_s)\right\}\kappa_i(x_s)=\sum_{i=1}^{m}Dg_i(x_s)(X_s^{\eps,\delta}-x_s)\kappa_i(x_s)\\
+\sum_{i=1}^{m}\frac{1}{2}\left(\left\langle D^2g_{i,1}(z_1)(X_s^{\eps,\delta}-x_s),(X_s^{\eps,\delta}-x_s) \right \rangle,...,\left\langle D^2g_{i,n}(z_{n})(X_s^{\eps,\delta}-x_s),(X_s^{\eps,\delta}-x_s) \right\rangle \right)^\top \kappa_i(x_s).
\end{multline}
From \eqref{E:AdditTermRep} and \eqref{E:AdditTermTaylorRep}, we have
\begin{multline*}
\begin{aligned}
&\left|\frac{[g(X_s^{\eps,\delta}) -g(x_s)]\kappa(X_s^{\eps,\delta})}{\eps} - \sum_{i=1}^{m} Dg_i(x_s)Z_s^{\eps,\delta}\kappa_i(x_s) \right|  \le \frac{|I_1|}{\eps}\\
&+ \frac{1}{2\eps}\left|\sum_{i=1}^{m}\left(\left\langle D^2g_{i,1}(z_1)(X_s^{\eps,\delta}-x_s),(X_s^{\eps,\delta}-x_s) \right\rangle,...,\left\langle D^2g_{i,n}(z_{n})(X_s^{\eps,\delta}-x_s),(X_s^{\eps,\delta}-x_s) \right\rangle \right)^\top \kappa_i(x_s)\right|\\
&\le \frac{|I_1|}{\eps}\\
& + \frac{1}{2 \eps} \sum_{i=1}^m \left|\left(\left\langle D^2g_{i,1}(z_1)(X_s^{\eps,\delta}-x_s),(X_s^{\eps,\delta}-x_s) \right\rangle,...,\left\langle D^2g_{i,n}(z_{n})(X_s^{\eps,\delta}-x_s),(X_s^{\eps,\delta}-x_s) \right\rangle \right)^\top \kappa_i(x_s)\right|\\
&= \frac{|I_1|}{\eps} + \frac{1}{2\eps} \sum_{i=1}^{m}\sum_{j=1}^{n}\left|\left(\left\langle D^2g_{i,j}(z_j)(X_s^{\eps,\delta}-x_s),(X_s^{\eps,\delta}-x_s) \right\rangle \right)\kappa_i(x_s)\right|.
\end{aligned}
\end{multline*}
Now, using the Cauchy-Schwarz inequality and the submultiplicative  property of matrix norms \cite{Hes_LST}, we have 
\begin{multline*}
\left|\frac{[g(X_s^{\eps,\delta}) -g(x_s)]\kappa(X_s^{\eps,\delta})}{\eps} - \sum_{i=1}^{m} Dg_i(x_s)Z_s^{\eps,\delta}\kappa_i(x_s) \right|\\
\le \frac{|I_1|}{\eps} + \frac{1}{2\eps} \sum_{i=1}^{m}|\kappa_i(x_s)|\sum_{j=1}^{n}\|D^2g_{i,j}(z_j)\|\|X_s^{\eps,\delta}-x_s\|^2\\
\le \frac{C}{\eps} \left[ \sup_{0\le s \le T}|X_s^{\eps,\delta}-x_s|^2 \right] + \frac{C}{2\eps} \sum_{i=1}^{m}|\kappa_i(x_s)|\sum_{j=1}^{n}|D^2g_{i,j}(z_j)|\left[\sup_{0\le s \le T}|X_s^{\eps,\delta}-x_s|^2\right].
\end{multline*}
Now, integrating both sides, then using the boundedness of the second order partial derivatives of $g$ (Assumption \ref{A:Boundedness}), the linear growth assumption on $\kappa$ (Assumption \ref{A:Smooth-LinearGrow}), and recalling Theorem \ref{T:LLN-II}, we get the required estimate.
\end{proof}

Before proving Proposition \ref{P:Sampling-Term-Approx}, we state and prove two Lemmas \ref{L:Kappa-Derivative-Est} and \ref{L:Highr-Power-Der-lemma}. These lemmas will be used in the proof of Proposition \ref{P:Sampling-Term-Approx} below, and also in Section \ref{S:CLT-Lemmas} when we furnish the proofs of Lemmas \ref{L:l2Andl3_Estimates} and \ref{L:P1AndP3Estimates} stated earlier.

\begin{lemma}\label{L:Kappa-Derivative-Est}
Let $X_t^{\eps,\delta}$ and $x_t$ be the strong solution and solution of \eqref{E:sie} and \eqref{E:nonlin-sys}, respectively. For $\kappa: \BR^n \rightarrow \BR^m,$ set $h\triangleq D\kappa : \BR^n \rightarrow \BR^{m \times n}.$ Then, there exists a positive constant $C_{\ref{L:Kappa-Derivative-Est}}$ such that for $t\in [0,T]$, we have
\begin{equation}\label{E:Kappa-Derivative-Est}
\begin{aligned}
\left|h(x_t)-h(X_{\pi_{\delta}(t)}^{\eps,\delta})\right| &\triangleq \left|D\kappa(x_t)-D\kappa(X_{\pi_{\delta}(t)}^{\eps,\delta})\right|\le C_{\ref{L:Kappa-Derivative-Est}}\left(|X_t^{\eps,\delta}-x_t|+\left|X_t^{\eps,\delta}-X_{\pi_{\delta}(t)}^{\eps,\delta}\right|\right),\\
|h(x_t)-h(x_{\pi_{\delta}(t)})| &\triangleq|D\kappa(x_t)-D\kappa(x_{\pi_{\delta}(t)})|\le \delta C_{\ref{L:Kappa-Derivative-Est}}e^{C_{\ref{L:Kappa-Derivative-Est}}T}.
\end{aligned}
\end{equation}
\end{lemma}

\begin{lemma}\label{L:Highr-Power-Der-lemma}
Let $X_t^{\eps,\delta}$ and $x_t$ be the strong solution and solution of \eqref{E:sie} and \eqref{E:nonlin-sys}, respectively. For $\kappa: \BR^n \rightarrow \BR^m,$ set $h\triangleq D\kappa : \BR^n \rightarrow \BR^{m \times n}.$ Then, there exists a positive constant $C_{\ref{L:Highr-Power-Der-lemma}}$ such that 
\begin{equation*}
\begin{aligned}
\BE\left[\sup_{0\le t \le T}\int_0^t\left|\left[h(x_s)-h(X_{\pi_{\delta}(s)}^{\eps,\delta})\right]\left(X_s^{\eps,\delta}-X_{{\pi_\delta(s)}}^{\eps,\delta}\right)\right| ds\right]& \le (\eps^2+\delta^2+\eps \delta)C_{\ref{L:Highr-Power-Der-lemma}}e^{C_{\ref{L:Highr-Power-Der-lemma}}T}, \\
 \BE\left[\sup_{0 \le t \le T} \int_0^t \left|X_s^{\eps,\delta}-X_{\pi_{\delta}(s)}^{\eps,\delta}\right|^2 ds \right] & \le (\eps^2+\delta^2)C_{\ref{L:Highr-Power-Der-lemma}}e^{C_{\ref{L:Highr-Power-Der-lemma}}T}.
\end{aligned}
\end{equation*}
\end{lemma}

\begin{proof}[Proof of Lemma \ref{L:Kappa-Derivative-Est}]
Set $h\triangleq D\kappa : \BR^n \rightarrow \BR^{m \times n}$ and write $h=(h_{ij}),$ where each $h_{ij}$ is a real-valued function. Note here that the entries of $h$ are the first order partial derivatives of $\kappa.$ Using Proposition \ref{T:Taylor-Theorem} (Taylor's formula)  for each entry of $h$, we have
\begin{equation}\label{E:Taylor-Deri-kappa}
h_{ij}(X_t^{\eps,\delta})-h_{ij}(x_t)=\nabla h_{ij}(z_{ij}).(X_t^{\eps,\delta}-x_t),
\end{equation}
for each $i=1,...,m$ and $j=1,...,n.$ Here, $z_{ij}\in \BR^n$ are  points lying on the line segment joining $x_t$ and $X_t^{\eps,\delta}.$ Further, using the Cauchy-Schwarz inequality and the boundedness of the \textit{second} order partial derivatives of $\kappa$ (Assumption \ref{A:Boundedness}) in \eqref{E:Taylor-Deri-kappa}, we get
\begin{equation}\label{E:Deri-Without-Sampling}
\begin{aligned}
|h(X_t^{\eps,\delta})-h(x_t)|=\max_{1 \le j \le n}\sum_{i=1}^{m}\left|h_{ij}(X_t^{\eps,\delta})-h_{ij}(x_t)\right| &\le \max_{1 \le j \le n}\sum_{i=1}^{m} \|\nabla h_{ij}(z_{ij})\|\|X_t^{\eps,\delta}-x_t\|\\
 & \le C|X_t^{\eps,\delta}-x_t|.
\end{aligned}
\end{equation}
Similarly,
\begin{equation}\label{E:Deri-With-Sampling}
\left|h(X_t^{\eps,\delta})-h(X_{\pi_{\delta}(t)}^{\eps,\delta})\right|\le C\left|X_t^{\eps,\delta}-X_{\pi_{\delta}(t)}^{\eps,\delta}\right|.
\end{equation}
Now, using the triangle inequality, we get
$\left|h(x_t)-h(X_{\pi_{\delta}(t)}^{\eps,\delta})\right|\le |h(X_t^{\eps,\delta})-h(x_t)|+\left|h(X_t^{\eps,\delta})-h(X_{\pi_\delta(t)}^{\eps,\delta})\right|.$
The first line in equation \eqref{E:Kappa-Derivative-Est} is now obtained by combining \eqref{E:Deri-Without-Sampling} and \eqref{E:Deri-With-Sampling}. For the second part of the lemma, we write
\begin{equation*}
|h(x_{\pi_{\delta}(t)})-h(x_t)|=\max_{1 \le j \le n}\sum_{i=1}^{m}|h_{ij}(x_{\pi_{\delta}(t)})-h_{ij}(x_t)| \le \max_{1 \le j \le n}\sum_{i=1}^{m} \|\nabla h_{ij}(z_{ij})\|\|x_{\pi_{\delta}(t)}-x_t\|.
\end{equation*}
Now recalling Lemma \ref{L:Sampling-Difference}, we obtain the second line in equation \eqref{E:Kappa-Derivative-Est}.
\end{proof}


\begin{proof}[Proof of Lemma \ref{L:Highr-Power-Der-lemma}]
Using the submultiplicative property of matrix norm and Lemma \ref{L:Kappa-Derivative-Est}, we have 
\begin{multline}\label{E:Higher-Power-Der}
\left|\left[h(x_s)-h(X_{\pi_{\delta}(s)}^{\eps,\delta})\right]\left(X_s^{\eps,\delta}-X_{{\pi_\delta(s)}}^{\eps,\delta}\right)\right|  \le C\left(|X_s^{\eps,\delta}-x_s|+\left|X_s^{\eps,\delta}-X_{\pi_{\delta}(s)}^{\eps,\delta}\right|\right)\left|X_s^{\eps,\delta}-X_{{\pi_\delta(s)}}^{\eps,\delta}\right|\\
\qquad \qquad \qquad \qquad \qquad \qquad \qquad \qquad \quad \le C \left(|X_s^{\eps,\delta}-x_s|\left|X_s^{\eps,\delta}-X_{{\pi_\delta(s)}}^{\eps,\delta}\right|+\left|X_s^{\eps,\delta}-X_{\pi_{\delta}(s)}^{\eps,\delta}\right|^2\right)\\
\qquad \qquad \qquad \qquad \qquad \qquad \quad  \le C\left(|X_s^{\eps,\delta}-x_s|^2+ |X_s^{\eps,\delta}-x_s|\left|x_s-X_{\pi_{\delta}(s)}^{\eps,\delta}\right|+\left|X_s^{\eps,\delta}-X_{\pi_{\delta}(s)}^{\eps,\delta}\right|^2\right)\\
\qquad \qquad \qquad \qquad \le C\left(\sup_{0 \le s \le T}\left|X_s^{\eps,\delta}-x_s\right|^2+ |X_s^{\eps,\delta}-x_s||x_s-x_{\pi_{\delta}(s)}|+\left|X_s^{\eps,\delta}-X_{\pi_{\delta}(s)}^{\eps,\delta}\right|^2\right).
\end{multline}
The third line in the above equation is obtained by adding and subtracting $x_s$ in the middle term of the second line, i.e., $\left|X_s^{\eps,\delta}-X_{{\pi_\delta(s)}}^{\eps,\delta}\right|.$
To deal with the last term of the above equation, first we add and subtract $x_s$ and then $x_{\pi_{\delta}(s)}$ to get
$\left|X_s^{\eps,\delta}-X_{\pi_{\delta}(s)}^{\eps,\delta}\right|^2  \le C\left[ \sup_{0\le s \le T}|X_s^{\eps,\delta}-x_s|^2+ \left|x_s-x_{\pi_\delta(s)}\right|^2 \right].$\\
Hence, from \eqref{E:Higher-Power-Der}, we obtain
\begin{multline*}
\left|\left[h(x_s)-h(X_{\pi_{\delta}(s)}^{\eps,\delta})\right]\left(X_s^{\eps,\delta}-X_{{\pi_\delta(s)}}^{\eps,\delta}\right)\right| \le \\
 C \left[\sup_{0 \le s \le T}\left|X_s^{\eps,\delta}-x_s \right|^2+ |X_s^{\eps,\delta}-x_s||x_s-x_{\pi_{\delta}(s)}|+\left|x_s-x_{\pi_\delta(s)}\right|^2 \right].
\end{multline*}
Finally, using Theorem \ref{T:LLN-II} and Lemma \ref{L:Sampling-Difference}, we get
\begin{equation*}
\begin{aligned}
\BE\left[\sup_{0\le t \le T}\int_0^t\left|\left[h(x_s)-h(X_{\pi_{\delta}(s)}^{\eps,\delta})\right]\left(X_s^{\eps,\delta}-X_{{\pi_\delta(s)}}^{\eps,\delta}\right)\right| ds\right] & \le (\eps^2\sqrt{n}+\delta^2+\eps \delta)T^2Ce^{CT} \qquad \text{and}\\
\BE\left[\int_0^t \left|X_s^{\eps,\delta}-X_{\pi_{\delta}(s)}^{\eps,\delta}\right|^2 ds \right] &\le (\eps^2\sqrt{n}+\delta^2)T^2Ce^{CT}.
\end{aligned}
\end{equation*}
\end{proof}


\begin{proof}[Proof of Proposition \ref{P:Sampling-Term-Approx}]
Here,
\begin{multline}\label{E:SamplingTermRep}
g(X_s^{\eps,\delta})\left[\kappa(X_s^{\eps,\delta})-\kappa(X_{\pi_\delta(s)}^{\eps,\delta})\right] = \\ \underbrace{\left[g(X_s^{\eps,\delta})-g(x_s)\right]\left\{\kappa(X_s^{\eps,\delta})-\kappa(X_{\pi_\delta(s)}^{\eps,\delta})\right\}}_{I_1} + \underbrace{g(x_s)\left[\kappa(X_s^{\eps,\delta})-\kappa(X_{\pi_\delta(s)}^{\eps,\delta})\right]}_{I_2}.
\end{multline}
For $I_1,$ using the Lipschitz continuity of $g$ and $\kappa$ from Assumption \ref{A:Smooth-LinearGrow}, we get
\begin{equation*}
\begin{aligned}
|I_1| \le C |X_s^{\eps,\delta}-x_s|\left|X_{\pi_\delta(s)}^{\eps,\delta}-X_s^{\eps,\delta}\right| &\le C |X_s^{\eps,\delta}-x_s|\left(|X_s^{\eps,\delta}-x_s|+\left|X_{\pi_\delta(s)}^{\eps,\delta}-x_s \right|\right)\\
& \le C\left[\sup_{0\le s \le T}\left|X_s^{\eps,\delta}-x_s\right|^2 \right]+ C\left[\sup_{0\le s \le T}\left|X_s^{\eps,\delta}-x_s \right| \left|x_s-x_{\pi_\delta(s)}\right|\right],
\end{aligned}
\end{equation*}
which yields
\begin{equation}\label{E:Approx-Sampling-first-term}
\BE\left[\sup_{0\le s \le T}|I_1|\right]\le C \BE\left[\sup_{0\le s \le T}\left|X_s^{\eps,\delta}-x_s\right|^2 \right]+C \left(\sup_{0\le s \le T}\left|x_s-x_{\pi_\delta(s)}\right|\right)\BE\left[\sup_{0\le s \le T}|X_s^{\eps,\delta}-x_s| \right].
\end{equation}
For $I_2$, using Proposition \ref{T:Taylor-Theorem} (Taylor's formula), we get
\begin{multline*}
\kappa(X_s^{\eps,\delta})=\kappa(X_{\pi_\delta(s)}^{\eps,\delta})+ D\kappa(X_{\pi_\delta(s)}^{\eps,\delta})\left(X_s^{\eps,\delta}-X_{\pi_\delta(s)}^{\eps,\delta}\right)\\
+\frac{1}{2}\left(\left\langle D^2\kappa_1(z_1)\left(X_s^{\eps,\delta}-X_{\pi_\delta(s)}^{\eps,\delta}\right),\left(X_s^{\eps,\delta}-X_{\pi_\delta(s)}^{\eps,\delta}\right) \right\rangle,\dots, \right.\\
\left. \left\langle D^2\kappa_m(z_m)\left(X_s^{\eps,\delta}-X_{\pi_\delta(s)}^{\eps,\delta}\right),\left(X_s^{\eps,\delta}-X_{\pi_\delta(s)}^{\eps,\delta}\right) \right \rangle \right)^\top,
 \end{multline*}
where $z_i$, $1 \le i \le m$, are points on the line segment joining $X_s^{\eps,\delta}$ and $X_{\pi_{\delta}(s)}^{\eps,\delta}$. 
Recalling the notation $f''$ in \eqref{E:prime-notation}, note that the above equation is simply 
 \begin{multline*}
 \kappa(X_s^{\eps,\delta})=\kappa(X_{\pi_\delta(s)}^{\eps,\delta})+ D\kappa(X_{\pi_\delta(s)}^{\eps,\delta})\left(X_s^{\eps,\delta}-X_{\pi_\delta(s)}^{\eps,\delta}\right)\\
 + \frac{1}{2}\left( \kappa_1''\left(z_1;\left\{X_s^{\eps,\delta}-X_{\pi_\delta(s)}^{\eps,\delta}\right\}\right),...,\kappa_m''\left(z_m;\left\{X_s^{\eps,\delta}-X_{\pi_\delta(s)}^{\eps,\delta}\right\}\right) \right)^\top.
\end{multline*}
Writing $D\kappa(X_{\pi_\delta(s)}^{\eps,\delta})$ as $\left(D\kappa(X_{\pi_\delta(s)}^{\eps,\delta})-D\kappa(x_s)+D\kappa(x_s)\right)$ in the right hand side of the above equation, we get
\begin{multline}\label{E:Kappa-Taylor-Rep}
\kappa(X_s^{\eps,\delta})=\kappa(X_{\pi_\delta(s)}^{\eps,\delta})+ \left[D\kappa(X_{\pi_\delta(s)}^{\eps,\delta})-D\kappa(x_s)\right]\left(X_s^{\eps,\delta}-X_{\pi_\delta(s)}^{\eps,\delta}\right)
+D\kappa(x_s)\left(X_s^{\eps,\delta}-X_{\pi_{\delta}(s)}^{\eps,\delta}\right)\\
+\frac{1}{2}\left(\left\langle D^2\kappa_1(z_1)\left(X_s^{\eps,\delta}-X_{\pi_\delta(s)}^{\eps,\delta}\right),\left(X_s^{\eps,\delta}-X_{\pi_\delta(s)}^{\eps,\delta}\right) \right \rangle,\dots, \right.\\
\left. \left\langle D^2\kappa_m(z_m)\left(X_s^{\eps,\delta}-X_{\pi_\delta(s)}^{\eps,\delta}\right),\left(X_s^{\eps,\delta}-X_{\pi_\delta(s)}^{\eps,\delta}\right) \right \rangle \right)^\top,
\end{multline}
Substituting \eqref{E:Kappa-Taylor-Rep} in the expression for $I_2$, and recalling that $\gDk(x_s) \triangleq g(x_s)D\kappa(x_s),$ yields
\begin{multline}\label{E:Sampling-App-I2}
I_2  \triangleq g(x_s)\left[\kappa(X_s^{\eps,\delta})-\kappa(X_{\pi_\delta(s)}^{\eps,\delta})\right]=g(x_s)\left[D\kappa(X_{\pi_\delta(s)}^{\eps,\delta})-D\kappa(x_s)\right]\left(X_s^{\eps,\delta}-X_{\pi_\delta(s)}^{\eps,\delta}\right) \\
+\gDk(x_s)\left(X_s^{\eps,\delta}-X_{\pi_\delta(s)}^{\eps,\delta}\right)
 + \frac{g(x_s)}{2}\left(\left\langle D^2\kappa_1(z_1)\left(X_s^{\eps,\delta}-X_{\pi_\delta(s)}^{\eps,\delta}\right),\left(X_s^{\eps,\delta}-X_{\pi_\delta(s)}^{\eps,\delta}\right) \right \rangle,\dots, \right.\\
\left.  \left\langle D^2\kappa_m(z_m)\left(X_s^{\eps,\delta}-X_{\pi_\delta(s)}^{\eps,\delta}\right),\left(X_s^{\eps,\delta}-X_{\pi_\delta(s)}^{\eps,\delta}\right) \right \rangle \right)^\top.
\end{multline}
We now use the boundedness of $g$ and the second order partial derivatives of $\kappa$ (Assumption \ref{A:Boundedness}). From \eqref{E:SamplingTermRep} and \eqref{E:Sampling-App-I2} and using the definition of $1$-norm, we get
\begin{multline*}
\left|\frac{g(X_s^{\eps,\delta})\left[\kappa(X_s^{\eps,\delta})-\kappa(X_{\pi_\delta(s)}^{\eps,\delta})\right]}{\eps}-\gDk(x_s)\frac{X_s^{\varepsilon, \delta}-X_{\pi_\delta(s)}^{\varepsilon, \delta}}{\eps}\right|\\  \le \frac{|I_1|}{\eps} 
+ \frac{1}{\eps}\left|g(x_s)\left[D\kappa(x_s)-D\kappa(X_{\pi_{\delta}(s)}^{\eps,\delta})\right]\left(X_s^{\eps,\delta}-X_{{\pi_\delta(s)}}^{\eps,\delta}\right)\right|\\
\qquad \qquad \qquad \qquad \qquad \qquad \qquad \quad \quad + \frac{|g(x_s)|}{2\eps}\sum_{i=1}^{m}\left|\left\langle D^2\kappa_i(z_i)\left(X_s^{\eps,\delta}-X_{\pi_\delta(s)}^{\eps,\delta}\right),\left(X_s^{\eps,\delta}-X_{\pi_\delta(s)}^{\eps,\delta}\right)\right \rangle\right|\\
 \le \frac{|I_1|}{\eps} + \frac{C}{\eps}\left|\left(D\kappa(x_s)-D\kappa(X_{\pi_{\delta}(s)}^{\eps,\delta})\right)\left(X_s^{\eps,\delta}-X_{{\pi_\delta(s)}}^{\eps,\delta}\right)\right| \\
\qquad \qquad \qquad \qquad \qquad \qquad \qquad \qquad \qquad \qquad \qquad \qquad \qquad \qquad + \frac{C}{\eps}\sum_{i=1}^{m}|D^2 \kappa_i(z_i)|\left|X_s^{\eps,\delta}-X_{\pi_{\delta}(s)}^{\eps,\delta}\right|^2 \\
 \qquad \qquad \le \frac{|I_1|}{\eps} + \frac{C}{\eps}\left|\left(D\kappa(x_s)-D\kappa(X_{\pi_{\delta}(s)}^{\eps,\delta})\right)\left(X_s^{\eps,\delta}-X_{{\pi_\delta(s)}}^{\eps,\delta}\right)\right|
  + \frac{C}{\eps} \left|X_s^{\eps,\delta}-X_{\pi_{\delta}(s)}^{\eps,\delta}\right|^2.
\end{multline*}
Using Theorem \ref{T:LLN-II} and Lemma \ref{L:Sampling-Difference} in \eqref{E:Approx-Sampling-first-term}, and Lemma \ref{L:Highr-Power-Der-lemma} in the equation above, we easily get the claimed result.
%
%
\end{proof}


\begin{proof}[Proof of Proposition \ref{P:Noise-Term-Est}]
Letting $|\cdot|$ represent the one norm, and denoting the columns of the matrix $\sigma$ by $\sigma_i\in \BR^n$, $1 \le i \le n$, we get
\begin{multline*}
\left|\int_0^t \{\sigma(X_s^{\eps,\delta})-\sigma(x_s)\}dW_s\right| = \left|\sum_{i=1}^{n}\int_0^t \{\sigma_i(X_s^{\eps,\delta})-\sigma_i(x_s)\}dW_s^i\right| \le \sum_{i=1}^{n}  \left|\int_0^t \{\sigma_i(X_s^{\eps,\delta})-\sigma_i(x_s)\}dW_s^i\right|\\
=\sum_{i=1}^{n}\sum_{j=1}^{n}  \left|\int_0^t \{\sigma_{ji}(X_s^{\eps,\delta})-\sigma_{ji}(x_s)\}dW_s^i\right|.
\end{multline*}
Therefore,
\begin{equation*}
\BE\left[\sup_{0 \le t \le T}\left|\int_0^t \{\sigma(X_s^{\eps,\delta})-\sigma(x_s)\}dW_s\right|\right] \le \sum_{i=1}^{n}\sum_{j=1}^{n} \BE \left[ \sup_{0 \le t \le T} \left|\int_0^t \{\sigma_{ji}(X_s^{\eps,\delta})-\sigma_{ji}(x_s)\}dW_s^i\right|\right].
\end{equation*}
Now, using the Burkholder-Davis-Gundy inequalities for the right hand side of the above equation, followed by Jensen's inequality for concave functions, we get
\begin{equation*}
\begin{aligned}
\BE\left[\sup_{0 \le t \le T}\left|\int_0^t \{\sigma(X_s^{\eps,\delta})-\sigma(x_s)\}dW_s\right|\right] & \le C\sum_{i=1}^{n}\sum_{j=1}^{n} \BE \left[ \sqrt{\int_0^T \left\{\sigma_{ji}(X_s^{\eps,\delta})-\sigma_{ji}(x_s)\right\}^2 ds} \right] \\
 & \le C\sum_{i=1}^{n}\sum_{j=1}^{n}\left[\BE \left(\int_0^T \left\{\sigma_{ji}(X_s^{\eps,\delta})-\sigma_{ji}(x_s)\right\}^2 ds \right) \right]^{\frac{1}{2}}.
\end{aligned}
\end{equation*}
Using Lipschitz continuity of $\sigma$ (Assumption \ref{A:Smooth-LinearGrow}), we next have
\begin{equation*}
\BE\left[\sup_{0 \le t \le T}\left|\int_0^t \{\sigma(X_s^{\eps,\delta})-\sigma(x_s)\}dW_s\right|\right] \le C \left[ \BE \left( \int_0^T |X_s^{\eps,\delta}-x_s|^2 ds\right) \right]^{1/2}.
\end{equation*}
Now, using Theorem \ref{T:LLN-II}, we get the required estimate.
\end{proof}

\section{Proofs of Lemmas \ref{L:M_estimates} through \ref{L:M4-Noise-Term-Small}}\label{S:CLT-Lemmas}

\begin{proof}[Proof of Lemma \ref{L:M_estimates}]
For $s\in [0,t],$ exploiting the integral representation for $X_s^{\eps,\delta}$ from \eqref{E:sie}, we get
\begin{equation*}
\begin{aligned}
\frac{X_s^{\eps,\delta}-X^{\eps,\delta}_{\pi_\delta(s)}}{\eps}&=\int_{\pi_\delta(s)}^{s}\frac{f(X_r^{\eps,\delta})+g(X_r^{\eps,\delta})\kappa(X_{\pi_\delta(r)}^{\eps,\delta})}{\eps}\thinspace dr + \int_{\pi_\delta(s)}^{s}\sigma(X_u^{\eps,\delta})dW_u\\
&= \int_{\pi_\delta(s)}^{s} \frac{f(X_r^{\eps,\delta})-f(X_{\pi_\delta(r)}^{\eps,\delta})}{\eps} \thinspace dr + \int_{\pi_\delta(s)}^{s} \frac{f(X_{\pi_\delta(r)}^{\eps,\delta})+g(X_r^{\eps,\delta})\kappa(X_{\pi_\delta(r)}^{\eps,\delta})}{\eps}\thinspace dr \\
& \qquad \qquad \qquad \qquad \qquad \qquad \qquad \qquad \qquad \qquad \qquad \qquad \qquad \quad+ \int_{\pi_\delta(s)}^{s}\sigma(X_u^{\eps,\delta})dW_u.
\end{aligned}
\end{equation*}
Writing $f(X_{\pi_\delta(r)}^{\eps,\delta})+g(X_r^{\eps,\delta})\kappa(X_{\pi_\delta(r)}^{\eps,\delta})$ in the middle term of the above equation as\\ $f(X_{\pi_\delta(r)}^{\eps,\delta})-f(x_{\pi_\delta(s)})+g(X_r^{\eps,\delta})\kappa(X_{\pi_\delta(r)}^{\eps,\delta})+f(x_{\pi_\delta(s)})$, it is easily seen that
\begin{multline*}
\int_0^t \gDk(x_s)\frac{X_s^{\varepsilon, \delta}-X_{\pi_\delta(s)}^{\varepsilon, \delta}}{\eps}\thinspace ds  = \int_0^t \gDk(x_s) f(x_{\pi_\delta(s)})\frac{s-\pi_\delta(s)}{\eps}\thinspace ds + \int_0^t \gDk(x_s)\int_{\pi_\delta(s)}^{s}\frac{g(X_r^{\eps,\delta})\kappa(X_{\pi_\delta(r)}^{\eps,\delta})}{\eps}\thinspace dr \thinspace ds\\
 + \int_0^t \gDk(x_s)\int_{\pi_\delta(s)}^{s}\frac{f(X_r^{\eps,\delta})-f(X_{\pi_\delta(r)}^{\eps,\delta})}{\eps} \thinspace dr \thinspace ds
 + \int_0^t \gDk(x_s) \frac{f(X_{\pi_\delta(s)}^{\varepsilon, \delta})-f(x_{\pi_\delta(s)})}{\eps}(s-\pi_\delta(s))\thinspace ds\\
 + \int_0^t \gDk(x_s)\int_{\pi_\delta(s)}^{s}\sigma(X_u^{\eps,\delta})dW_u \thinspace ds.
\end{multline*}
We now recognize the right hand side as being $\sum_{i=1}^{4} {\sfM}_i^{\eps,\delta}(t)$, where ${\sfM}_i^{\eps,\delta}(t)$, $1 \le i \le 4$, are as in the statement of Lemma \ref{L:M_estimates}.
\end{proof}


\begin{proof}[Proof of Lemma \ref{L:l_estimates}]
Recalling that $\gDk(x)\triangleq g(x)D\kappa(x)$ for $x \in \BR^n$, together with the expression for ${\sf M}_1^{\eps,\delta}(t)$ from Lemma \ref{L:M_estimates}, a little algebra yields
\begin{equation*}
\begin{aligned}
\sfM_1^{\eps,\delta}(t) &= \int_0^t g(x_s)(D\kappa(x_s)-D\kappa(x_{\pi_\delta(s)}))f(x_{\pi_\delta(s)})\frac{s-\pi_\delta(s)}{\eps}\thinspace ds \\
& \qquad \qquad \qquad \qquad \qquad \qquad \qquad \qquad \qquad \qquad \quad  + \int_0^t g(x_s) D\kappa(x_{\pi_\delta(s)})f(x_{\pi_\delta(s)})\frac{s-\pi_\delta(s)}{\eps}\thinspace ds\\
&= \int_0^t g(x_s)(D\kappa(x_s)-D\kappa(x_{\pi_\delta(s)}))f(x_{\pi_\delta(s)})\frac{s-\pi_\delta(s)}{\eps}\thinspace ds \\
&+ \int_0^t (g(x_s)-g(x_{\pi_\delta(s)}))D\kappa(x_{\pi_\delta(s)})
f(x_{\pi_\delta(s)})
\frac{s-\pi_\delta(s)}{\eps}\thinspace ds + \int_0^t \gDk(x_{\pi_\delta(s)})f(x_{\pi_\delta(s)})\frac{s-\pi_\delta(s)}{\eps}\thinspace ds.
\end{aligned}
\end{equation*}
The right hand side is easily recognized to be $\sum_{i=1}^3 l_{i}^{\eps,\delta}(t)$, with $l_{i}^{\eps,\delta}(t)$, $1 \le i \le 3$, defined as in the statement of Lemma \ref{L:l_estimates}.  
\end{proof}


Before proceeding with the proof of Lemma \ref{L:ell_1Approx}, we will find it convenient to prove an auxiliary result, namely, Lemma \ref{L:Remainder-Term-Simplification} below, which enables us to estimate the term 
$\sup_{0 \le t \le T}|R_1^{\eps,\delta}(t) + R_2^{\eps,\delta}(t)|$, where $R_1^{\eps,\delta}(t)$ and $ R_2^{\eps,\delta}(t)$ are defined by
\begin{equation}\label{E: Riemann-Approx}
\begin{aligned}
{R}^{\eps,\delta}_t & \triangleq R_1^{\eps,\delta}(t) + R_2^{\eps,\delta}(t)
\triangleq \frac{1}{2}\frac{\delta}{\eps} \left| \sum_{i=0}^{\lfloor \frac{t}{\delta}\rfloor-1}\delta \thinspace \gDk(x_{i\delta})f(x_{i\delta}) -\int_0^t \gDk(x_s)f(x_s)\thinspace ds\right|\\
& \qquad \qquad \qquad \qquad \qquad \qquad \qquad \qquad \quad + \frac{1}{2}\frac{\delta}{\eps} \left|\sum_{i=0}^{\lfloor \frac{t}{\delta}\rfloor-1} \delta  \thinspace \gDk(x_{i\delta})g(x_{i\delta})\kappa(x_{i\delta}) - \int_0^t \gDk(x_s)g(x_s)\kappa(x_s)\thinspace ds \right|.
\end{aligned}
\end{equation}
Essentially, Lemma \ref{L:Remainder-Term-Simplification} assures us that $R_t^{\eps,\delta}=\mathscr{O}(\delta)$  whenever $0<\eps<\eps_0$. The terms $R_1^{\eps,\delta}(t)$ and $ R_2^{\eps,\delta}(t)$ appear in the proofs of Lemmas \ref{L:ell_1Approx} and \ref{L:ell2Approximation}, respectively.



\begin{lemma}\label{L:Remainder-Term-Simplification}
Let $R_t^{\eps,\delta}$ be defined in \eqref{E: Riemann-Approx}. Then, for any fixed $T>0,$ there exists a positive constant $C_{\ref{L:Remainder-Term-Simplification}}$ such that whenever $0<\eps<\eps_0$,
\begin{equation*}
\sup_{0 \le t \le T}|R_t^{\eps,\delta}|=\sup_{0 \le t \le T}|R_1^{\eps,\delta}(t) + R_2^{\eps,\delta}(t)| \le \delta (\cc+1) C_{\ref{L:Remainder-Term-Simplification}}e^{C_{\ref{L:Remainder-Term-Simplification}}T}.
\end{equation*}
\end{lemma}
\begin{proof}
Let ${\sf h}:\BR^n \rightarrow \BR^n$ be defined by ${\sf h}(x)\triangleq \gDk(x)f(x)=\sum_{i=1}^{n}\gDk_i(x)f_i(x)$ where the matrix $\gDk$ has columns $\gDk_1,\dots,\gDk_n$ with $\gDk_i:\BR^n \rightarrow \BR^n$, and $f_i: \BR^n \rightarrow \BR$ for each $i=1,...,n.$ First, for $R_1^{\eps,\delta}(t)$, we observe that
\begin{equation}\label{E:Rem-Eq}
\begin{aligned}
\left| \sum_{i=0}^{\lfloor \frac{t}{\delta}\rfloor-1}\delta \thinspace {\sf h}(x_{i\delta}) -\int_0^t {\sf h}(x_s)\thinspace ds\right| &\le 
\left|\sum_{i=0}^{\lfloor \frac{t}{\delta}\rfloor-1}\int_{i\delta}^{(i+1)\delta}\left\{{\sf h}(x_s)-{\sf h}(x_{i\delta})\right\}ds\right|+ \left| \int_{\delta\lfloor \frac{t}{\delta}\rfloor}^{t}{\sf h}(x_s)ds\right|\\
& \le \sum_{i=0}^{\lfloor \frac{t}{\delta}\rfloor-1}\int_{i\delta}^{(i+1)\delta}\left|{\sf h}(x_s)-{\sf h}(x_{i\delta})\right|ds+  \int_{\delta\lfloor \frac{t}{\delta}\rfloor}^{t}\left|{\sf h}(x_s)\right|ds.
\end{aligned}
\end{equation}
For $s \in [i\delta,(i+1)\delta)$, $i \in \{0,\dots,\lfloor t/\delta\rfloor -1 \}$, we can now write
\begin{multline*}
{\sf h}(x_s)-{\sf h}(x_{i\delta})
=\sum_{j=1}^{n}\left\{\gDk_j(x_s)f_j(x_s)-\gDk_j(x_s)f_j(x_{i\delta})+\gDk_j(x_s)f_j(x_{i\delta})-\gDk_j(x_{i\delta})f_j(x_{i\delta})\right\}\\=\gDk(x_s)\big[f(x_s)-f(x_{i\delta}) \big] + \left[\gDk(x_s)-\gDk(x_{i\delta})\right]f(x_{i\delta}).
\end{multline*}
Noting that $\left[\gDk(x_s)-\gDk(x_{i\delta})\right]=\left[g(x_s)-g(x_{i\delta})\right]D\kappa(x_s)+g(x_{i\delta})\left[D\kappa(x_s)-D\kappa(x_{i\delta})\right]$, the above equation yields
\begin{multline*}
|{\sf h}(x_s)-{\sf h}(x_{i\delta})| \le \big|\gDk(x_s)\left[f(x_s)-f(x_{i\delta}) \right]\big| \\ + \big\{\big|\left\{g(x_s)-g(x_{i\delta})\right\}D\kappa(x_s)\big| + \big|g(x_{i\delta})\left\{D\kappa(x_s)-D\kappa(x_{i\delta})\right\}\big|\big\}|f(x_{i\delta})|.
\end{multline*}
Using Lipschitz continuity of $f,g$ (Assumption \ref{A:Smooth-LinearGrow}), the boundedness of $g$ and the partial derivatives of $\kappa$ (Assumption \ref{A:Boundedness}), and the second line of equation \eqref{E:Kappa-Derivative-Est} in Lemma \ref{L:Kappa-Derivative-Est}, we get
$|{\sf h}(x_s)-{\sf h}(x_{i\delta})| \le C|x_s-x_{i\delta}|+ C\left(\delta + |x_s-x_{i\delta}|\right)|f(x_{i\delta})|.$ 
Now, using the linear growth property of $f$  (Assumption \ref{A:Smooth-LinearGrow}), and Lemmas \ref{L:Nonlinear-Sys-Sampling-Est} and \ref{L:Sampling-Difference}, we have $|{\sf h}(x_s)-{\sf h}(x_{i\delta})| \le \delta Ce^{CT}$. Combining this estimate with equation \eqref{E:Rem-Eq}, we get $| \sum_{i=0}^{\lfloor {t}/{\delta}\rfloor-1}\delta \thinspace {\sf h}(x_{i\delta}) -\int_0^t {\sf h}(x_s)\thinspace ds| \le \delta Ce^{CT}$. Recalling \eqref{E:eps0}, and the definition of $R_1^{\eps,\delta}(t)$, we obtain $\sup_{0 \le t \le T}|R_1^{\eps,\delta}(t)|\le \delta (\cc+1) Ce^{CT}.$ Similar calculations enable us to get the desired bound for $R_2^{\eps,\delta}(t)$, thereby completing the proof.
\end{proof}


\begin{proof}[Proof of Lemma \ref{L:ell_1Approx}]
Recalling $l_1^{\eps,\delta}(t)$ from Lemma \ref{L:l_estimates}, we have
\begin{multline*}
l_1^{\eps,\delta}(t)-\ell_1(t)= \frac{1}{\eps}\sum_{i=0}^{\lfloor \frac{t}{\delta}\rfloor-1}\int_{i\delta}^{(i+1)\delta} \gDk(x_{i\delta})f(x_{i\delta})[s-\pi_{\delta}(s)]\thinspace ds - \frac{1}{2}\thinspace \cc \int_0^t \gDk(x_s)f(x_s)\thinspace ds\\
+ \frac{1}{\eps}\int_{\delta\lfloor \frac{t}{\delta}\rfloor}^{t}\gDk(x_{\pi_\delta(t)})f(x_{\pi_\delta(t)})[s-\pi_\delta(s)]\thinspace ds.
\end{multline*}
Noting that $\int_{i\delta}^{(i+1)\delta} \gDk(x_{i\delta})f(x_{i\delta})[s-\pi_{\delta}(s)]\thinspace ds=\gDk(x_{i\delta})f(x_{i\delta})(\delta^2/2)$, we have
\begin{equation*}
\begin{aligned}
|l_1^{\eps,\delta}(t)-\ell_1(t)|& \le \left|\frac{1}{2}\frac{\delta}{\eps}\sum_{i=0}^{\lfloor \frac{t}{\delta}\rfloor-1}\delta \thinspace \gDk(x_{i\delta})f(x_{i\delta})-\frac{1}{2}\thinspace \left(\cc -\frac{\delta}{\eps}+\frac{\delta}{\eps}\right) \int_0^t \gDk(x_s)f(x_s)\thinspace ds\right|+ \delta \thinspace|\gDk(x_{\pi_\delta(t)})f(x_{\pi_\delta(t)})|\frac{\delta}{\eps}\\
& \le \left| \frac{1}{2}\frac{\delta}{\eps}\sum_{i=0}^{\lfloor \frac{t}{\delta}\rfloor-1}\delta \thinspace \gDk(x_{i\delta})f(x_{i\delta}) -\frac{1}{2}\frac{\delta}{\eps}\int_0^t \gDk(x_s)f(x_s)\thinspace ds \right|+ \frac{1}{2}\left|\frac{\delta}{\eps}-\cc \right|\left|\int_0^t \gDk(x_s)f(x_s)\thinspace ds\right|\\
& \qquad \qquad \qquad \qquad \qquad \qquad \qquad \qquad \qquad \qquad \qquad \qquad \qquad+\delta \thinspace \left|\gDk(x_{\pi_\delta(t)})f(x_{\pi_\delta(t)})\right|\frac{\delta}{\eps}.
\end{aligned}
\end{equation*}
Now, using the linear growth condition for $f$ from Assumption \ref{A:Smooth-LinearGrow}, and the boundedness of $g$ and first partial derivatives
of $\kappa$ from Assumption \ref{A:Boundedness}, we get
\begin{equation*}
\begin{aligned}
|l_1^{\eps,\delta}(t)-\ell_1(t)|& \le \frac{1}{2}\frac{\delta}{\eps} \left| \sum_{i=0}^{\lfloor \frac{t}{\delta}\rfloor-1}\delta \thinspace \gDk(x_{i\delta})f(x_{i\delta}) -\int_0^t \gDk(x_s)f(x_s)\thinspace ds\right|+\left( \frac{\delta^2}{\eps}+\frac{\varkappa(\eps)T}{2}\right)\left(1+\sup_{0\le t \le T}|x(t)|\right)C.
\end{aligned}
\end{equation*} 
Taking supremum on both sides of the above equation and using Lemmas \ref{L:Nonlinear-Sys-Sampling-Est} and \ref{L:Remainder-Term-Simplification}, we get the required result.
\end{proof}

\begin{proof}[Proof of Lemma \ref{L:l2Andl3_Estimates}]
Using the boundedness of $g$ and the partial derivatives of $\kappa$ from Assumption \ref{A:Boundedness}, linear growth condition on $f$ from Assumption \ref{A:Smooth-LinearGrow}, and Lemmas \ref{L:Nonlinear-Sys-Sampling-Est}, \ref{L:Sampling-Difference}, we get
\begin{equation*}
\begin{aligned}
|l_2^{\eps,\delta}(t)| &\le \frac{\delta}{\eps}\int _0^t \left|D\kappa(x_s)-D\kappa(x_{\pi_\delta(s)})\right||g(x_s)||f(x_{\pi_\delta(s)})|\thinspace ds 
 \le \frac{\delta}{\eps}Ce^{CT}\int_0^t\left|D\kappa(x_s)-D\kappa(x_{\pi_\delta(s)})\right|ds,\\
|l_3^{\eps,\delta}(t)| & \le \frac{\delta}{\eps} \int_0^t \left|g(x_s)-g(x_{\pi_\delta(s)})\right||D\kappa(x_{\pi_\delta(s)})||
f(x_{\pi_\delta(s)})|ds \le \frac{\delta}{\eps} Ce^{CT}\int_0^t \left|g(x_s)-g(x_{\pi_\delta(s)})\right|ds. 
\end{aligned}
\end{equation*}
Taking the supremum on both sides, we now use the Lipschitz continuity of $g$ (Assumption \ref{A:Smooth-LinearGrow}), equation \eqref{E:eps0}, and the second line of equation \eqref{E:Kappa-Derivative-Est} in Lemma \ref{L:Kappa-Derivative-Est} to get the required estimate.
\end{proof}


\begin{proof}[Proof of Lemma \ref{L:L_1_Estimate}]
Recalling equation \eqref{E:L_1AndL_2Estimates}, we easily see that 
\begin{equation*}
\begin{aligned}
L_1^{\eps,\delta}(t)&= \int_0^t \gDk(x_s)\int_{\pi_\delta(s)}^{s} \frac{g(X_r^{\eps,\delta})-g(x_r)+g(x_r)-g(x_{\pi_\delta(r)})}{\eps}\left[\kappa(X_{\pi_\delta(r)}^{\eps,\delta})-\kappa(x_{\pi_\delta(r)})+\kappa(x_{\pi_\delta(r)})\right] \thinspace dr \thinspace ds\\
& = \sum_{i=1}^4 \sfN_i^{\eps,\delta}(t), \quad \text{where}
\end{aligned}
\end{equation*}
\begin{equation*}
\begin{aligned}
{\sf N}_1^{\eps,\delta}(t) &\triangleq \int_0^t \gDk(x_s)\int_{\pi_\delta(s)}^{s} \frac{g(X_r^{\eps,\delta})-g(x_r)}{\eps}\left(\kappa(X_{\pi_\delta(r)}^{\eps,\delta})-\kappa(x_{\pi_\delta(r)})\right)\thinspace dr \thinspace ds, \\
{\sf N}_2^{\eps,\delta}(t) &\triangleq \int_0^t \gDk(x_s)\int_{\pi_\delta(s)}^{s}\frac{g(X_r^{\eps,\delta})-g(x_r)}{\eps}\kappa(x_{\pi_\delta(r)})\thinspace dr \thinspace ds,\\
{\sf N}_3^{\eps,\delta}(t) &\triangleq \int_0^t \gDk(x_s)\int_{\pi_\delta(s)}^{s}\frac{g(x_r)-g(x_{\pi_\delta(r)})}{\eps}\left(\kappa(X_{\pi_\delta(r)}^{\eps,\delta})-\kappa(x_{\pi_\delta(r)})\right)\thinspace dr \thinspace ds,\\
  {\sf N}_4^{\eps,\delta}(t) &\triangleq \int_0^t \gDk(x_s)\int_{\pi_\delta(s)}^{s}\frac{g(x_r)-g(x_{\pi_\delta(r)})}{\eps}\kappa(x_{\pi_\delta(r)})\thinspace dr \thinspace ds. \\
\end{aligned}
\end{equation*}
Recalling the boundedness of $g$ and the partial derivatives of $\kappa$ from Assumption \ref{A:Boundedness}, and the Lipschitz condition on $\kappa$ from Assumption \ref{A:Smooth-LinearGrow}, we get 
$|{\sf N}_1^{\eps,\delta}(t)|\le \frac{1}{\eps}C\int_0^t \int_{\pi_\delta(s)}^{s}\sup_{0\le r \le s}|X_r^{\eps,\delta}-x_r|\thinspace dr \thinspace ds$. Now, by Theorem \ref{T:LLN-II}, we get
\begin{equation*}
\begin{aligned}
\BE\left[\sup_{0\le t \le T}|{\sf N}_1^{\eps,\delta}(t)|\right] & \le \frac{1}{\eps}C\int_0^T \int_{\pi_\delta(s)}^{s}\BE \left[\sup_{0\le r \le s}|X_r^{\eps,\delta}-x_r|\right]\thinspace dr \thinspace ds
 \le \frac{\delta}{\eps}(\eps\sqrt{n}+\delta)Ce^{CT}. 
\end{aligned}
\end{equation*}
Using the same argument for ${\sf N}_3^{\eps,\delta}(t),$ we get
 $ \BE\left[\sup_{0\le t \le T}|{\sf N}_3^{\eps,\delta}(t)|\right] \le \frac{\delta}{\eps}(\eps\sqrt{n}+\delta)Ce^{CT}.$
For ${\sf N}_2^{\eps,\delta}(t)$, using the Lipschitz continuity of $g$ and the linear growth property of $\kappa$ from Assumption \ref{A:Smooth-LinearGrow}, together with the boundedness of $g$ and the partial derivatives of $\kappa$ from Assumption \ref{A:Boundedness}, we get 
$|{\sf N}_2^{\eps,\delta}(t)|\le \frac{C}{\eps}\int_0^t  \int_{\pi_\delta(s)}^{s}\displaystyle \sup_{0\le r \le s}|X_r^{\eps,\delta}-x_r| dr  ds.$
Recalling Theorem \ref{T:LLN-II} and Lemma \ref{L:Nonlinear-Sys-Sampling-Est}, we have
\begin{equation*}
\begin{aligned}
\BE\left[\sup_{0\le t \le T}|{\sf N}_2^{\eps,\delta}(t)|\right]  \le \frac{1}{\eps}C\int_0^T \int_{\pi_\delta(s)}^{s}\BE\left[\sup_{0\le r \le s}|X_r^{\eps,\delta}-x_r|\right]\thinspace dr \thinspace ds
 \le \frac{\delta}{\eps}(\eps\sqrt{n}+\delta)Ce^{CT}.
\end{aligned}
\end{equation*}
Finally, for ${\sf N}_4^{\eps,\delta}(t),$ we recall Lemmas \ref{L:Nonlinear-Sys-Sampling-Est} and \ref{L:Sampling-Difference}, and use once again the Lipschitz continuity of $g$ and the linear growth property of $\kappa$ (Assumption \ref{A:Smooth-LinearGrow}), together with the boundedness of $g$ and the partial derivatives of $\kappa$ (Assumption \ref{A:Boundedness}) to obtain
$|{\sf N}_4^{\eps,\delta}(t)|\le \frac{1}{\eps}C\left(1+\sup_{0\le t \le T}|x(t)|\right)\int_0^t \int_{\pi_\delta(s)}^{s}|x_r-x_{\pi_\delta(r)}|\thinspace dr \thinspace ds 
 \le \frac{\delta^2}{\eps}Ce^{CT}.$
Putting together the estimates for ${\sf N}_i^{\eps,\delta}(t)$, $1 \le i \le 4$, and using \eqref{E:eps0}, we get the required estimate.
\end{proof}


\begin{proof}[Proof of Lemma \ref{L:G1-Est}]
Recalling the boundedness of $g$ (Assumption \ref{A:Boundedness}) and the Lipschitz continuity of $\kappa$ (Assumption \ref{A:Smooth-LinearGrow}), the result follows from Theorem \ref{T:LLN-II} and equation \eqref{E:eps0}.
\end{proof}


\begin{proof}[Proof of Lemma \ref{L:P_estimates}]
It is easily seen that $G_2^{\eps,\delta}(t) = I_1+I_2,$ where $I_1$ and $I_2$ are defined as follows:
\begin{equation*}
\begin{aligned}
I_1 &\triangleq \frac{1}{\eps}\int_0^t g(x_s)\left[D\kappa(x_s)-D\kappa(x_{\pi_\delta(s)})\right]g(x_{\pi_\delta(s)})\int_{\pi_\delta(s)}^{s}\kappa(x_{\pi_\delta(r)})\thinspace dr \thinspace ds 
 \quad \text{and}\\
I_2 &\triangleq \frac{1}{\eps}\int_0^t g(x_s)D\kappa(x_{\pi_\delta(s)})g(x_{\pi_\delta(s)})\int_{\pi_\delta(s)}^{s}\kappa(x_{\pi_\delta(r)})\thinspace dr \thinspace ds.
\end{aligned}
\end{equation*}
Setting ${\sf P}_{1}^{\eps,\delta}(t)=I_1$, and decomposing $I_2$ in two parts, we get
\begin{multline*}
G_2^{\eps,\delta}(t)={\sf P}_{1}^{\eps,\delta}(t)+\frac{1}{\eps}\int_0^t \gDk(x_{\pi_\delta(s)})g(x_{\pi_\delta(s)})\int_{\pi_\delta(s)}^{s}\kappa(x_{\pi_\delta(r)})\thinspace dr \thinspace ds\\
+ \frac{1}{\eps}\int_0^t\left[g(x_s)-g(x_{\pi_\delta(s)})\right] D\kappa(x_{\pi_\delta(s)})g(x_{\pi_\delta(s)})\int_{\pi_\delta(s)}^{s}\kappa(x_{\pi_\delta(r)})\thinspace dr \thinspace ds.
\end{multline*}
The last two terms on the right hand side are easily recognized to be ${\sf P}_{2}^{\eps,\delta}(t)$ and ${\sf P}_{3}^{\eps,\delta}(t)$, respectively.
\end{proof}


\begin{proof}[Proof of Lemma \ref{L:ell2Approximation}] For any $t\ge 0$, recalling the expressions for ${\sf P}_2^{\eps,\delta}(t)$ and $\ell_2(t)$ from Lemma \ref{L:P_estimates} and equation \eqref{E:ell_1-and-ell_2}, we have 
\begin{multline*}
\sfP_2^{\eps,\delta}(t)-\ell_2(t) = \frac{1}{\eps}\sum_{i=0}^{\lfloor \frac{t}{\delta}\rfloor-1}\int_{i\delta}^{(i+1)\delta} \gDk(x_{i\delta})g(x_{i\delta})\kappa(x_{i\delta})[s-i \delta]\thinspace ds - \frac{1}{2} \cc \int_0^t \gDk(x_s)g(x_s)\kappa(x_s) \thinspace ds\\
 + \frac{1}{\eps}\int_{\delta\lfloor \frac{t}{\delta}\rfloor}^{t}\gDk(x_{\pi_\delta(t)})g(x_{\pi_\delta(t)})\kappa(x_{\pi_\delta(t)}){[s-\pi_\delta(s)]}\thinspace ds.
\end{multline*}
Noting that $\int_{i\delta}^{(i+1)\delta} \gDk(x_{i\delta})g(x_{i\delta})\kappa(x_{i\delta})[s-i \delta]\thinspace ds=\gDk(x_{i\delta})g(x_{i\delta})\kappa(x_{i\delta})(\delta^2/2)$, we now easily get
\begin{multline*}
|\sfP_2^{\eps,\delta}(t)-\ell_2(t)|   \le \thinspace \left|\frac{1}{2}\frac{\delta}{\eps}\sum_{i=0}^{\lfloor \frac{t}{\delta}\rfloor-1} \delta  \thinspace \gDk(x_{i\delta})g(x_{i\delta})\kappa(x_{i\delta})-\frac{1}{2}\left(\cc-\frac{\delta}{\eps}+\frac{\delta}{\eps} \right)  \int_0^t \gDk(x_s)g(x_s)\kappa(x_s) \thinspace ds\right|\\ 
+ \delta \thinspace \left|\gDk(x_{\pi_\delta(t)})g(x_{\pi_\delta(t)})\kappa(x_{\pi_\delta(t)})\right|\frac{\delta}{\eps}\\
 \le \frac{1}{2}\frac{\delta}{\eps} \left|\sum_{i=0}^{\lfloor \frac{t}{\delta}\rfloor-1} \delta  \thinspace \gDk(x_{i\delta})g(x_{i\delta})\kappa(x_{i\delta}) - \int_0^t \gDk(x_s)g(x_s)\kappa(x_s)\thinspace ds \right| + \frac{1}{2}\left|\frac{\delta}{\eps}-\cc\right| \left| \int_0^t \gDk(x_s)g(x_s)\kappa(x_s) \thinspace ds\right| \\
+ \delta \thinspace \left|\gDk(x_{\pi_\delta(t)})g(x_{\pi_\delta(t)})\kappa(x_{\pi_\delta(t)})\right|\frac{\delta}{\eps}.
\end{multline*}
Recalling Lemma \ref{L:Nonlinear-Sys-Sampling-Est}, we now use the linear growth condition on $\kappa$ from Assumption \ref{A:Smooth-LinearGrow}, together with the  boundedness of $g$ and the partial derivatives of $\kappa$ from Assumption \ref{A:Boundedness}, to get
\begin{equation*}
\begin{aligned}
|\sfP_2^{\eps,\delta}(t)-\ell_2(t)| \le \frac{1}{2}\frac{\delta}{\eps} \left|\sum_{i=0}^{\lfloor \frac{t}{\delta}\rfloor-1} \delta  \thinspace \gDk(x_{i\delta})g(x_{i\delta})\kappa(x_{i\delta}) - \int_0^t \gDk(x_s)g(x_s)\kappa(x_s)\thinspace ds \right|+ \left(\frac{\varkappa(\eps)T}{2}+ \frac{\delta^2}{\eps}\right)Ce^{CT}.
\end{aligned}
\end{equation*}
Taking supremum on both sides, the required estimate is obtained from equation \eqref{E:eps0} and Lemma \ref{L:Remainder-Term-Simplification}.
 \end{proof}


\begin{proof}[Proof of Lemma \ref{L:P1AndP3Estimates}]
Recalling the expressions for ${\sf P}_{1}^{\eps,\delta}(t)$, ${\sf P}_{3}^{\eps,\delta}(t)$ from Lemma \ref{L:P_estimates}, we use the linear growth condition on $\kappa$ from Assumption \ref{A:Smooth-LinearGrow} and boundedness of $g$ from Assumption \ref{A:Boundedness},
to get
\begin{equation*}
|{\sf P}_{1}^{\eps,\delta}(t)|+|{\sf P}_{3}^{\eps,\delta}(t)|\le \frac{\delta}{\eps}C\left(1+\sup_{0\le t \le T}|x(t)|\right)\left[\int_0^t \left\{|D\kappa(x_s)-D\kappa(x_{\pi_\delta(s)})|+|g(x_s)-g(x_{\pi_\delta(s)})| \right\} \thinspace ds \right]. 
\end{equation*}
Lemmas \ref{L:Nonlinear-Sys-Sampling-Est}, \ref{L:Sampling-Difference}, together with the Lipschitz continuity of $g$ (Assumption \ref{A:Smooth-LinearGrow}) and the second line in equation \eqref{E:Kappa-Derivative-Est}, yield the desired result.
\end{proof}




\begin{proof}[Proof of Lemma \ref{L:M3_Estimate}]
The expression for ${\sf M}_3^{\eps,\delta}(t)$ from equation \eqref{E:M1234} in Lemma \ref{L:M_estimates} can be written $\sfM_3^{\eps,\delta}(t)=Q_1^{\eps,\delta}(t)+Q_2^{\eps,\delta}(t)$, where
\begin{equation*}
\begin{aligned}
Q_1^{\eps,\delta}(t) &\triangleq \int_0^t \gDk(x_s)\int_{\pi_\delta(s)}^{s}\frac{f(X_r^{\eps,\delta})-f(X_{\pi_\delta(r)}^{\eps,\delta})}{\eps}\thinspace dr \thinspace ds,\\
 Q_2^{\eps,\delta}(t) &\triangleq \int_0^t \gDk(x_s) \frac{f(X_{\pi_\delta(s)}^{\varepsilon, \delta})-f(x_{\pi_\delta(s)})}{\eps}(s-\pi_\delta(s)) \thinspace ds.
\end{aligned}
\end{equation*}
Using the Lipschitz continuity of $f$ from Assumption \ref{A:Smooth-LinearGrow}, we get
\begin{equation*}
\begin{aligned}
|Q_1^{\eps,\delta}(t)|&  
\le \frac{C}{\eps}\int_0^t|\gDk(x_s)|\int_{\pi_\delta(s)}^{s}\sup_{0\le r \le s}|X_r^{\eps,\delta}-X_{\pi_\delta(r)}^{\eps,\delta}|\thinspace dr \thinspace ds\\
& \le \frac{C}{\eps}\int_0^t|\gDk(x_s)|\int_{\pi_\delta(s)}^{s}\sup_{0\le r \le s}|X_r^{\eps,\delta}-x_r+x_r-x_{\pi_\delta(r)}+x_{\pi_\delta(r)}-X_{\pi_\delta(r)}^{\eps,\delta}|\thinspace dr \thinspace ds\\
& \le \frac{C}{\eps}\int_0^t\int_{\pi_\delta(s)}^{s}\sup_{0\le r \le s}|X_r^{\eps,\delta}-x_r|\thinspace dr \thinspace ds+ \frac{C}{\eps}\int_0^t\int_{\pi_\delta(s)}^{s}\sup_{0\le r \le s}|x_r-x_{\pi_\delta(r)}|\thinspace dr \thinspace ds,
\end{aligned}
\end{equation*}
where the last line uses boundedness of $g$ and partial derivatives of $\kappa$ from Assumption \ref{A:Boundedness} (recall that $\gDk(x_s)\triangleq g(x_s)D\kappa(x_s)$).
Hence,
\begin{multline*}
\BE\left[\sup_{0\le t \le T}|Q_1^{\eps,\delta}(t)|\right] \le \frac{C}{\eps}\left\{\int_0^T \int_{\pi_\delta(s)}^{s} \BE\left[ \sup_{0\le r \le s}|X_r^{\eps,\delta}-x_r|\right] dr  ds
+ \int_0^T \int_{\pi_\delta(s)}^{s}\sup_{0\le r \le s}|x_r-x_{\pi_\delta(r)}| dr ds \right\}.
\end{multline*}
Theorem \ref{T:LLN-II} and Lemma \ref{L:Sampling-Difference} now yield 
$\BE\left[\sup_{0\le t \le T}|Q_1^{\eps,\delta}(t)|\right] \le \frac{\delta}{\eps}\{(\eps \sqrt{n}+\delta)T+\delta\}Ce^{CT}$.
Using the Lipschitz continuity of $f$ (Assumption \ref{A:Smooth-LinearGrow}), and the boundedness of $g$ and the partial derivatives of $\kappa$ (Assumption \ref{A:Boundedness}), it is easy to show that
$\BE\left[\sup_{0\le t \le T}|Q_2^{\eps,\delta}(t)|\right]\le \frac{\delta}{\eps}(\eps \sqrt{n}+\delta)Ce^{CT}$. Now, putting the $Q_1^{\eps,\delta}(t)$ and $Q_2^{\eps,\delta}(t)$ estimates together and employing \eqref{E:eps0}, the claim easily follows.
\end{proof}


\begin{proof}[Proof of Lemma \ref{L:M4-Noise-Term-Small}]
Recalling equation \eqref{E:M1234}, we have ${\sf M}_4^{\eps,\delta}(t)=\int_0^t \gDk(x_s) \int_{\pi_\delta(s)}^{s}\sigma(X_u^{\eps,\delta})dW_u \thinspace ds$. 
Using H$\ddot{\text{o}}$lder's inequality for ${\sf M}_4^{\eps,\delta}(t),$ and boundedness of $g$ and the partial derivatives of $\kappa$ (Assumption \ref{A:Boundedness}), we get
\begin{equation*}
|{\sf M}_4^{\eps,\delta}(t)|\le \left(\int_0^t |\gDk(x_s)|^2 ds\right)^{\frac{1}{2}}\left(\int_0^t \left|\int_{\pi_{\delta}(s)}^{s}\sigma(X_u^{\eps,\delta})dW_u\right|^2ds\right)^{\frac{1}{2}} \le C\left(\int_0^T \left|\int_{\pi_{\delta}(s)}^{s}\sigma(X_u^{\eps,\delta})dW_u\right|^2ds\right)^{\frac{1}{2}}.
\end{equation*}
Taking the supremum over $t \in [0,T]$, followed by expectation, we get 
\begin{equation}\label{E:M4-CLT-I}
\BE\left[\sup_{0\le t \le T}|{\sf M}_4^{\eps,\delta}(t)|\right]\le
C\BE\sqrt{\int_0^T \left|\int_{\pi_{\delta}(s)}^{s}\sigma(X_u^{\eps,\delta})dW_u\right|^2ds} \le 
C\left(\int_0^T \BE \left|\int_{\pi_{\delta}(s)}^{s}\sigma(X_u^{\eps,\delta})dW_u\right|^2ds\right)^{\frac{1}{2}},
\end{equation}
where we have used Jensen's inequality for concave functions.
Letting $\sigma_i \in \BR^n$, $1 \le i \le n$, represent the columns of the matrix $\sigma$, we use the Ito isometry for stochastic integrals to get
\begin{equation*}
\begin{aligned}
\BE \left|\int_{\pi_{\delta}(s)}^{s}\sigma(X_u^{\eps,\delta})dW_u\right|^2 
&=\BE \left|\sum_{i=1}^{n}\int_{\pi_{\delta}(s)}^{s}\sigma_i(X_u^{\eps,\delta})dW_u^i\right|^2
\le C\BE \left[\sum_{i=1}^{n}\left|\int_{\pi_{\delta}(s)}^{s}\sigma_i(X_u^{\eps,\delta})dW_u^i \right|^2 \right]\\
&\le C\sum_{i=1}^{n} \sum_{j=1}^{n} \BE \left| \int_{\pi_{\delta}(s)}^{s}\sigma_{ji}(X_u^{\eps,\delta})dW_u^i\right|^2
= C\sum_{i=1}^{n} \sum_{j=1}^{n} \BE \int_{\pi_{\delta}(s)}^{s}\sigma_{ji}^2(X_u^{\eps,\delta})du. 
\end{aligned}
\end{equation*}
Now employing the linear growth assumption on $\sigma$ (Assumption \ref{A:Smooth-LinearGrow}), we have
\begin{equation}\label{E:M4-CLT-II}
\begin{aligned}
\BE \left|\int_{\pi_{\delta}(s)}^{s}\sigma(X_u^{\eps,\delta})dW_u\right|^2 \le C \int_{\pi_{\delta}(s)}^{s}(1+\BE|X_u^{\eps,\delta}|^2) du 
& \le C\int_{\pi_{\delta}(s)}^{s} \left( 1+ |x_u|^2 +\BE \left[\sup_{0 \le u \le T} |X_u^{\eps,\delta}-x_u|^2 \right]  \right)du.
\end{aligned}
\end{equation}
Using this last expression in equation \eqref{E:M4-CLT-I}, and then using Theorem \ref{T:LLN-II} and Lemma \ref{L:Nonlinear-Sys-Sampling-Est}, we get the desired result.
\end{proof}




\section{Numerical Example and Simulation}\label{S:NumExa}
Our goal in this section is to illustrate, using numerical simulations, our main result Theorem \ref{T:fluctuations-R-1-2}. The chosen example will be the inverted pendulum \cite{aracil2004inverted,aastrom2005new,aracil2008family} whose dynamics are governed by
\begin{equation*}
\begin{aligned}
\dot{x}_1 &=x_2, \\
\dot{x}_2 &=\sin x_1- (\cos x_1) u.
\end{aligned}
\end{equation*}
Here, $x_1$ and $x_2$ represent the angular position and velocity of the pendulum, respectively, while $u$ denotes the control input. We will use the control law $u=\kappa(x_1,x_2)$ where $\kappa(x_1,x_2) \triangleq 2 \sin x_1 + 1.35 x_2 \cos^2 x_1 -0.22 x_2 \cos x_1$. Our choice here is inspired by the law used in \cite{aracil2004inverted,aastrom2005new,aracil2008family} for stabilization of the inverted pendulum, modified to be consistent with Assumption \ref{A:Smooth-LinearGrow} (Lipschitz continuity and linear growth).\footnote{In \cite{aracil2004inverted,aastrom2005new,aracil2008family}, the authors use $\kappa'(x_1,x_2) \triangleq 2 \sin x_1 + 0.45 x_2^3 \cos x_1 + 1.35 x_2 \cos^2 x_1 -0.67 x_2 \cos x_1$. Our choice of $\kappa$ can be obtained from $\kappa^\prime$ by replacing $x_2^3$ by $x_2$ in the $0.45 x_2^3 \cos x_1$ term.}
%
Setting $x=(x_1,x_2)\in \BR^2$, and comparing with \eqref{E:nonlin-sys}, we have 
\begin{equation*}
f(x)=(x_2,\sin x_1)^\top, \quad g(x)=(0,-\cos x_1)^\top, \quad \kappa(x)= 2 \sin x_1 + 1.35 x_2 \cos^2 x_1 -0.22 x_2 \cos x_1.
\end{equation*}
For simplicity, we take $\sigma$ to be the $2\times 2$ identity matrix. Working in Regime 2, the equations for $X^{\eps,\delta}_t, $ $x(t)$, and $Z_t$ are now given by \eqref{E:sie}, \eqref{E:nonlin-sys} and \eqref{E:lim-fluct-R-1-2}, with $f,g,\kappa,\sigma$ as above. 

We now use the Euler-Maruyama method \cite{Higham2001,KloedenPlaten} to generate 1000 sample paths of $X^{\eps,\delta}_t$ and $Z_t$. Since our result on approximating $X^{\eps,\delta}_t$ by $S_t^{\eps}\triangleq x(t)+\eps Z(t)$ is valid in a path-wise sense, we use the same Brownian increments for $X^{\eps,\delta}_t$ and $Z_t$. Fixing the parameters $\Delta t=0.0977, T=25, \delta= 2^{-4}$, and with initial conditions $X_1^{\eps,\delta}(0)=x_1(0)=1, X_2^{\eps,\delta}(0)=x_2(0)=0$ and $Z_1(0)=Z_2(0)=0,$ the comparison of sample paths of the processes $X_t^{\eps,\delta}$ and $S_t^{\eps}\triangleq x(t)+\eps Z(t)$ for the values $\eps= 2^{-5}$ and $\eps= 2^{-3}$ is demonstrated in Figure \ref{fig:test}.

\begin{figure}
\centering
\begin{subfigure}{.5\textwidth}
  \centering
  \includegraphics[width=.97\linewidth]{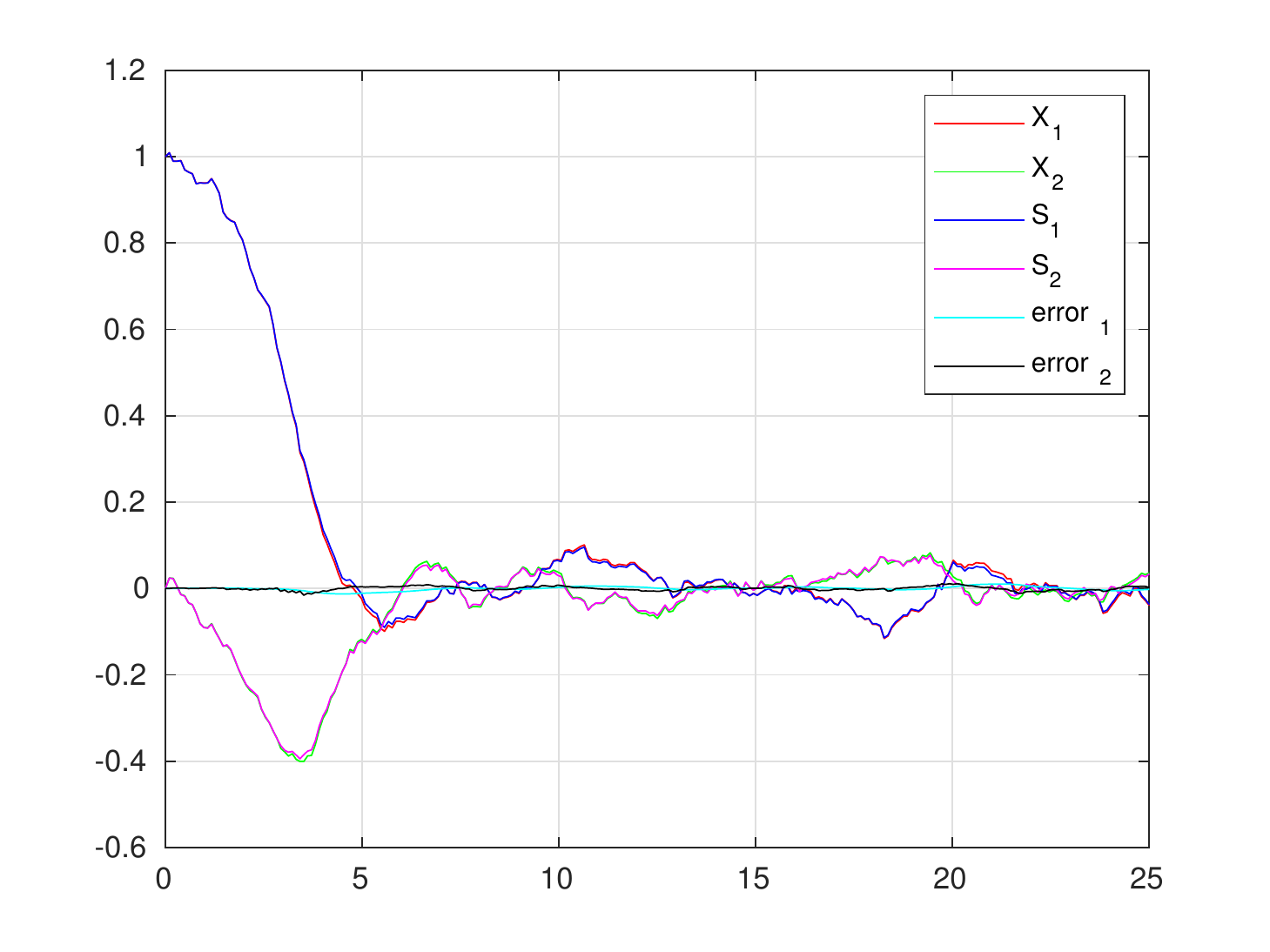}
  \caption{$\eps=2^{-5}$}
  \label{fig:sub1}
\end{subfigure}%
\begin{subfigure}{.5\textwidth}
  \centering
  \includegraphics[width=.97\linewidth]{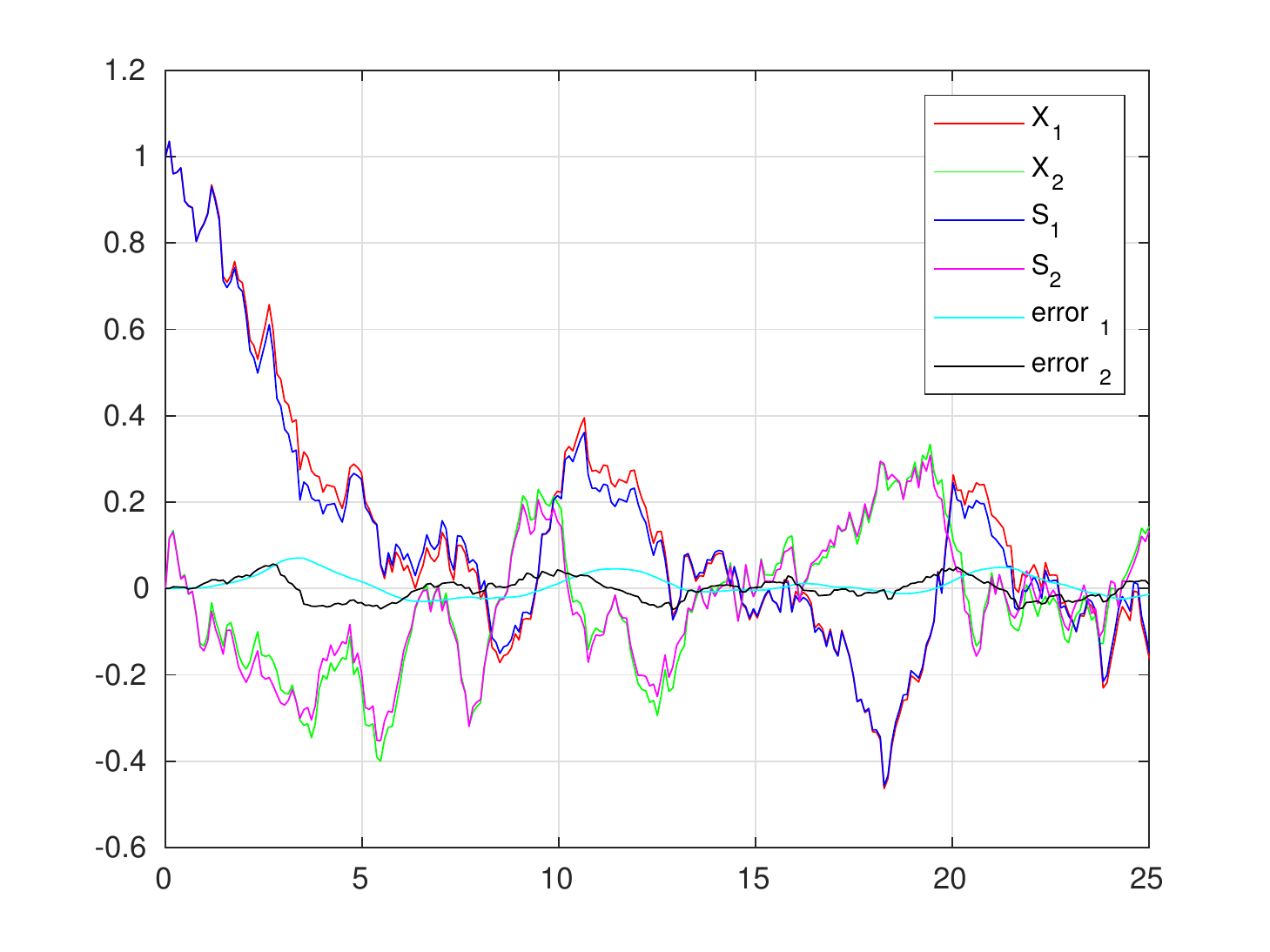}
  \caption{$\eps=2^{-3}$}
  \label{fig:sub2}
\end{subfigure}
\caption{Sample paths of the components $X_i\triangleq X_i^{\eps,\delta}(t),S_i\triangleq S_i^{\eps}(t)\triangleq x_i(t)+\eps Z_i(t)$ and $\text{error}_i \triangleq X_i-S_i$ for $i=1,2$ and $\eps=2^{-5}$ and $\eps=2^{-3}.$}
\label{fig:test}
\end{figure}
\begin{figure}[h!]
\begin{center}
\includegraphics[height=6cm,width=9.6cm]{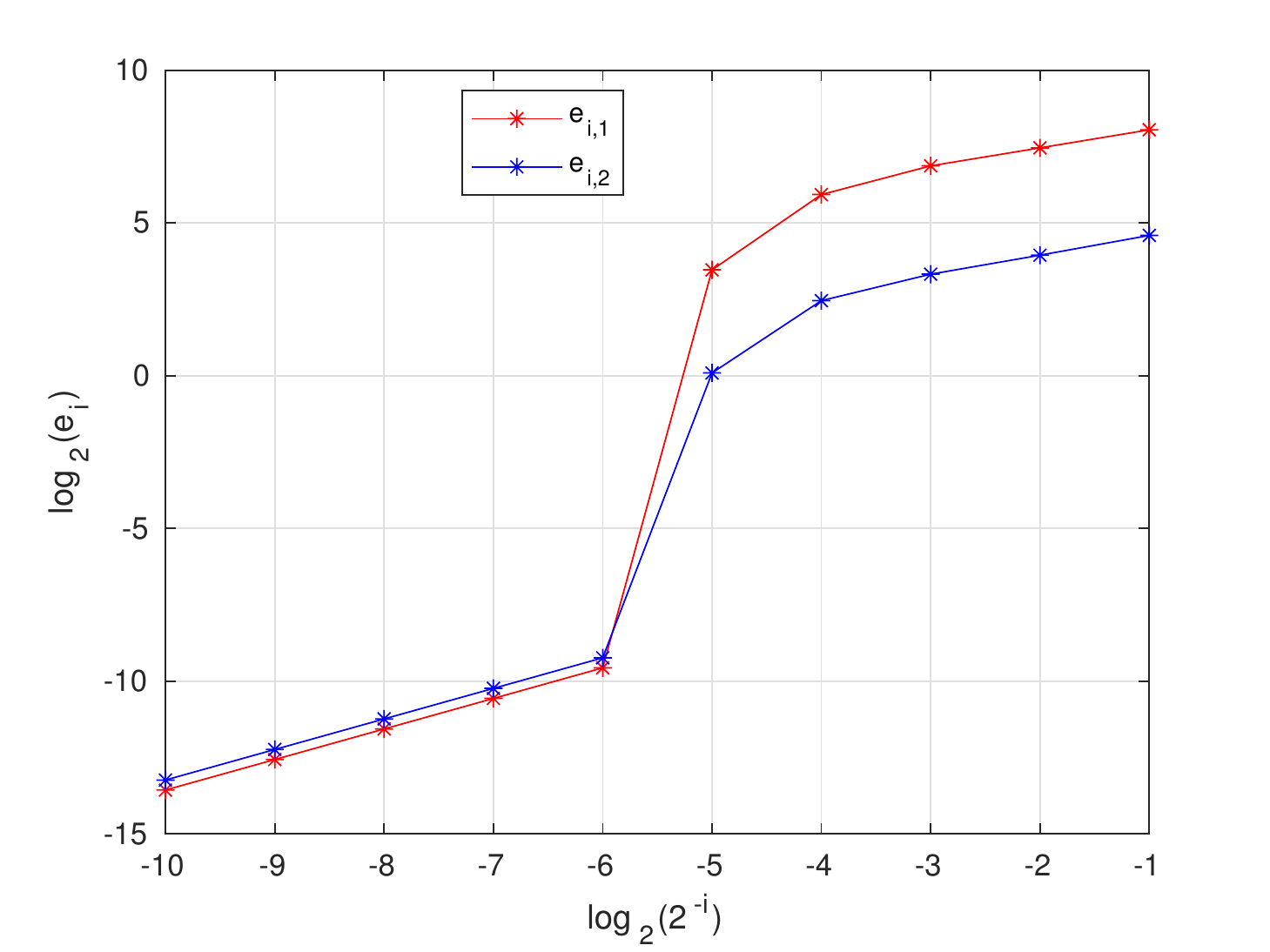}
\caption{The components $(e_{i,1},e_{i,2})$ which are the mean of $|X_i^{\eps,\delta}(T)-S_i^{\eps}(T)|$ over $1000$ sample paths are plotted on a $\log_2$-$\log_2$ scale. The values of $e_{i,j},j=1,2$ decrease with increasing $i$, where the values of $\eps=2^{-i},1\le i \le 10.$}\label{F:Error}
\end{center}
\end{figure}

Next, to explore the dependence of the quantity $\BE\left[\sup_{0\le t \le T}|X_t^{\eps,\delta}-x(t)-\eps Z_t|\right]$ on $\eps$, we fix $\delta$ as above, and 
consider varying values of $\eps=2^{-i},1\le i \le 10$. Once again, we generate 1000 sample paths of the process $X_t^{\eps,\delta}-S_t^{\eps}$. Let the components of vector $e_i \triangleq (e_{i,1},e_{i,2})$ represent the average of $X_j^{\eps,\delta}(T)-S_j^{\eps}(T), j=1,2$. Figure \ref{F:Error} plots the components of $e_i$ against $\eps$ on $\log_2$-$\log_2$ scale, and one can easily see that the $e_i$'s decrease as the values of $\eps$ decreases.

\section{Conclusions}\label{S:conclusion}
In the present paper, we have analyzed the dynamics of a nonlinear differential equation with fast periodic sampling under the influence of small white-noise perturbations. 
For the resulting stochastic process indexed by two small parameters, we have computed the zeroth and first order terms for a perturbation expansion in terms of the small parameters, together with rigorous estimates on the ensuing error. 
The zeroth order term, which describes the mean dynamics, is given by a limiting {\sc ode}, while the first order term, which captures fluctuations about the mean, is given by a linear non-homogeneous {\sc sde}.
This {\sc sde} is found to vary depending on the relative rates at which the two small parameters approach zero. 
Our key finding is the identification of an effective drift term in the limiting {\sc sde} for the fluctuations, which captures the combined effect of small noise and fast sampling. Finally, the theoretical results are illustrated numerically in the context of a control problem. 
The problems studied here suggest several avenues for future exploration.
One possibility, with an eye towards applications in Networked Control Systems \cite{ModlingandAnalysisNCS14}, is to consider the case when state samples are taken at \textit{random} times \cite{TanwaniChatterjeeLiberzon,SCS-IET2018,liu2021event}, e.g., at the event times of a renewal process.
Some preliminary calculations focusing on the embedded discrete-time process (obtained by recording the state at just the sampling instants) were carried out in \cite{dhama2021asymptotic} for the case of linear systems. 
The calculations there used, in an essential way, explicit solution representations for the state process, and the possibility of linking the analysis to some limit theorems for products of random matrices.
It would be interesting to see if these results can be lifted to the full continuous-time state process for both linear and nonlinear systems.

\bibliographystyle{alpha}
\bibliography{Nonlinear}

\end{document}